\documentclass{article}

\pdfoutput=1

\usepackage{dirtytalk}
\usepackage[utf8]{inputenc}
\usepackage{amssymb}
\usepackage[shortlabels]{enumitem}
\usepackage{amsmath}
\usepackage{amsthm}
\usepackage[T1]{fontenc}
\usepackage{tikz}
\usepackage{tikz-cd}
\usepackage{mathtools}
\usepackage{verbatim}
\usepackage{comment}
\usepackage{graphicx}
\usepackage{color}
\usepackage{float}
\usepackage[margin=1.15in]{geometry}
\usepackage[hidelinks]{hyperref}
\usepackage{relsize}
\usepackage{authblk}
\usepackage{subfigure}
\usepackage[linesnumbered,lined,ruled]{algorithm2e}
\usepackage{algpseudocode}

\usepackage[backend=biber, style=alphabetic]{biblatex}
\renewbibmacro{in:}{%
  \ifentrytype{article}{}{\printtext{\bibstring{in}\intitlepunct}}}

\theoremstyle{definition}
\newtheorem{theorem}{Theorem}[section]

\newtheorem{lemma}[theorem]{Lemma}
\newtheorem{proposition}[theorem]{Proposition}
\newtheorem{definition}[theorem]{Definition}
\newtheorem{example}[theorem]{Example}
\newtheorem{corollary}[theorem]{Corollary}
\newtheorem{conjecture}[theorem]{Conjecture}
\newtheorem{remark}[theorem]{Remark}

\newtheorem{manualtheoreminner}{Theorem}
\newenvironment{manualtheorem}[1]{%
  \renewcommand\themanualtheoreminner{#1}%
  \manualtheoreminner
}{\endmanualtheoreminner}

\numberwithin{equation}{theorem}

\DeclareMathOperator{\Des}{\text{Des}}
\DeclareMathOperator{\des}{\text{des}}
\DeclareMathOperator{\Supp}{\text{Supp}}
\DeclareMathOperator{\inv}{inv}
\DeclareMathOperator{\Cl}{Cl}
\DeclareMathOperator{\St}{St}
\DeclareMathOperator{\Lk}{Lk}
\DeclareMathOperator{\Conv}{Conv}

\DeclarePairedDelimiter\abs{\lvert}{\rvert}

\newcommand{\R}{\mathbb{R}}
\newcommand{\Z}{\mathbb{Z}}

\newcommand{\mb}[1]{\mathbb{#1}}
\newcommand{\mf}[1]{\mathfrak{#1}}
\newcommand{\mc}[1]{\mathcal{#1}}
\renewcommand{\cal}[1]{\mathcal{#1}}

\newcommand{\sq}{{\mathsmaller{\square}}}
\renewcommand{\top}{\text{top}}

\newcommand{\flip}{\text{flip}}

\newcommand{\tld}[1]{\widetilde{#1}}

\newcommand{\tlN}[1]{\tld{\mc{N}}(}
\newcommand{\calB}{\mathcal{B}}
\newcommand{\calF}{\mathcal{F}}
\newcommand{\calG}{\mathcal{G}}
\newcommand{\calH}{\mathcal{H}}
\newcommand{\calI}{\mathcal{I}}

\newcommand{\calN}{\mathcal{N}}
\newcommand{\calP}{\mathcal{P}}
\newcommand{\calS}{\mathcal{S}}

\begin{document}

\title{Extended Nestohedra and their Face Numbers}

\author{Quang Dao\thanks{Columbia University, \url{qvd2000@columbia.edu}} \hspace{1.0cm} Christina Meng\thanks{Massachusetts Institute of Technology, \url{mengc@mit.edu}} \hspace{1.0cm} Julian Wellman\thanks{Massachusetts Institute of Technology, \url{wellman@mit.edu}}\and \hspace{-0.7cm} Zixuan Xu\thanks{Massachusetts Institute of Technology, \url{zixuanxu@mit.edu}} \hspace{0.9cm} Calvin Yost-Wolff\thanks{Massachusetts Institute of Technology,  \url{calvinyw@mit.edu}} \hspace{1.2cm} Teresa Yu\thanks{Williams College, \url{twy1@williams.edu}} }
\date{}

\maketitle
\vspace{-0.5cm}

\begin{abstract}
    Nestohedra are a family of convex polytopes that includes permutohedra, associahedra, and graph associahedra. In this paper, we study an extension of such polytopes, called extended nestohedra. We show that these objects are indeed the boundaries of simple polytopes, answering a question of Lam and Pylyavskyy. We also study the duals of (extended) nestohedra, giving a complete characterization of isomorphisms (as simplicial complexes) between the duals of extended nestohedra and a partial characterization of isomorphisms between the duals of nestohedra and extended nestohedra. In addition, we give formulas for their $f$-, $h$-, and $\gamma$-vectors. This includes showing that the $f$-vectors of the extended nestohedron corresponding to a forest $F$ and the nestohedron corresponding to the line graph of $F$ are the same, as well as showing that all flag extended nestohedra have nonnegative $\gamma$-vectors, thus proving Gal's conjecture for a large class of flag simple polytopes. We also relate the $f$- and $h$-vectors of the nestohedra and extended nestohedra, as well as give explicit formulas for the $h$- and $\gamma$-vectors in terms of descent statistics for a certain class of flag extended nestohedra. Finally, we define a partial ordering on partial permutations that is a join semilattice quotient of the weak Bruhat order on the symmetric group, and such that any linear extension of the partial order provides a shelling of the dual of the stellohedron.
\end{abstract}

\tableofcontents

\section{Introduction}

Stasheff's \textbf{associahedron} has spurred a rich study of polytopes with combinatorial and algebraic connections. It has been generalized to the \textbf{graph associahedron}, introduced by Carr and Devadoss \cite{carrDevadoss2006coxeter}, which in turn has been generalized to the \textbf{nestohedron} in \cite{postnikov2009permutohedra,FS2005matroidminkowski}. The dual simplicial complex of the nestohedron, called the \textbf{nested complex}, was introduced by DeConcini and Procesi, along with the notion of \textbf{building sets} \cite{DCP95}.

The nestohedron and nested set complex have together provided many interesting avenues of study, among them polytopal realizations and face numbers. The nestohedron has been geometrically realized as the boundary of a polytope through a variety of methods, including through shaving faces of a simplex \cite{carrDevadoss2006coxeter} and through Minkowski sums \cite{postnikov2009permutohedra,FS2005matroidminkowski}. The face numbers of the nestohedron are extensively studied in \cite{PRW}, while Gal's conjecture was proven for all flag nestohedra in \cite{volodin}.

More recently, Lam and Pylyavskyy introduced the \textbf{extended nested complex} in their study of linear Laurent-phenomenon algebras \cite{LP2016laurentphenomenon}. In the more general context of Laurent-phenomenon algebras, they conjectured that the extended nested complex $\mc N^{\sq}(\mc B_\Gamma)$ is dual to the boundary of a polytope, where $\Gamma$ is a directed graph \cite[Conjecture 7.6]{LP2016laurentphenomenon}. Independently, Devadoss, Heath, and Vipismakul introduced dual complexes for the extended nested complex when the building sets are based on undirected graphs; they refer to such complexes as \textbf{graph cubeahedra}. They also show that graph cubeahedra can be realized by shaving the faces of an $n$-dimensional cube, partially proving Lam and Pylyavskyy's conjecture.

Extended nested complexes are not as well-understood as their non-extended counterparts. However, in certain cases, extended nested complexes are isomorphic to non-extended nested complexes. Manneville and Pilaud characterize such isomorphisms for graphical building sets, and show that the corresponding graphs must be \textbf{spider} and \textbf{octopus} graphs (see \cite{MP17} for further details).

We now give preliminary definitions before stating the main results of our paper. A \textbf{building set $\mc B$} on $S$ is a collection of nonempty subsets of $S$ such that
\begin{enumerate}
    \item if $I,J\in \mc B$ and $I\cap J\neq\varnothing$, then $I\cup J\in \mc B$, and
    \item $\mc B$ contains all singletons $\{i\}$ for $i\in S$ (see Definition~\ref{defn:buildingset}).
\end{enumerate}
A building set on $S$ is \textbf{connected} if $S$ itself is an element of the building set.
An \textbf{extended nested collection} on a building set $\mc B$ on $S$ is a collection $N=\{I_1,\ldots,I_m,x_{i_1},\ldots,x_{i_r}\}$ of elements $I_j\in\mc B$ and $x_i$ for $i\in S$ satisfying the following three properties.
\begin{enumerate}
    \item For any $i\neq j$, either $I_i \subseteq I_j, I_j \subseteq I_i,$ or $I_i \cap I_j  = \varnothing$.
    \item For any collection $I_{i_1},\ldots, I_{i_k} \in N$ of $k\geq 2$ pairwise disjoint elements of $N$, their union $\bigcup\limits_{\ell = 1}^k I_{i_\ell}$ is not an element of  $\mc B$.
    \item For all $x_{i_{\ell}}, I_j \in N$, the set $I_j$ does not contain $i_\ell$.
\end{enumerate}
A \textbf{(non-extended) nested collection} is an extended nested collection with no $x_i$ elements (see Definitions~\ref{defn:nestedcoll}, \ref{defn:extnestedcoll}). If $\mc B$ is a building set on $S$, the \textbf{nested complex} $\mc N(\mc B)$ is the simplicial complex with vertices $\{I\mid I\in \mc B\}$ and faces given by non-extended nested collections $\{I_1,\ldots,I_r\}$. The \textbf{extended nested complex} $\mc N^{\sq}(\mc B)$ is the simplicial complex with vertices $\{I\mid I\in \mc B\}\cup\{x_i\mid i\in S\}$ and faces given by extended nested collections $\{I_1,\ldots,I_m\}\cup\{x_{i_1},\ldots,x_{i_r}\}$ (see Definitions~\ref{defn:nestedcomplex}, \ref{defn:extnestedcomplex}).

We now state several of our results.

\begin{manualtheorem}{\ref{thm:bndpolytope}}\label{thm:intro_bndpolytope}
If $\mc B$ is a building set, then $\mc N^{\sq}(\mc B)$ can be realized geometrically as the boundary of a simplicial polytope.
\end{manualtheorem}

For a building set $\mc B$, it is known that the nested complex $\mc N(\mc B)$ is isomorphic to the boundary of a simplicial polytope, whose polar dual is a simple polytope $\mc P(\mc B)$ called the \textbf{nestohedron}. Similarly, the simplicial polytope in Theorem~\ref{thm:intro_bndpolytope} is polar dual to a simple polytope $\mc P^{\sq}(\mc B)$ that we call the \textbf{extended nestohedron}.
We also find a way to realize $\mc P^{\sq}(\mc B)$ as a Minkowski sum over elements of our building set (Theorem~\ref{thm:minkowski_polytopality}).

In \cite{MP17}, Manneville and Pilaud show that, if $\mc B$ and $\mc B'$ are graphical building sets, then the only possible non-trivial isomorphisms between $\mc N^{\sq}(\mc B)\simeq\mc N(\mc B')$ occur when the building sets correspond to octopus and spider graphs, and the only possible isomorphism between $\mc N^{\sq}(\mc B)\simeq\mc N^\sq(\mc B')$ occur when the building sets correspond to octopus graphs. After defining \textbf{interval}, \textbf{octopus}, and \textbf{spider building sets} (see Subsection~\ref{subsec:interval_spider_octopus_bs}), we partially generalize the former result of Manneville and Pilaud to non-graphical building sets, and generalize the latter result to the following statement.

\begin{manualtheorem}{\ref{thm:ext_to_ext_isom_summary}}
If $\mc N^{\sq}(\mc B)\simeq\mc N^\sq(\mc B')$, then $\mc B$ and $\mc B'$ are isomorphic building sets, corresponding octopus building sets, or corresponding interval building sets.
\end{manualtheorem}

We also give several examples of non-trivial isomorphisms within nested set complexes and extended nested set complexes.

For a simple $d$-dimensional polytope $P$, the \textbf{$f$-vector} and \textbf{$h$-vector} of $P$ are $(f_0,f_1,\ldots,f_d)$ and $(h_0,h_1,\ldots,h_d)$, where $f_i$ is the number of $i$-dimensional faces of $P$ and $h_i$'s are given by $\sum h_i(t+1)^i=\sum f_it^i$. It is well known that the $h$-vectors of simple polytopes are positive and symmetric. 

\begin{manualtheorem}{\ref{thm:f_poly_extended_to_original} and Corollary \ref{cor:h_polynomial_extended_nest}}
For a building set $\mc B$ on $[n]$, the $f$- and $h$-polynomials of the extended nestohedron $\mc P^{\sq}(\mc B)$ satisfy the following formulas:
\begin{align*}
    f_{\calP^\sq(\mc B)}(t) &= \sum_{S \subseteq [n]} (t+1)^{n-\abs{S}}f_{\calP(\mc B|_S)}(t),\\
    h_{\mc P^{\sq}(\mc B)}(t)&=\sum_{S\subseteq[n]}t^{n-|S|}h_{\mc P(\mc B|_S)}(t),
\end{align*}
where $\mc B|_S$ is the building set restricted to elements of $S$, i.e., $\{I\in\mc B\mid I\subseteq S\}$.
\end{manualtheorem}

In addition, for a graph $G$, the line graph of $G$, denoted $L(G)$ is the graph constructed by associating a vertex with each edge of $G$ and connecting two vertices with an edge if the corresponding edges of $G$ have a vertex in common. We show that for certain graphs $G$, the $f$-vector of the nestohedron $\mc P(\mc B_G)$ is the same as the $f$-vector of extended nestohedron $\mc P^{\sq}(\mc B_{L(G)})$.

\begin{manualtheorem}{\ref{thm:same_h_vector}}
Let $G$ be a forest and $L(G)$ be the line graph of $G$. Then \[f_{\calP(\mc B_G)}(t)=f_{\calP^\sq(\mc B_{L(G)})}(t).\] 
\end{manualtheorem}

One can compactify the $h$-vector using another vector called the \textbf{$\gamma$-vector} $(\gamma_0,\ldots,\gamma_{\lfloor{d/2\rfloor}})$, which is defined by the relation
\[\sum_{i=0}^d h_it^i=\sum_{i=0}^{\lfloor{d/2\rfloor}}\gamma_i t^i(1+t)^{d-2i}.\]

A simplicial complex $\Delta$ is called \textbf{flag} if every collection of pairwise adjacent vertices forms a face of $\Delta$. We say that a simple polytope is a \textbf{flag polytope} if its dual is a flag simplicial complex. Gal conjectured in \cite{gal} that the $\gamma$-vector is nonnegative for any flag simple polytope. 

\begin{manualtheorem}{\ref{thm:gal_ext_flag}}
If $\mc B$ is a building set with $\mc P^{\sq}(\mc B)$ a flag extended nestohedron, then the $\gamma$-vector of $\mc P^{\sq}(\mc B)$ is nonnegative.
\end{manualtheorem}

A connected building set $\mc B$ is \textbf{chordal} if for any element $I=\{i_1<\cdots <i_r\}\in \mc B$, all subsets of the form $\{i_s,i_{s+1},\ldots,i_r\}$ also belong to $\mc B$; see Definition~\ref{defn:chordal}. All extended nestohedra $\mc P^{\sq}(\mc B)$ from chordal building sets are flag simple polytopes by Lemma~\ref{lem:chordal_flag}, so their $\gamma$-vectors are nonnegative. 

For a building set $\mc B$ on $[n]$, define the set of \textbf{$\mc B$-partial permutations}, denoted $\mf P_n(\mc B)$, as the set of partial permutations $w\in \mf S_S$ for some $S=\{s_1,\ldots,s_k\}\subseteq [n]$ such that for any $i\in [k]$, the elements $w(s_i)$ and $\max\{w(s_1),w(s_2),\ldots,w(s_i)\}$ lie in the same connected component of the restricted building set $\mc B|_{\{w(s_1),\ldots,w(s_i)\}}$. Define the map $\varphi_n:\mf P_n(\mc B)\to\mf S_{n+1}$ as follows. For a permutation $w\in \mf P_n(\mc B)$ on entry set $S$, let $\varphi_n(w)$ be the permutation formed by appending $[n+1]\setminus S$ to the end of $w$ in descending order. Let $\mf S_{n+1}^{\sq}(\mc B)\coloneqq \varphi_n(\mf P_n(\mc B))$ denote the set of \textbf{extended $\mc B$-permutations}. See Definition~\ref{defn:b_partial_perm}. This set is in bijection with vertices of the extended nestohedron $\mc P^{\sq}(\mc B)$. 

Let $\des(w)=|\{i\mid w(i)>w(i+1)\}|$ denote the number of descents in a permutation $w$. Let $\widehat{\mf S}_{n+1}$ be the subset of permutations $w$ of size $n+1$ without two consecutive descents and no final descent. 

\begin{manualtheorem}{\ref{thm:ext_h_perm} and Theorem~\ref{thm:gamma_chordal}}
Let $\mc B$ be a connected chordal building set on $[n]$. Then the $h$-vector of the extended nestohedron $\mc P^{\sq}(\mc B)$ is given by
\[\sum_i h_i t^i = \sum_{w\in \mf S_{n+1}^{\sq}(\mc B)}t^{\des(w)},\]
and the $\gamma$-vector of the extended nestohedron $\mc P^{\sq}(\mc B)$ is given by
\[\sum_i \gamma_i t^i=\sum_{w\in \widehat{\mf S}_{n+1}\cap \mf S^{\sq}_{n+1}(\mc B)} t^{\des(w)}.\]
\end{manualtheorem}

Let $K_{1,n}$ denote the star graph consisting of a central vertex connected to $n$ vertices. The simple polytope $\mc P(\mc B_{K_{1,n}})$ is known as the \textbf{stellohedron}. Manneville and Pilaud showed that the extended nested complex $\mc N^{\sq}(\mc B_{K_n})$ is isomorphic to the non-extended nested complex $\mc N(\mc B_{K_{1,n}})$, where $K_n$ denotes the complete graph on $n$ vertices and $K_{1,n}$ denotes the star graph on $n+1$ vertices \cite{MP17}. Thus, the facial structure of the extended nestohedron $\mc P^{\sq}(\mc B_{K_n})$ is isomorphic to that of the stellohedron, and we also refer to $\mc P^{\sq}(\mc B_{K_n})$ as the stellohedron. Let $\mf P_n$ denote the set of all partial permutations $w\in\mf S_S$ for some $S\subseteq [n]$. Such permutations are in bijection with the vertices of the stellohedron. 

We can use the weak Bruhat order on $\mf S_{n+1}$ and the map $\varphi_n$ defined above to induce a partial order on $\mf P_n$ as follows. Let $\pi,\sigma\in\mf P_n$ be two partial permutations. Then $\pi\leq \sigma$ in the \textbf{partial weak Bruhat order} if and only if $\varphi_n(\pi)\leq\varphi_n(\sigma)$ in the weak Bruhat order on $\mf S_{n+1}$ (see Definition~\ref{def:partialweak}). 


We say that a pure $n$-dimensional simplicial complex $\Delta$ with $r$ facets is \textbf{shellable} if its facets can be arranged into a linear ordering $F_1,\ldots,F_r$ such that $\left(\bigcup_{i=1}^{k-1} F_k\right)\cap F_k$ is pure of dimension $n-1$ for all $k=2,\ldots,r$. Such an ordering is called a \textbf{shelling} of $\Delta$.

\begin{manualtheorem}{\ref{thm:stellohedronshelling}}
Let $\mc B=\mc B_{K_n}$, and $\mc N^{\sq}(\mc B)$ be the extended nested set complex whose facets $F_\pi$ are labeled by the partial permutations $\pi\in\mf P_n$. If a total ordering $\pi_1\leq\cdots\leq\pi_m$ is a linear extension of the partial weak Bruhat order on $\mf P_n$, then $F_{\pi_1},\ldots,F_{\pi_m}$ is a shelling for $\mc N^{\sq}(\mc B)$.
\end{manualtheorem}

This paper is structured as follows. Section~\ref{sec:background} contains preliminary definitions and results relating to simplicial complexes, polytopes, and building sets. In Section~\ref{sec:polytopality}, it is shown that the extended nested complex is isomorphic to the boundary of a polytope, and in Section~\ref{sec:isomorphisms}, we provide some examples of isomorphisms between extended and non-extended nested complexes. Section~\ref{sec:hgamma} provides recursive formulas for the $f$- and $h$-polynomials of the extended nestohedron, and we prove that the $f$-vectors for certain nestohedra and extended nestohedra are the same. We also prove Gal's conjecture for all flag extended nestohedra. In Section~\ref{sec:comb_gamma}, we discuss extended $\mc B$-forests and extended $\mc B$-permutations, which are both in bijection with the vertices of the extended nestohedron $\mc P^{\sq}(\mc B)$. We then show that the $h$- and $\gamma$-polynomials for extended nestohedra on a special class of building sets can be written as a combinatorial formula in terms of the extended $\mc B$-permutations. Section~\ref{sec:parperm} provides a partial ordering on partial permutations and several nice properties of this partial ordering.
\section{Background}\label{sec:background}
In this section, we provide some preliminary definitions and results on simplicial complexes, polytope theory, building sets, and (extended) nestohedra.

\subsection{Simplicial Complexes and Polytope Theory}

A \textbf{simplicial complex} $\Delta$ on a finite set $S$ is a collection of subsets of $S$ such that for every $X \in \Delta$ and every $Y \subseteq X$, we have $Y \in \Delta$ as well. The elements of $\Delta$ are called \textbf{faces}, and the maximal elements are called \textbf{facets}. Note that, by definition, a simplicial complex is uniquely determined by its facets. The \textbf{dimension} of a face $F \in \Delta$ is $\dim F=\abs{F}-1$; the \textbf{dimension} of $\Delta$ is $\dim \Delta = \max\{\dim F \mid F \in \Delta\}$. A simplicial complex is called \textbf{pure} if all of its facets have the same dimension.

For a simplicial complex $\Delta$, the \textbf{closure} of a collection $\calS$ of its faces is \[\Cl \calS=\{F \in \Delta \mid F \subseteq \sigma\text{ for some }\sigma \in \calS\}.\] 
The \textbf{(open) star} of a face $\sigma$ is 
\[\St\sigma = \{F\in\Sigma\mid\sigma\subseteq F\},\]
and the star of a collection $\mc S$ of faces is
\[\St \mc S=\bigcup_{\sigma\in\mc S}\St\sigma.\]
The \textbf{link} of $\calS$ is then defined as \[\Lk \calS=\Cl \St S-\St \Cl S.\] If $\calS$ consists of just one face $\sigma$, then the link of $\sigma$ in $\Delta$ can be described as \[\Lk_\Delta \sigma =\{F \in \Delta \mid \sigma \cap F = \varnothing, \sigma \cup F \in \Delta\}.\]

A \textbf{subcomplex} of a simplicial complex $\Delta$ is a subset of $\Delta$ that is also a simplicial complex. Note that for any collection of faces $\calS$ of $\Delta$, the closure and link of $\calS$ are always subcomplexes, but the star of $\calS$ need not be. We also call the closure of $\calS$ the subcomplex generated by the elements of $\calS$.

A \textbf{morphism} of simplicial complexes $\Phi:\Delta \rightarrow \Delta'$ is a map from vertices of $\Delta$ to vertices of $\Delta'$ such that if $\{v_1,v_2,\ldots, v_r\}$ is a face of $\Delta$ then $\{\Phi(v_1),\Phi(v_2) ,\ldots,\Phi(v_r)\}$ is a face of $\Delta'$. A morphism of simplicial complexes induces a morphism on the link of any face in a simplicial complex. A morphism $\Phi:\Delta \rightarrow \Delta'$ is an \textbf{isomorphism} of simplicial complexes if there exists an inverse $\Phi^{-1}$ such that $\Phi\circ\Phi^{-1}$ and $\Phi^{-1}\circ\Phi$ are the identities on $\Delta$ and $\Delta'$ respectively. The following proposition follows from these definitions.

\begin{proposition}
\label{prop:def_simplic_iso}
$\Phi:\Delta \rightarrow \Delta'$ is a simplicial complex isomorphism if and only if $\Phi$ extends to a bijection between the faces of $\Delta$ and $\Delta'$.
\end{proposition}

Given a finite set of points $S=\{x_1,\dots,x_k\}$ in $\R^n$, the \textbf{convex hull} of $S$, denoted $\Conv(S)$, is defined to be the smallest convex set that contains $X$, or \[\Conv(S)\coloneqq\left\{\sum_{i=1}^k a_ix_i \mid a_i \ge 0 \; \text{ for all } i=1,\dots,k\text{ and }\sum_{i=1}^k a_i=1\right\}.\] 
Such a space is called a \textbf{convex polytope}. We can also define a convex polytope as a bounded intersection of finitely many \textbf{closed half-spaces}, where a closed half-space $\calH$ is the collection of points $(x_1,\dots,x_n)$ in $\R^n$ satisfying a linear inequality $a_1x_1+\dots+a_nx_n \le b$ for some $a_1,\dots,a_n,b \in \R$. Since we only consider convex polytopes in our paper, we shall simply write polytope instead of convex polytope.

The \textbf{vertices} of a polytope $P$ are the minimal set of points $S$ such that $P=\Conv(S)$. The \textbf{dimension} of $P$ is then defined to be the dimension of the affine linear span of $S$, considered as an affine subspace in $\R^n$ 
A \textbf{face} of a polytope $P$ is an intersection of $P$ with a closed half-space $\calH$ such that none of the interior points of $P$ lie on the boundary of the half-space. Note that each face of a polytope $P$ is itself a polytope. The \textbf{boundary} of a polytope $P$, denoted $\partial P$, is the union of all proper faces of $P$. If $P$ is a $d$-dimensional polytope, then the \textbf{vertices}, \textbf{edges}, \textbf{ridges}, and \textbf{facets} of $P$ are the $1$, $2$, $(d-2)$, and $(d-1)$-dimensional faces respectively.

Given a polytope $P$, the \textbf{dual polytope} $P^*$ can be defined as the set of points $y$ in the dual space $(\R^n)^*$ such that $\langle y,x \rangle >0$ for all $x \in P$, where $\langle \cdot,\cdot\rangle$ is the usual pairing. We can see that when dualizing, $\dim P^*=d$ if $P$ is $d$-dimensional, and that the $k$-dimensional faces of $P$ correspond to the $(d-k)$-dimensional faces of $P^*$ for all $k=0,\dots,d$. Furthermore, the double dual of a polytope $P$ is isomorphic to $P$ itself.

A polytope is called \textbf{simplicial} if all of its faces are simplices. In other words, all of its faces are the convex hull of $d+1$ points if the face is $d$-dimensional. A polytope is called \textbf{simple} if every vertex is adjacent to exactly $d$ edges, where $d$ is the dimension of the polytope. A polytope is simplicial if and only if its dual is simple. 

Given a simplicial polytope $P$, its boundary can be viewed as a simplicial complex by taking the set of vertices for each face $F$ of $P$. More generally, given a simplicial complex $\Delta$, a \textbf{geometric realization} of $\Delta$ is a simplicial polytope $P$ and a map $\phi: V(\Delta) \to V(P)$ of the vertex sets of $\Delta$ and $P$ such that $\sigma$ is a face of $\Delta$ if and only if $\Conv(\phi(\sigma))$ is a face of $P$.

If $P$ is a $d$-dimensional polytope, then the \textbf{face number} $f_i(P)$ is the number of $i$-dimensional faces of $P$. We call the vector $(f_0(P),\ldots,f_d(P))$ the \textbf{$f$-vector} of $P$, and the polynomial
\[f_P(t)\coloneqq \sum_{i=0}^d f_i(P)t^i\]
the \textbf{$f$-polynomial} of $P$. Note that if $f=(f_{-1},f_0,\dots,f_d)$ is the $f$-vector of a polytope $P$ with $f_{-1}\coloneqq 1$, then $f'=(f_d,\dots,f_0,f_{-1})$ is the $f$-vector of its dual $P^*$.

We can more compactly encode the face numbers of $P$ using smaller nonnegative integers. The \textbf{$h$-vector} $(h_0(P),\ldots,h_d(P))$ and \textbf{$h$-polynomial} $h_P(t)\coloneqq \sum_{i=0}^d h_i(P)t^i$ of a simple polytope $P$ are determined uniquely by the relation
\[f_P(t)=h_{P}(t+1).\]

From the Dehn-Sommerville relation, the $h$-vector of a simple polytope is always symmetric; in other words, we have $h_i=h_{d-i}$ for all $i=0,\dots,\lfloor{ d/2\rfloor}$.

\begin{definition}\label{defn:flag}
A simplicial complex $\Delta$ is called \textbf{flag} if a collection $C$ of vertices of $\Delta$ forms a simplex in $\Delta$ if and only if there exists an edge in the $1$-skeleton of $\Delta$ between any two vertices in $C$.
\end{definition}

If $\Delta_P$ is a flag simplicial complex, then we say that its dual simplicial complex, $P$ is a \textbf{flag polytope}. Definition~\ref{defn:flag} is equivalent to saying that a simple polytope $P$ is \textbf{flag} if any collection of pairwise intersecting facets has non-empty intersection.

The \textbf{$\gamma$-vector} gives another encoding of the $f$- and $h$-vectors of a simple poytope $P$, but with smaller integers. The entries $\gamma_i(P)$ of the $\gamma$-vector $(\gamma_0,\gamma_1,\ldots\gamma_{\lfloor{d/2\rfloor}})$ and $\gamma$-polynomial $\gamma_P(t)\coloneqq\sum_{i=0}^{\lfloor{d/2\rfloor}}\gamma_i t^i$ are determined by the $h$-polynomial:
\[h_P(t)=\sum_{i=0}^{\lfloor{d/2\rfloor}}\gamma_i t^i (1+t)^{d-2i}=(1+t)^d\gamma_P\left(\frac{t}{(1+t)^2}\right).\]

Gal conjectured the following statement about the $\gamma$-vector.

\begin{conjecture}[\cite{gal}]\label{conj:gal}
The $\gamma$-vector of any flag simple polytope has nonnegative entries.
\end{conjecture}

In Section \ref{subsec:gals_conjecture}, we prove Gal's Conjecture for all flag extended nestohedra $\calP^\sq(\calB)$.

\subsection{Building Sets, Nested Set Complexes, and Extended Nested Set Complexes}

In this subsection, we introduce definitions related to building sets, nested set complexes, and extended nested set complexes.

\begin{definition}\label{defn:buildingset}
A \textbf{building set} $\mc B$ on a finite set $S$ is a collection of subsets of $S$ satisfying two conditions:
\begin{enumerate}[(B1)]
    \item $\{i\} \in \mc B$ for all $i \in S$;\label{item:bs_B1}
    \item For any $I,J \in\mc B$ such that $I \cap J \ne \varnothing$, $I \cup J \in\mc B$.\label{item:bs_B2}
\end{enumerate}
\end{definition}

\begin{definition}\label{defn:restrict/contract}
Let $\mc B$ be a building set on $S$ and $I\subseteq S$. The \textbf{restriction of $\mc B$ to $I$} is the building set on $I$
\[\mc B|_I\coloneqq \{J\mid J\subseteq I,J\in\mc B\}.\]
The \textbf{contraction of $\mc B$ by $I$} is the building set on $S\setminus I$
\[\mc B/I\coloneqq \{J\setminus (J\cap I)\mid J\in \mc B,J\nsubseteq I\}.\]
\end{definition} 

\begin{definition}
For any building set $\mc B$, let $\mc B_{\max}$ denote the set of maximal elements of $\mc B$ with respect to inclusion. Then for any $M \in \mc B_{\max}$, the restriction $\mc B|_{M}$ is called a \textbf{connected component} of $B$.

If $\mc B$ is a building set on $S$ and $S\in\mc B$, then we say that $\mc B$ is \textbf{connected}. Note that the elements of $\mc B_{\max}$ form a disjoint union of $S$, with $\mc B_{\max}=\{S\}$ if $\mc B$ is connected.
\end{definition}

We say that building sets $\mc B,\mc B'$ on $S$ are \textbf{equivalent} if there exists a permutation $\sigma:S\to S$ giving a one-to-one correspondence $\mc B\to \mc B'$. 

We can use graphs to define a large family of building sets.

\begin{definition}\label{defn:graphbs} 
Let $\Gamma$ be a directed graph without loops and multiple edges on node set $S$. The \textbf{graphical building set} $\mc B_\Gamma$ is defined to be $\{I \subseteq S \mid \Gamma|_I$ is strongly connected$\}$.
\end{definition}

\begin{example}\label{eg:path_building_set}
Let $\Gamma$ be the path graph on $[n]$, denoted $P_n$.
If the graph is labeled from left to right in increasing order, then the building set $\mc B_\Gamma$ consists of all subsets of $[n]$ that are intervals. For example, if $n=4$, then
\[\mc B_\Gamma = \{\{1\},\{2\},\{3\},\{4\},\{1,2\},\{2,3\},\{3,4\},\{1,2,3\},\{2,3,4\},\{1,2,3,4\}\}.\]
\end{example}

\begin{example}\label{eg:bs_notall_graphical}
Not all building sets are graphical. For example, consider \[\mc B=\{\{1\},\{2\},\{3\},\{4\},\{1,2\},\{2,3\},\{1,3\},\{1,2,3\},\{1,2,3,4\}\}.\] If $\mc B=\mc B_\Gamma$ for some directed graph $\Gamma$, then $\{1,2\},\{2,3\},\{1,3\} \in \mc B$ implies that $\Gamma$ must contain anti-parallel edges between $1$ and $2$, $2$ and $3$, and $3$ and $1$. However, $\{1,2,3,4\} \in \mc B$ implies that $\Gamma$ contains an edge from $4$ to some $u \in \{1,2,3\}$, and an edge from some $v  \in \{1,2,3\}$ to $4$. Then we would have $\{4,u,v\} \in \mc B$, but this is impossible.
\end{example}

\begin{definition}\label{defn:nestedcoll}
A \textbf{nested collection} $N$ of a building set $\mc B$ on \(S\) is a collection of elements $\{I_1,\dots,I_m\}$ of $\mc B \setminus \mc B_{\max}$, such that:
\begin{enumerate}[(N1)]
    \item For any $i \ne j$, $I_i$ and $I_j$ are pairwise disjoint or nested;\label{item:nested_N1}
    \item For any $I_{i_1},\dots,I_{i_k}$, pairwise disjoint, \(k \geq 2\), their union is not an element of $\mc B$.\label{item:nested_N2}
\end{enumerate}
\end{definition}

We call a nested collection \textbf{maximal} if no other nested collection properly contains it.

\begin{example}\label{eg:nested_collection}
Consider the building set $\mc B_\Gamma$ from Example~\ref{eg:path_building_set}. An example of a nested collection is
\[N = \{\{1\},\{3\},\{1,2,3\}\}.\]
It turns out $N$ is also a maximal nested collection for $\mc B_\Gamma$. An example of something that is not a nested collection is
\[N' = \{\{2,3\},\{3\},\{4\}\},\]
since $\{3\}\cup\{4\}=\{3,4\}$ is an element of our building set.
\end{example}

\begin{definition}\label{defn:nestedcomplex}
Let $\mc B$ be a building set. The \textbf{nested complex} $\calN(\mc B)$ is defined to be the simplicial complex with vertices $\{I \mid I \in \mc B \setminus \mc B_{\max}\}$ and faces $\{I_1,\ldots,I_m\}$ for every nested collection $\{I_1,\dots,I_m\} \subseteq \mc B \setminus \mc B_{\max}$.
\end{definition}

We now extend our definitions for nested collections and nested complexes.

\begin{definition}\label{defn:extnestedcoll}
Let $\mc B$ be a building set on $S$. An \textbf{extended nested collection} \[N^\sq=\{I_1,\ldots,I_m,x_{i_1},\ldots,x_{i_r}\}\] on $\mc B$ is a collection of elements of subsets $I_j\in\mc B$ and $x_i$ for $i\in S$ such that:
\begin{enumerate}[(E1)]
    \item The collection $\{I_1,\ldots,I_m\} \subseteq \mc B$ form a nested collection;\label{item:extnested_E1}
    \item For all $1\leq k\leq r$ and $1\leq j\leq m$, we have that $i_k\neq I_j$. \label{item:extnested_E2}
\end{enumerate}
\end{definition}

The $x_i$ elements are extensions of what \cite{DHV11} call \textbf{design tubings} on graphs. We refer to the $x_i$'s as \textbf{design vertices}. 

We say an extended nested collection is \textbf{maximal} if no other extended nested collection properly contains it. Notice that any (non-extended) nested collection can be considered as an extended nested collection that does not contain any $x_i$ elements. 

\begin{example}\label{eg:ext_nested_collection}
Again consider the building set $\mc B_\Gamma$ from Example~\ref{eg:path_building_set}. An example of an extended nested collection is
\[N = \{\{1\},\{3\},\{3,4\},x_2\}.\]
This is a maximal nested collection for our building set. An example of a collection that \textbf{is not} an extended nested collection is
\[N' = \{\{2,3\},\{3\},x_1,x_2\},\]
since the number $2$ appears as both an index of an $x_i$ and as an element of one of the subsets in $N'$.
\end{example}

\begin{definition}\label{defn:extnestedcomplex}
Let $\mc B$ be a building set on $S$. The \textbf{extended nested complex} $\calN^{\sq}(\mc B)$ is defined to be the simplicial complex with vertices $\{I \mid I \in \mc B\} \cup \{x_i\mid i\in S\}$ and faces $\{I_1,\ldots,I_m\} \cup \{x_{i_1},\ldots,x_{i_r}\}$ where $\{I_1,\ldots,I_m,x_{i_1},\ldots,x_{i_r}\}$ is an extended nested collection.
\end{definition}

The extended nested complex is referred to as the \textbf{design nested complex} in \cite{MP17}. For a building set $\mc B$ and extended nested complex $\mc N^{\sq}(\mc B)$, we refer to the elements of $\mc B$ and the design vertices collectively as the vertices of $\mc N^{\sq}(\mc B)$. At times, we may also refer to the elements of $\mc B$ as the vertices of $\mc N(\mc B)$.

Notice that for a building set $\mc B$ on \(S\), the nested complex $\mc N(\mc B)$ is isomorphic to the subcomplex of the extended nested complex $\mc N^{\sq}(\mc B)$ involving neither the $x_i$ vertices nor the vertices corresponding to elements in \(\mc B_{\max}\).

In \cite{postnikov2009permutohedra,FS2005matroidminkowski,carrDevadoss2006coxeter}, it was shown that for a building set $\mc B$, the nested complex $\mc N(\mc B)$ is isomorphic to the boundary of a simplicial polytope. The polar dual of this polytope is the simple polytope $\mc P(\mc B)$ called the \textbf{nestohedron}. In Section~\ref{sec:polytopality}, we show that the extended nested complex $\mc N^{\sq}(\mc B)$ is also isomorphic to the boundary of a simplicial complex. We call the polar dual of this polytope the \textbf{extended nestohedron}, and we denote it by $\mc P^{\sq}(\mc B)$.

We now state some basic properties of the nested complex and the extended nested complex. The first observation is that these complexes are pure. For a building set $\mc B$ on $S$, Zelevsinsky showed that the nested complex $\mc N(\mc B)$ is pure of dimension $\abs{S}-\abs{\mc B_{\max}}$ (see \cite[Proposition 4.1]{zelevinsky2006nested}). We state and prove the result for the extended case.

\begin{proposition}\label{prop:pure_ext}
For a building set $\mc B$ on $S$, the extended nested complex $\mc N^{\sq}(\mc B)$ is pure of dimension $\abs{S}$.
\end{proposition}

\begin{proof}
Note that for an extended nested collection \[N^\sq=\{I_1,\dots,I_k\} \cup \{x_{i_1},\dots,x_{i_\ell}\}\] to be a facet of $\calN^\sq(\calB)$, we must have that $I_1 \cup \dots \cup I_k \cup \{i_1,\dots,i_\ell\}=S$, and that $N=\{I_1,\dots,I_k\}$ is a maximal nested collection of $\calB|_{S\setminus\{i_1,\dots,i_\ell\}}$. We then use the result from the non-extended case to conclude that $N$ has $\abs{S}-\ell$ elements, and so $N^\sq$ has $\abs{S}$ elements, completing the proof. 
\end{proof}

Next, we prove that $\calN(\calB)$ and $\calN^\sq(\calB)$ only depend on the connected components of $\calB$. To make this precise, we first recall the definition of the join of simplicial complexes.

\begin{definition}\label{def:join_simplicial_complexes}
For two simplicial complexes $X,Y$, their \textbf{join} $X*Y$ is the simplicial complex such that:
\begin{enumerate}[(\alph*)]
    \item The vertex set is equal to the disjoint union of the vertex sets of $X$ and $Y$;
    \item The faces are of the form $F_X \sqcup F_Y$ where $F_X$, $F_Y$ are faces of $X$ and $Y$ respectively.
\end{enumerate}
\end{definition}

\begin{lemma}\label{lem:IsJoinOfConnectedComponents}
Let $\mc B$ be a building set, and let $\mc B_1,\dots,\mc B_k$ be its connected components. We then have
\[\calN(\mc B) \simeq \calN(\mc B_1) * \dots * \calN(\mc B_k)\quad\quad\text{and}\quad\quad \calN^\sq(\mc B) \simeq \calN^\sq(\mc B_1) * \dots * \calN^\sq(\mc B_k).\]
\end{lemma}

\begin{proof}
The proof of both parts is based on the observation that a nested collection (resp., extended nested collection) on $\calB$ is the same as the disjoint union of nested collections (resp. extended nested collections) on its connected components $\calB_1,\dots,\calB_k$.
\end{proof}

\begin{remark}\label{rmk:suffice_to_consider_connected} In \cite[p.5 Corollary 5]{volodin}, the author shows that for any building set $\calB$ on $[n]$, there exists a connected building set $\calB'$ on $[n-\abs{\calB_{\max}}+1]$ such that $\calN(\calB) \simeq \calN(\calB')$. Thus in most cases, it suffices to consider building sets that are connected for the nested complex.
\end{remark}

\begin{remark}
When $G$ is one of the three graphs in Figure \ref{fig:not_nested_complex}, then computer checking shows that $\mc N^{\sq}(\mc B_G)$ is not isomorphic to $\calN(\mc B')$ for any connected building set $\mc B'$ on $5$ elements. Thus, the class of extended nested complexes is not contained within the class of nested complexes.
\end{remark}

\begin{figure}[H]
    \centering
    \includegraphics{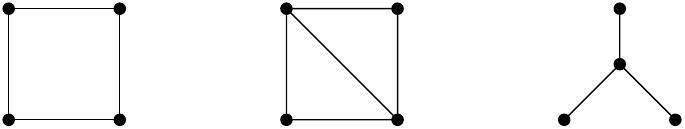}
    \caption{Graphs $G$ for which $\mc N^{\sq}(\mc B_G)$ is not a nested complex}
    \label{fig:not_nested_complex}
\end{figure}

Next, we characterize flag extended nested complexes in terms of when a non-extended nested complex is flag.

\begin{proposition}\label{lem:flag}
For any building set $\mc B$, $\calN(\mc B)$ is flag if and only if $\calN^{\sq}(\mc B)$ is flag.
\end{proposition}
\begin{proof}
The vertex set of $\calN(\mc B)$ is a subset of the vertex set of $\calN^{\sq}(\mc B)$, so we can refer to collections of elements of $\mc B\setminus \{[n]\}$ as a potential face in either one. Note that a collection of elements of $\mc B\setminus \{[n]\}$ is a face in $\calN(\mc B)$ if and only if it is a face in $\calN^{\sq}(\mc B)$. Recall that a simplicial complex is flag if and only if these minimal non-faces all only use two vertices. We will look at the minimal non-faces of these two simplicial complexes. 

Suppose $\calN^{\sq}(\mc B)$ is flag. Notice that any minimal non-face $F$ in $\calN(\mc B)$ is also a non-face in $\calN^{\sq}(\mc B)$. In particular, it is also minimal because subsets of the vertices of $F$ in $\calN^{\sq}(\mc B)$ are still only a face if they are a face in $\calN(\mc B)$. Therefore we must have $\abs{F} = 2$, and so $\calN(\mc B)$ is flag.

Now suppose that $\calN(\mc B)$ is flag. If $F$ is a non-face in $\mc N^{\sq}(\mc B)$, then $F$ can be of one of the following forms:
\begin{enumerate}
    \item $F$ consists only of vertices from $\mc B\setminus\{[n]\}$, or
    \item $F$ contains some vertex $x_i$, or
    \item $F$ contains the vertex $[n]$.
\end{enumerate}
If we are in the first case, then by the same argument above, any minimal non-face of $\calN^{\sq}(\mc B)$ which uses only vertices which are also in $\calN(\mc B)$ must be of size $2$. 

Now suppose we are in the second case. If $F \setminus \{x_i\}$ is not a face, then $F$ is not a minimal non-face. If $F\setminus \{x_i\}$ is a face, then $x_i$ must be stopping $F$ from being a face somehow. Therefore there exists $I \in F$ such that $i \in I$. But then $\{I,x_i\}\subseteq F$ is a non-face, so either $F$ is not minimal, or it has size $2$. 

The last case, when $[n]$ is the vertex of $F$ which doesn't appear in $\calN(\mc B)$, is analogous to the second case.
\end{proof}

\begin{remark}\label{rmk:graphical_is_flag}
\cite[Propostion 7.1]{PRW} characterizes when the building set $\mc B$ has $\mc N(\mc B)$ flag; in particular, for any graphical building set $\mc B_\Gamma$, the nested complex $\mc N(\mc B_\Gamma)$ is flag. Thus, we have that all of the equivalent characterizations given by \cite{PRW} are also equivalent to the extended nested complex $\mc N^{\sq}(\mc B)$ being flag, and the complex $\mc N^{\sq}(\mc B_\Gamma)$ is flag, where $\mc B_\Gamma$ is a graphical building set.
\end{remark}

We now provide a characterization of the link of a vertex $v$ in the extended nested complex $\calN^\sq(\calB)$. This allows us to ``build up'' the complex in terms of smaller complexes. 

\begin{definition}\label{defn:extlink}
Let \(\mc B\) be a building set on \(S\), and \(\calN(\mc B)\) the associated nested complex. For every \(C \in \mc B \setminus \mc \mc \calB_{\max}\), the \textbf{link} of \(C\) in \(\calN(\mc B)\) is 
\[\calN(\mc B)_C = \{N \in \calN(\mc B) \mid N \cap \{C\} = \varnothing, N \cup \{C\} \in \calN(\mc B)\}.\]
For any \(C \in \{I \mid I \in \mc B\} \cup \{x_i \mid i \in S\}\), the \textbf{link} of \(C\) in the extended nested complex \(\calN^\sq(\mc B)\) is
\[\calN^\sq(\mc B)_C = \{N^\sq \in \calN^\sq(\mc B) \mid N^\sq \cap \{C\} = \varnothing, N^\sq \cup \{C\} \in \calN^\sq(\mc B)\}.\]
\end{definition}

Zelevinsky first found a formula for the link decomposition for the nested complex $\calN(\calB)$ in \cite{zelevinsky2006nested}, and Aisbett provides an alternative proof in \cite[Lemma 3.2]{aisbett12}.

\begin{proposition}[\cite{zelevinsky2006nested}, Proposition 3.2] Let $\calB$ be a building set on $S$. Then the link of $C \in \calB$ in $\calN(\calB)$ is isomorphic to $\calN(\calB|_C) *\calN(\calB /C)$.
\end{proposition}

We now state the analogous link decompositions for extended nested complexes. Since the extended nested complexes have two kinds of vertices, those labeled by elements of the building set and those labeled by design vertices, we have two different formulas for a link decomposition.
Note that these formulas have appeared without proof in \cite[Lemma 84]{MP17} for graphical building sets.

\begin{proposition}\label{prop:linkdecomp_extended_x}
Let $\mc B$ be a building set on $S$, and let $v$ be a vertex of the extended nested complex $\mc N^{\sq}(\mc B)$ corresponding to the design vertex $x_i$. Then, the link of $v$ in $\mc N^{\sq}(\mc B)$ is given by
\[\mc N^{\sq}(\mc B)_v\simeq\mc N^{\sq}(\mc B_1)*\cdots*\mc N^{\sq}(\mc B_k),\]
where $\mc B_1,\ldots,\mc B_k$ are the connected components of $\mc B|_{S\setminus\{i\}}$.
\end{proposition}

\begin{proposition}\label{prop:linkdecomp_extended_I}
Let $\mc B$ be a building set on $S$, and let $v$ be a vertex of the extended nested complex $\mc N^{\sq}(\mc B)$ corresponding to an element of the building set $C \in\mc B$. Then, the link of $v$ in $\mc N^{\sq}(\mc B)$ is given by
\[\mc N^{\sq}(\mc B)_v\simeq \mc N(\mc B|_C)*\mc N^{\sq}(\mc B/C).\]
\end{proposition}

\begin{proof}[Proof of Proposition~\ref{prop:linkdecomp_extended_x}]
By definition, we have: \begin{align*}
        \calN^\sq(\mc B)_{x_i}&=\{\{x_{j_1},\dots,x_{j_\ell}\} \cup \{I_1,\dots,I_k\} \in \calN^\sq(\mc B) \mid \{x_i,x_{j_1},\dots,x_{j_\ell}\} \cup \{I_1,\dots,I_k\} \in \calN^\sq(\mc B)\}\\
        &=\{\{x_{j_1},\dots,x_{j_\ell}\} \cup \{I_1,\dots,I_k\} \in \calN^\sq(\mc B) \mid \cup_{s=1}^k I_s \cup \{j_1,\dots,j_\ell\} \subseteq S\setminus \{i\}\} \\
        &= \left\{\{x_{j_1},\dots,x_{j_\ell}\} \cup \{I_1,\dots,I_k\} \in \calN^\sq\left(\mc B|_{S \setminus \{i\}}\right)\right\} \\
        &= \calN^\sq\left(\mc B|_{S \setminus \{i\}}\right) \\
        &\simeq \calN^\sq(\mc B_1) * \dots * \calN^\sq(\mc B_k),
    \end{align*}
    where the last isomorphism follows from Lemma~\ref{lem:IsJoinOfConnectedComponents}.
\end{proof}

\begin{proof}[Proof of Proposition~\ref{prop:linkdecomp_extended_I}]
For a vertex $v$ corresponding to the building set element $C\in \mc B$, the link of $v$ in $\mc N^{\sq}(\mc B)$ corresponds to all extended nested collections containing $C$. Thus, we will show that the complex of extended nested collections of $\mc B$ that contain $C$ is isomorphic to $\mc N(\mc B|_C) * \mc N^{\sq}(\mc B/C)$. Let one direction of the isomorphism be given by the map
    \[(N_1,N_2)\in\mc N(\mc B|_C) * \mc N^{\sq}(\mc B/C)\mapsto N_1\cup N'_2\cup \{C\},\]
    where 
    \[N'_2\coloneqq \{I\mid I\in N_2\text{ and }I\cup C\notin\mc B\}\cup\{I\cup C\mid I\in N_2\text{ and }I\cup C\in\mc B\}\cup\{x_i\mid x_i\in N_2\}.\]

    If $N$ is an extended nested set of $\mc B$ containing $C$, then the inverse of the above map is given by
    \[N\mapsto N_1\cup N_2,\]
    where
    \[N_1\coloneqq \{I\in N \mid I\subsetneq C\}\quad\text{and}\quad N_2\coloneqq \{I\setminus (I\cap C)\mid I\in N,I\nsubseteq C\}\cup\{x_i\mid x_i\in N\}.\]
    Notice that $N_1\in \mc N(\mc B|_C)$ and $N_2\in\mc N^{\sq}(\mc B/C)$. 
    
    Both of these maps preserve inclusion. Thus we have an isomorphism.
\end{proof}
\section{Polytopality}\label{sec:polytopality}
In this section, we provide two proofs of the fact that $\calN^\sq(\calB)$ can be realized as the boundary of a polytope; this is equivalent to showing that its dual $\calP^\sq(\calB)$ can also be realized as the boundary of a polytope. Our first proof is based on stellar subdivisions of a cross polytope, and our second proof is based on Minkowski sums.

\begin{definition}
Let $\Delta$ be a simplicial complex and $F$ be a face of $\Delta$. The \textbf{stellar subdivision} on the face $F$ of $\Delta$ is defined to be \[\Delta'=(\Delta \setminus \Cl \St_\Delta(F)) \sqcup (\{v\} * \partial F * \Lk_\Delta(F)).\] In other words, we remove the subcomplex generated by all facets of $\Delta$ containing $F$, then add in the subcomplex generated by all facets of the form $\{v\} \cup (F \setminus \{u\}) \cup G$ where $v$ is a new vertex, $u$ is an element of $F$ and $G$ is a facet of $\Lk_\Delta(F)$.
\end{definition}

\begin{remark}
If $\Delta$ has a geometric realization $P$, and $\Delta'$ is the result of stellar subdivision on a face $F$ of $\Delta$, then $\Delta'$ can be realized as $\Conv(P \cup \{v\})$ for any point $v$ that lies beyond the facet hyperplanes for facets that contain $F$ and beneath any other facet hyperplanes of $P$. For a precise definition of \say{beyond} and \say{beneath}, see \cite{ewald1974stellar}.
\end{remark}

We now provide a geometric realization of $\calN^\sq(\calB)$ for any building set $\calB$.

\begin{theorem}
\label{thm:bndpolytope}
Let $\calB$ be a building set on $[n]$. Then $\mc N^{\sq}(\mc B)$ can be realized as the boundary of a polytope $\calN_\calB$ in the following way:
\begin{enumerate}[(i)]
    \item Consider $\R^n$ with standard basis vectors $e_1,\dots,e_n$. Start with the cross polytope in $\R^n$ with vertices $e_i$ labeled $\{i\}\in \mc B$ and vertices $-e_i$ labeled $x_i$ for all $i \in [n]$.
    \item Order the non-singletons of $\mc B$ by decreasing cardinality, then for each $I \in \mc B$ a non-singleton, perform stellar subdivision on the face $\mathcal{I}=\{\{i\} \mid i \in I\}$, with the new added vertex labeled $I$.
    \item The boundary of the resulting polytope $P_\calB$ will be isomorphic to $\calN^\sq(\calB)$.
\end{enumerate}
\end{theorem}

Before providing the proof, we give an example of the process of obtaining the polytope $P_\mc B$.

\begin{example}\label{eg:polytopality1}
Consider the building set $\mc B=\{\{1\},\{2\},\{3\},\{1,2\},\{1,2,3\}\}$. We begin with a three-dimensional cross polytope $P_0$ with vertices labeled by singletons $1,2,3$ as well as $x_1,x_2$, and $x_3$. The polytope $P_0$ is illustrated in Figure~\ref{fig:poly1}.

The remaining non-singleton elements of the building set are $\{1,2\}$ and $\{1,2,3\}$. Since $\{1,2,3\}$ is larger in terms of cardinality, we start by adding a vertex corresponding to this element to the polytope. To do this, we stellarly subdivide the face $\{\{1\},\{2\},\{3\}\}$ of $P_0$, obtaining a new polytope $P_1$. This is shown in Figure~\ref{fig:poly2}. Next, we add a vertex corresponding to the element $\{1,2\}$ by stellarly subdividing the face $\{\{1\},\{2\}\}$ of $P_1$, obtaining the final polytope $P_\mc B$, shown in Figure~\ref{fig:poly3}. The boundary of this polytope is isomorphic to $\mc N^{\sq}(\mc B)$.

\begin{figure}[H]
\centering
\subfigure[$P_0$]{
\label{fig:poly1}
\includegraphics[scale=0.6]{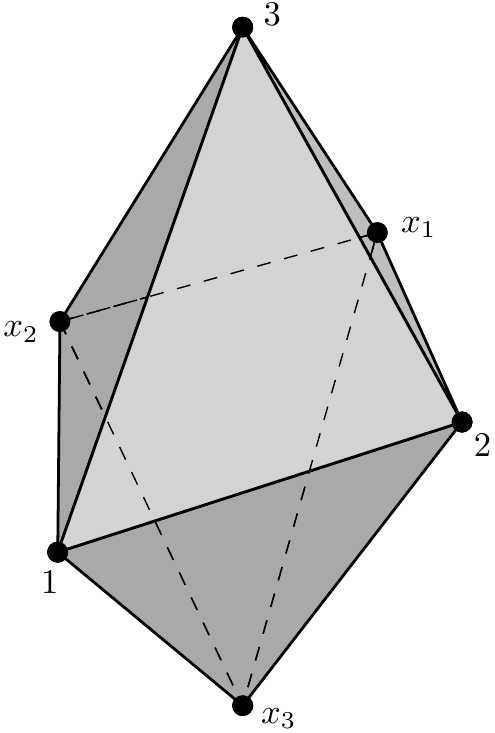}}\qquad 
\subfigure[$P_1$]{
\label{fig:poly2}
\includegraphics[scale=0.6]{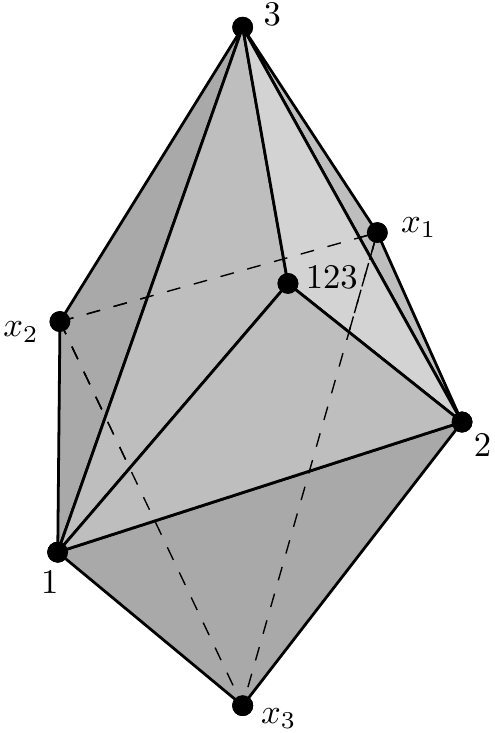}}\qquad
\subfigure[$P_\mc B$]{
\label{fig:poly3}
\includegraphics[scale=0.6]{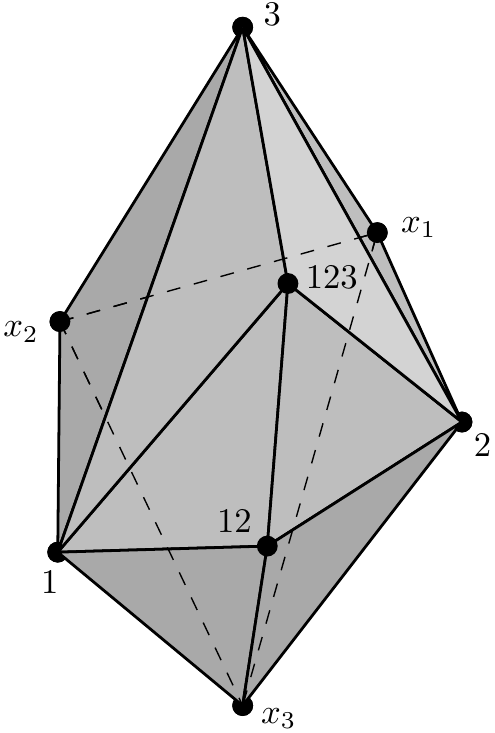}}
\caption{Process of stellar subdivision of a cross polytope to obtain $P_\mc B$.}
\label{fig:polytopality}
\end{figure}
\end{example}

\begin{proof}[Proof of Theorem~\ref{thm:bndpolytope}]

First, note that our construction is always well-defined. When we stellarly subdivide any face $F$ of a polytope $P$, the faces that are removed are precisely the faces containing $F$, which means we do not remove any face with lower dimension than that of $F$, or any different face of the same dimension. This implies that the faces of the form $\calI=\{\{i\}\mid i\in I\}$ remain faces of $P_\calB$ after stellarly subdividing every face of dimension at least $\dim \calI$, hence step (ii) of our construction is always doable.

Our strategy is to show that the facets of $P_\calB$ are in bijection with the facets of $\calN^{\sq}(\mc B)$, where the bijection is the one sending $x_i$ to the vertex $x_i$ of $P_\calB$, and $I \in \calB$ to the vertex $I$ of $P_\calB$. Such a vertex always exists because either $I=\{i\}$ is a singleton, hence is in the original cross polytope, or $I$ is not a singleton and is obtained after stellar subdivision on the face $\calI$. With this identification of vertices, we can say that the facets of $P_\calB$ are the same as those of $\calN^\sq(\calB)$.

We prove the claim for any building set $\mc B$ on $[n]$ by downward induction on $\min\{\abs{I} \mid I \in \calB,\abs{I}>1\}$. If $\calB$ has no non-singletons, then define this minimum to be $\infty$. This is also the base case of our induction, where we have $\calB=\{\{1\},\dots,\{n\}\}$. Here, $P_\calB$ is the cross polytope, and its facets are \[\{\{j_1\},\dots,\{j_s\},x_{i_1},\dots,x_{i_r}\},\] for any $J=\{j_1,\dots,j_s\}$ and $I=\{i_1,\dots,i_r\}$ that satisfies $I \cap J = \varnothing$ and $I \cup J = [n]$. Notice that these are also the facets of $\calN^\sq(\calB)$ for $\calB=\{\{1\},\dots,\{n\}\}$, so our claim is proved for this case.

Now, assume that our claim is true for some building set $\calB$ on $[n]$, whose non-singletons are all of order at least $m$. It then suffices to show that if we add to $\calB$ a new subset $I \subseteq[n]$ of size at most $m$, then the facets of $\calN^\sq(\calB \cup \{I\})$ are the same as the facets of the polytope obtained by stellarly subdividing $P_{\mc B}$ on the face $\mc I$; denote this polytope by $P_{\mc B}'$. By definition, we obtain the facets of $P_\calB'$ by removing the facets $\calF$ containing $\calI$ of $P_\calB$, and adding in facets of the form $\{I\} \cup (\calI \setminus \{i\}) \cup \calG$, where $i \in I$ and $\calG$ is a facet of $\Lk_{P_\calB}(\calI)$. We now show that the same addition and removal of facets get us from $\calN^\sq(\calB)$ to $\calN^\sq(\calB \cup \{I\})$.

First, consider the facets of $\calN^\sq(\calB)$ that no longer remain facets of $\calN^\sq(\calB \cup \{I\})$. If 
\[N^\sq\coloneqq\{I_1,\dots,I_k\} \cup \{x_{j_1},\dots,x_{j_\ell}\}\] 
is such a facet, then $N^\sq$ not being a facet of $\calN^\sq(\calB \cup \{I\})$ means that it is no longer an extended nested collection. The only condition that could prevent $N^\sq$ from being an extended nested collection is the condition that there does not exist $I_{i_1},\dots,I_{i_r}\in\{I_1,\ldots,I_k\}$ pairwise disjoint such that $I_{i_1} \cup \dots \cup I_{i_r}=I$. This implies that there does exist such a collection whose union is $I$. Since the size of $I$ is the smallest among the non-singletons of $\calB$, each $I_{i_t}$ must be a singleton for all $t=i_1,\ldots,i_r$, hence $\{I_{i_1},\dots,I_{i_r}\}=\calI$. In other words, $N^\sq$ is a facet of $\calN^\sq(\calB)$ that contains $\calI$. Conversely, any facet of $\mc N^{\sq}(\mc B)$ that contains $\calI$ will fail to be a facet of $\calN^\sq(\calB \cup \{I\})$ for the same reason. Thus, the facets removed by stellarly subdividing $P_\mc B$ are indeed the maximal extended nested collections of $\mc B$ that fail to remain maximal extended nested collections of $\mc B\cup\{I\}$.

Now consider any facet $N^\sq$ of $\calN^\sq(\calB \cup \{I\})$ that is not a facet of $\calN^\sq(\calB)$. This can only happen if $I \in N^\sq$. Furthermore, by Proposition \ref{prop:linkdecomp_extended_I}, we have \[\calN^\sq(\calB \cup \{I\})_I \simeq \calN(\calB|_I) * \calN^\sq(\calB / I).\] Since $N^\sq \setminus \{I\}$ is a facet of $\calN^\sq(\calB \cup \{I\})_I$, it corresponds to the join of a facet of $\calN(\calB|_I)$ with a facet of $\calN^\sq(\calB / I)$. Since $\calB|_I=\calI \cup \{I\}$, its facets are of the form $\calI \setminus \{i\}$ for some $i \in I$, and so $N^\sq \setminus \{I\}$ must contain $\calI \setminus \{i\}$ for some $i \in I$. We now write $N^\sq=\{I\} \cup (\calI \setminus \{i\}) \cup N_1^\sq$, where $N_1^\sq$ is an extended nested collection for $\calB$; in fact, we have that $N_1^{\sq}$ is in $\Lk_{\calN^\sq(\calB)}(\calI)$. Thus, this is one of the facets that are added in the stellar subdivision on the face $\calI$ of $P_\calB$. Conversely, any facet of the form $\{I\} \cup (\calI \setminus \{i\}) \cup \calG$ for some $i \in I$ and $\calG$ a facet of $\Lk_{P_\calB}(\calI)$ is a facet of $\calN^\sq(\calB \cup \{I\})$, but is not a facet of $\calN^\sq(\calB)$ since $I \not\in \calB$. Therefore, the facets added by stellarly subdividing $P_\mc B$ are the maximal extended nested collections added when we go from $\mc B$ to $\mc B\cup\{I\}$. This proves the induction hypothesis.
\end{proof}

\begin{remark}\label{rmk:shave_faces}
In \cite{DHV11}, the authors provide a polytopal realization for the extended nestohedron $\calP^{\sq}(\mc B)$ when $\mc B=\calB_G$ is the building set of an undirected graph. Their argument works the same for general building set, and could be seen as the dual to our stellar subdivision approach. In particular, $\calP^\sq(\calB)$ is obtained as the boundary of the following polytope:

\begin{enumerate}[(i)]
    \item Take the $n$-dimensional cube $\mathcal{C}^n$, whose opposite facets are labeled by $\{i\}$ and $x_i$ for every $i \in [n]$,
    \item For each $I \in \mc B$ (ordered by decreasing cardinality), we shave face $I$, i.e., the face corresponding to the intersection of all facets $\{i\mid i\in I\}$.
\end{enumerate}

Here, a shaving of a face $F$ corresponding to a polytope $\calP$ is defined as follows: consider any closed half-space $\calH$ that intersect $\calP$ at exactly $F$. Then the shaving of $F$ corresponds is the intersection of $\calP$ with a closed half-space $\calH_\epsilon$ parallel to $\calH$, and is moved a small amount $\epsilon$ toward the polytope.
\end{remark}

We obtain a different polytopal realization of the extended nestohedron $\calP^\sq(\calB)$, as the following Minkowski sum.

\begin{theorem}\label{thm:minkowski_polytopality}
For a building set on $[n]$, the extended nestohedron $\calP^\sq(\calB)$ is isomorphic to the boundary of the polytope:
    \[
    \mc P \coloneqq \sum_{ i \in [n]} \Conv(0,e_i)  + \sum_{I \in \mc B} \Conv(\{e_S | S \subsetneq I\}),
    \]
    where $e_1,\dots,e_n$ are the standard basis vectors of $\R^n$, and $e_S=\sum_{i \in S} e_i$ for all $S \subseteq [n]$.
\end{theorem}

The intuition for the Minkowski sum is that we start with the cube $[0,1]^n$, and then each added term $\Conv(\{e_S \mid S \subsetneq I\})$ corresponds to shaving face $I$ of the cube. We again provide an example before providing the proof.

\begin{example}\label{eg:minkowski_polytope}
Consider the building set $\mc B=\{\{1\},\{2\},\{3\},\{1,2\},\{1,2,3\}\}$, which is the same building set as in Example~\ref{eg:polytopality1}. Then the desired polytope $\mc P$ will be the Minkowski sum
\[\Conv(0,e_1)+\Conv(0,e_2)+\Conv(0,e_3)+\Conv(e_1,e_2)+\Conv(e_1,e_2,e_3,e_1+e_2,e_1+e_3,e_2+e_3),\]
as illustrated in Figure~\ref{fig:minkowski}.
\begin{figure}[H]
    \centering
    \includegraphics[scale=0.8]{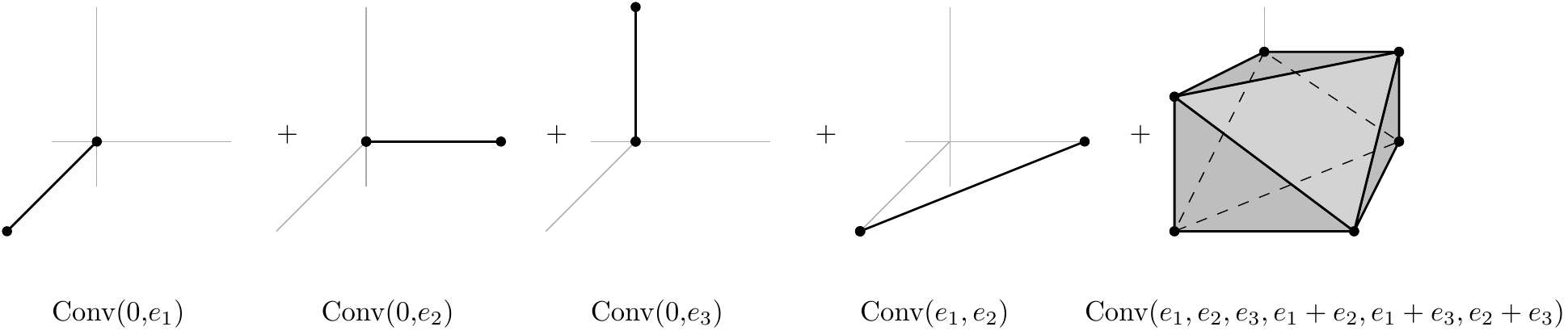}
    \caption{Decomposition of $\mc P$ into Minkowski sum.}
    \label{fig:minkowski}
\end{figure}
The resulting polytope $\mc P$ is shown in Figure~\ref{fig:polytope_minkowski} (not drawn to scale). Labelling the vertices by the maximal extended nested collections of $\mc B$, we see that $\mc P$ is indeed the dual of the polytope constructed in Example~\ref{eg:polytopality1}.

\begin{figure}[H]
    \centering
    \includegraphics[scale=0.6]{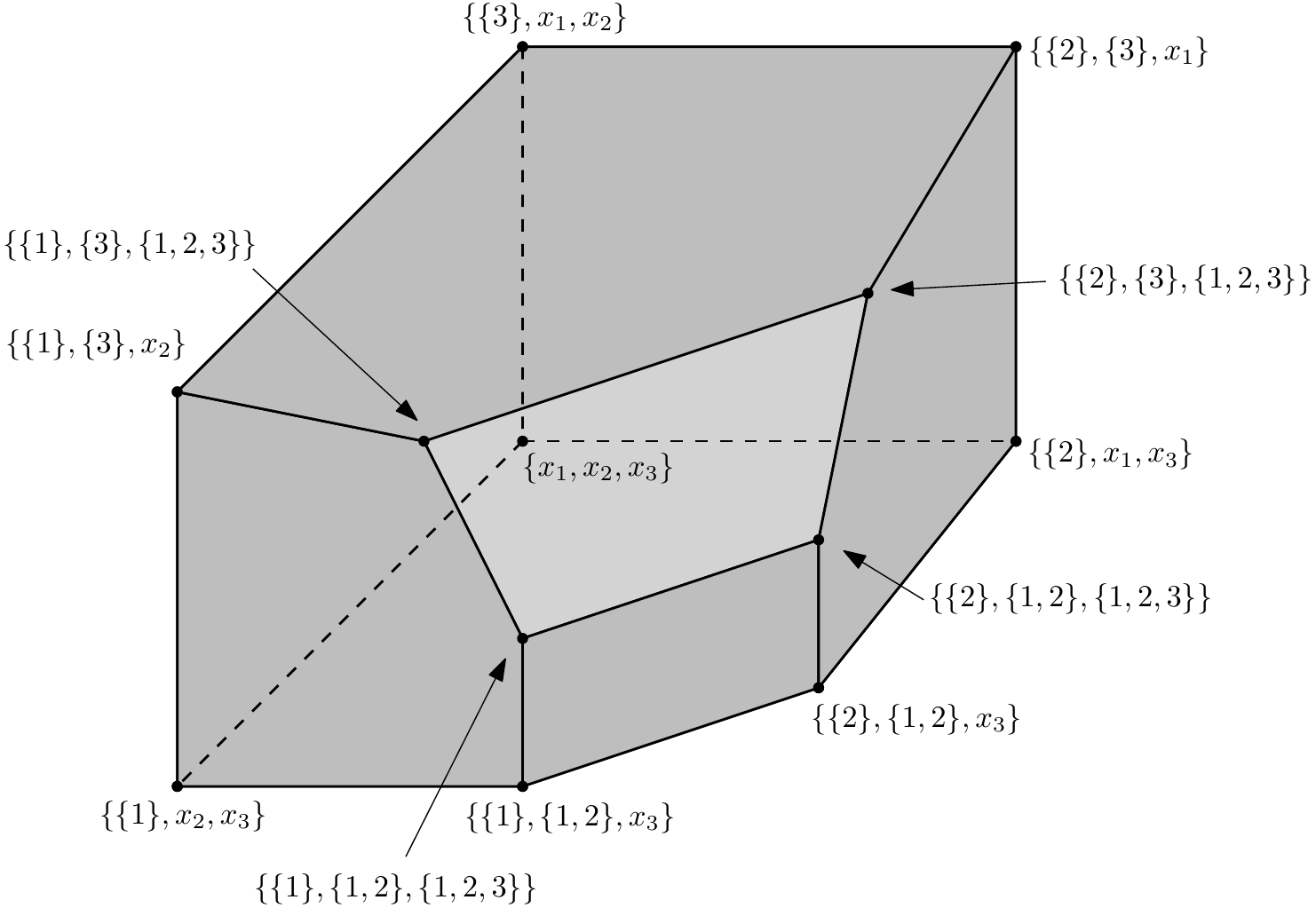}
    \caption{Polytope $\mc P=\mc P^{\sq}(\mc B)$.}
    \label{fig:polytope_minkowski}
\end{figure}

\end{example}

\begin{proof}[Proof of Theorem~\ref{thm:minkowski_polytopality}]
As in \cite{postnikov2009permutohedra}, each face of a polytope can be identified by the set of linear equations $f$ which are maximized on that face. The face which maximizes an equation $f$ in a Minkowski sum is the Minkowski sum of faces which are maximized at $f$. Our proof will involve two steps: (1) We form a map from extended nested collections of $\mc B$ to faces of $\mc P$ which is inclusion-reversing with facets corresponding to vertices, and (2) we show every linear equation is maximized exactly at a face corresponding (via our map from (1)) to an extended nested collection.

(1) View an extended nested collection $\tld{\sigma}$ for $\mc B$ as a nested collection $N_{\tld{\sigma}}$ of $\mc B$ and a set of $x_i$'s. Define $A_I$ for $I \in N_{\tld{\sigma}}$ as $A_{I} \coloneqq \{i \in I \mid \textnormal{ for any } J \in N_{\tld{\sigma}}, J \subsetneq I, i \not \in J \}$. Define $A_{-1} \coloneqq \{i \in [n]\mid x_i \in \tld{\sigma}\}$ and $A_0 \coloneqq [n] \setminus (A_{-1} \cup \bigcup_{I \in \mc B} A_I)$. A linear equation
\[
    f(\Vec{t}) = a_1t_1 + a_2t_2 + \cdots + a_nt_n,
\]
such that 
\begin{itemize}
    \item If $i \in A_{-1}$, then $a_i \leq 0$;
    \item If $i \in A_{0}$, then $a_i = 0$;
    \item If $i,j \in A_{I}$ for $I \in N_{\tld{\sigma}}$, then $a_i= a_j \geq 0$;
    \item If $i \in A_{I}, j \in A_J$ for $I,J \in N_{\tld{\sigma}}$ and $I \subsetneq J$, then $a_i \geq a_j$;
\end{itemize}
will be maximized on the same face $Q_I$ of $\Conv(\{x_S\mid S \subsetneq I\})$ for $I \in \mc B$ and the same face $F$ of $\sum_{i \in [n]} \Conv(0,e_i)$. Namely, 
\[
    F = \sum_{i \not \in A_{-1} \cup A_0} e_i + \sum_{i \in C_0} \Conv(0,e_i),
\]
for $I \in\mc B$ such that $I \cap A_{-1} \neq \varnothing$, we have that $Q_I = x_{I \setminus (A_{-1} \cap I)}$. For $I \in \mc B$ such that $I \cap A_{-1} =\varnothing$ and $I \cap A_{0} \neq \varnothing$, 
\[
Q_I = \Conv(x_{S}\mid S = I \setminus \{\textnormal{non-zero subsets of } (A_0 \cap I)\}).
\]
For $J \in \mc B$ such that $J \cap (A_{-1} \cup A_{0}) =\varnothing$, let $I$ be the smallest element of $N_{\tld{\sigma}}$ containing $I$. Then
\[ Q_J = \Conv(x_{S}\mid  S = I \setminus \{\textnormal{an element of } A_I\}). \]
Thus $f$ is maximized on the face $ F + \sum_{I \in \mc B} Q_I$, which we denote $Q_{\tld{\sigma}}$. Notice that if $\tld \sigma \neq \tld \pi$ are distinct extended nested collections for $\mc B$, then $Q_{\tld \sigma}\neq Q_{\tld \pi}$. To show that the map $\tld{\sigma} \mapsto Q_{\tld{\sigma}}$ is inclusion-reversing, first notice that adding an $x_j$ to $\tld{\sigma}$ corresponds to expanding the set of $f$ which are maximized on $Q_{\tld{\sigma}}$ to include $a_j \leq 0$ rather than just $a_j=0$. Next, denote the nested collection formed from adding an element $I^*$ to $\tld{\sigma}$ by $\tld{\sigma}^*$. Then letting $A_I^*$ be the $A_I$s for the nested collection $\tld{\sigma}^*$, we see that for a parent $I \in \tld{\sigma}$ of $I^*$, $A_{I}^* = A_I \cap ([n] \setminus I^*)$, and for any other $I \in \tld{\sigma}$, $A_{I}^* = A_I$. It follows that if $I^*$ has a parent in $\tld{\sigma}$, then the set of $f$ which are maximized on $Q_{\tld{\sigma}}^*$ consists of expanding the set of $f$ maximized on $Q_{\tld{\sigma}}$ to include $a_i \geq a_j$ for $i \in A(I^*)^*$ and $j$ in the parent of $I^*$ rather than just $a_i = a_j$ for $i \in A(I^*)^*$ and $j$ in the parent of $I^*$. If $I^*$ does not have a parent in $\tld{\sigma}'$, then the set of $f$ which are maximized on $Q_{\tld{\sigma}}^*$ consists of expanding the set of $f$ maximized on $Q_{\tld{\sigma}}$ to include $0 \leq a_j \leq a_i$ for $i \in A(I^*)^*$ and $j \in J$ for some $J \in N_{\tld{\sigma}}$ rather than just $a_i=0$ for $i \in A(I^*)^*$. 

(2) We will find the face $F_I$ of $\Conv(0,\{x_S| S \subsetneq I\})$ and the face $F$ of $\sum_{i \in [n]} \Conv(0,e_i)$ which maximizes an arbitrary linear equation of the form
\[
    f(\textbf{t}) = a_1t_1 + a_2t_2+ \cdots + a_nt_n.
\]
First, divide the $a_i$ into sets $C_{-1},C_0,C_1,\ldots,C_r$ as follows:
\begin{itemize}
    \item $C_{-1} = \{a_i\mid a_i<0 \}$;
    \item $C_0 = \{a_i\mid a_i=0 \}$;
    \item If $i,j \in C_k$, then $a_i = a_j$;
    \item If $k<\ell$ with $i \in C_k$ and $j \in C_{\ell}$, then $a_i<a_j$;
        \item $C_r$ is nonempty for every $r \geq 0$.
\end{itemize}
There is a unique set of subsets of $[n]$, $C_{-1},C_0 ,\ldots, C_r$ which satisfy these conditions ($C_r$ will be the set of largest $a_i$s, and so on).  The face $F$ is 
\[\sum_{i \not \in C_{-1}\cup C_{0}} e_i + \sum_{i \in C_0} \Conv(0,e_i).
\] 
For an element of the building set $I$ with $I \cap C_{-1} \neq \varnothing$, then $F_I = x_{I \setminus (C_{-1} \cap I)}$. For an element of the building set $I$ with $I \cap C_{-1} = \varnothing$ and $I \cap C_{0} \neq \varnothing$, 
\[
F_I = \Conv(x_{S}\mid S = I \setminus \{\textnormal{non-zero subsets of } (C_0 \cap I)\}).
\]
For $J \in \mc B$ such that $J \cap (C_{-1} \cup C_{0}) =\varnothing$, let $a$ be the smallest integer such that $J \cap C_a \neq \varnothing$, then
\[
F_J = \Conv(x_{S}\mid S = I \setminus \{\textnormal{an element of } C_a\}).
\]
Thus $f$ is maximized on the face $ F + \sum_{I \in\mc B} F_I$, which we denote $F_{\tld{\sigma}}$. We construct the extended nested set $\tld{\sigma}$ which corresponds to this face as follows: 
\begin{enumerate}[a)]
    \item If $i \in C_{-1}$, then $x_i \in \tld{\sigma}$.
    \item Add in maximal elements (which are unique) of $\mc B|_{[n] \setminus (C_0 \cup C_{-1})}$ to $\tld{\sigma}$. These are the maximal elements of $\tld{\sigma}$.
    \item Recursively, when we add an element $I$ to $(\tld \sigma)$, we do the following: Partition $I$'s elements into sets $I \cap C_1$, $I \cap C_2,$ and so on. Let $I \cap C_a$ be the first nonempty subset in this partition. Add in the maximal element under $I$ in $\mc B|_{I \setminus (I \cap C_a)}$.
\end{enumerate}
Conditions a) and b) show that $f$ is maximized on the same face of $\sum_{i \in [n]} \Conv(0,e_i)$ as the face $Q_{\tld{\sigma}}$ in the Minkowski sum. Let $Q_{\tld{\sigma}}|_I$ refer to the face of $\Conv(\{x_S| S\subsetneq I\}$ in the Minkowski sum of $Q_{\tld{\sigma}}$. We will show $Q_{\tld{\sigma}}|_I = Q_I$ for all $I \in B$ using the descriptions of $Q_I$ above.

For any $I \in \mc B$ with $I \cap C_{-1} \neq \varnothing$, $Q_{\tld{\sigma}}|_I = Q_I$ since condition a) implies $A_{-1} = C_{-1}$. For any $I \in \mc B$, with $I \cap C_{-1} = \varnothing$ and $I \cap C_0 \neq \varnothing$, conditions b) and c) implies $Q_{\tld{\sigma}}|_I = Q_I$ since they imply $A_0 = C_0$. For any other $J \in \mc B$, let $I$ be the smallest element of $\tld{\sigma}$ which contains $J$. Then, condition c) shows that 
\begin{align*}
   Q_{\tld{\sigma}}|_J &= \Conv(\{x_S \mid S = J \setminus i \textnormal{ for } i \in A(I)\}) \\
   &= \Conv(\{x_S \mid S = J \setminus i \textnormal{ for } i \in C_a \textnormal{ for minimal } a \textnormal{ such that } C_a \cap I \neq \varnothing)\}) = Q|_J,
\end{align*}
with the last equality coming from the fact that $J \cap C_a \neq \varnothing$; otherwise, there would be a smaller element of $\tld{\sigma}$ which contains $J$ by condition c). It follows that $f$ is maximized exactly on the face $Q_{\tld{\sigma}}$
\end{proof}
\section{Isomorphisms}\label{sec:isomorphisms}

In this section, we study isomorphisms within and between extended nested set complexes and nested set complexes. 
We begin by constructing some examples of isomorphisms within and between extended nested set complexes and nested set complexes. Once we have constructed these, we work towards characterizing the isomorphisms between extended nested set complexes.

\subsection{Graphical building sets}
First we consider isomorphisms in the graphical case. Manneville and Pilaud call isomorphisms between (extended) nested set complexes \textbf{trivial} if they arise from a graph isomorphism. We extend this definition to non-graphical building sets: We say building sets $\mc B$ and $\mc B'$ are isomorphic if there exists a bijection $\sigma:[n] \rightarrow [n]$ which extends to a bijection between the building set elements. Then $\sigma$ induces an isomorphism between (extended) nested set complexes. We call such an isomorphism \textbf{trivial}.

Given this terminology, Manneville and Pilaud show that the following graphs are the only graphs that we have to consider when studying non-trivial isomorphisms \cite{MP17}.

\begin{definition}\label{defn:spider_ocotpi}
Let $\underline{n}\coloneqq (n_1,\ldots,n_\ell)\in\mb N^\ell$.
\begin{itemize}
    \item The \textbf{spider graph} $\mf X_{\underline{n}}$ is a complete graph on vertices $\{v^i_0\}_{i\in[\ell]}$ (the body of the spider) and $\ell$ legs $[v_1^i,v_{n_i}^i]$ attached to vertex $v_0^i$.
    \item The \textbf{octopus graph} $\mc X_{\underline{n}}$ consists of a single vertex labeled $*$ (the head of the spider) and $\ell$ legs $[v_0^i,v_{n_i}^i]$ attached.
\end{itemize}
\end{definition}

\begin{figure}[H]
\centering
\subfigure[$\mf X_{\underline{n}}$]{
\label{fig:spider}
\includegraphics{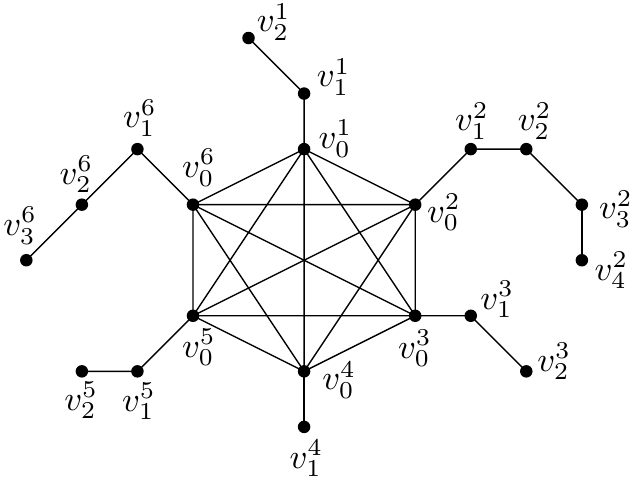}}\qquad\qquad
\subfigure[$\mc X_{\underline{n}}$]{
\label{fig:octopus}
\includegraphics{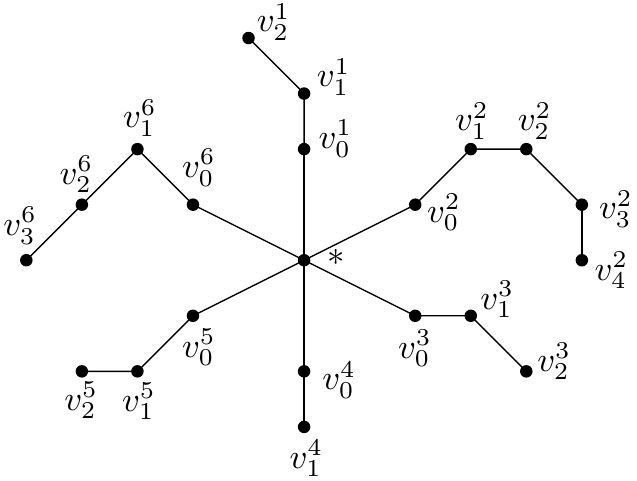}}
\caption{Spider and octopus graphs for $\underline{n}=(2,4,2,1,2,3)$.}
\label{fig:spi_oct}
\end{figure}

In particular, Manneville and Pilaud found that all non-trivial isomorphisms between graphical nested set complexes, i.e., those not induced by a graph isomorphism, occur between spider graphs. They also prove an analogous result for graphical extended nested set complexes and octopus graphs.

\begin{theorem}[\cite{MP17}, Theorem 44]\label{thm:mp_nonextended}
Let $G$ and $G'$ be two connected graphs and $\Phi:\mc N(\mc B_G) \rightarrow \mc N(\mc B_{G'})$ be a non-trivial isomorphism. Then $G$ and $G'$ are both the spider graph $\mf X_{\underline{n}}$.
\end{theorem}

\begin{theorem}[\cite{MP17}, Theorem 67]\label{thm:mp_extended}
Let $G$ and $G'$ be two connected graphs and $\Phi:\mc N^{\square}(\mc B_G) \rightarrow \mc N^{\square}(\mc B_{G'})$ be a non-trivial isomorphism. Then $G$ and $G'$ are both the octopus graph $\mc X_{\underline{n}}$.
\end{theorem}

Manneville and Pilaud also found that, for graphical building sets, the only isomorphism between extended and non-extended nested set complexes is between spider and octopus graphs.

\begin{theorem}[\cite{MP17}, Proposition 64]\label{thm:oct_spider_isomorphism}\label{thm:mp_spider_octopus}
Let $G$ and $G'$ be two connected graphs such that $\mc N^{\sq}(\mc B_G)\simeq\mc N(\mc B_{G'})$. Then $G$ is a spider $\mf X_{\underline{n}}$ and $G'$ is an octopus $\mc X_{\underline{n}}$.
\end{theorem}

Notice that if $\underline{n}=(0,\ldots,0)\in\Z_{\geq 0}^\ell$, then $\mf X_{\underline{n}}$ is $K_\ell$, the complete graph on $\ell$ vertices, and $\mc X_{\underline{n}}$ is the star graph on $\ell+1$ vertices $K_{1,\ell}$. Furthermore, if $\underline{n}=(k)\in\mathbb{Z}_{\geq 0}$, then $\mf X_{\underline{n}}$ is $P_{k+1}$, the path graph on $k+1$ vertices, and $\mc X_{\underline{n}}$ is $P_{k+2}$, the path graph on $k+2$ vertices. Theorem~\ref{thm:oct_spider_isomorphism} therefore has the following consequences.

\begin{corollary}[\cite{MP17}, Example 62]\label{cor:complete_star_isomorph}
For all $n\in \mb N$, the simplicial complexes $\mc N^{\sq}(\mc B_{K_n})$ and $\mc N(\mc B_{K_{1,n}})$ are isomorphic.
\end{corollary}

\begin{corollary}[\cite{MP17}, Example 62]\label{thm:path_isomorph}
For all $n\in \mb N$, the simplicial complexes $\mc N^{\sq}(\mc B_{P_n})$ and $\mc N(\mc B_{P_{n+1}})$ are isomorphic.
\end{corollary}

\subsection{Interval, spider, and octupus building sets}\label{subsec:interval_spider_octopus_bs}

We now generalize Manneville and Pilaud's results to non-graphical building sets. For such building sets, there are more isomorphisms between extended nested set complexes and nested set complexes. We first introduce an important difference between graphical and non-graphical building sets.
\begin{definition}
For a building set $\calB$, a collection of its elements $\{I_1,\ldots,I_k\}$ is a \textbf{minimal non-nested collection} if $\{I_1,\ldots,I_k\}$ is not a nested collection for $\calB$, but $\{I_i \mid i \in S\}$ is a nested collection for any subset $S \subsetneq [k]$.
\end{definition}

We extend this definition for extended nested set complexes by saying a minimal non-nested collection of elements of $\mc B \cup \{ \textnormal{design vertices of }\mc B\}$ is not an extended nested collection but any proper subset is an extended nested collection.

Using the definition from \cite{MP17}, we will call a pair of building set elements (or design vertices) $e$ and $e'$ \textbf{incompatible} if $\{e,e'\}$ is a minimal non-nested collection, and otherwise we will call them \textbf{compatible}. In the graphical case, all minimal non-nested collections are pairs of incompatible elements by Remark \ref{rmk:graphical_is_flag}. This fact allows Manneville and Pilaud to define and use a ``compatibility degree'' to determine many aspects of graphical nested and extended nestohedron, including characterizing their isomorphisms. In our case, we have some extra structure in minimal non-nested collections with more than $2$ elements. 

\begin{lemma}\label{lem:min_are_disjoint}
    Let $\{I_1,I_2,\ldots,I_k\}$ be a minimal non-nested collection of $\calB$. Then the union $\bigcup_i I_i$ is in $\calB$. Furthermore, $I_1,\ldots,I_k$ are pairwise disjoint except for when $k=2$, in which case $I_1$ and $I_2$ can have non-trivial intersection, i.e., $I_1 \cap I_2$ is non-empty and is neither $I_1$ or $I_2$.
\end{lemma}

\begin{proof}
If $k=2$ and $\{I_1,I_2\}$ is not a nested collection, then $I_1$ and $I_2$ either have non-empty intersection, in which case $I_1 \cup I_2 \in \calB$ by property \ref{item:bs_B2} of a building set, or are disjoint and satisfy $I_1 \cup I_2 \in \calB$. If $k \ge 3$, then since any two-element subset $\{I_i,I_j\}$ is a nested collection, condition \ref{item:nested_N1} of a nested collection is satisfied for $\{I_1,\ldots,I_k\}$. Similarly, since $\{I_i \mid i \in S\}$ is a nested collection for any $S \subsetneq [k]$, condition \ref{item:nested_N2} is satisfied for $\{I_1,\ldots,I_k\}$ except possibly for when $I_1,\ldots,I_k$ are pairwise disjoint and their union is an element of $\calB$. But we know that $\{I_1,\ldots,I_k\}$ is not a nested collection, hence the $I_i$'s are pairwise disjoint and $I_1\cup\cdots\cup I_k$ is an element of $\mc B$.
\end{proof}

Note that a simplicial complex isomorphism is a bijection on the vertices, which induces a bijection between collections of minimal non-faces.

\begin{proposition}
\label{prop:min_bij_equiv_iso}
    $\mc N(\mc B) \simeq \mc N(\mc B')$ if and only if there is a bijection $\Phi: \mc B \rightarrow \mc B'$ which extends to a bijection between minimal non-nested collections. Similarly, $\mc N^\sq(\mc B) \simeq \mc N(\mc B')$ or $\mc N^\sq(\mc B) \simeq \mc N^\sq(\mc B')$ if and only if there is a bijection $\Phi$ between vertices of the simplicial complexes which extends to a bijection between minimal non-faces.
\end{proposition}

\begin{proof}
From Proposition~\ref{prop:def_simplic_iso}, $\Phi:\Delta \rightarrow \Delta'$ is a simplicial complex isomorphism if and only if $\Phi$ extends to a bijection between the subsets of vertices of $\Delta$ and $\Delta'$ which are not faces. Notice that for a simplicial complex, $\sigma$ is not a face implies for every $\tau \supseteq \sigma$, $\tau$ is not a face; thus, $\Phi$ is an isomorphism if and only if $\Phi$ extends to a bijection between minimal non-faces. In the case of $\calN(\calB) \simeq \calN(\calB')$, these are the same as minimal non-nested collections.
\end{proof}

We now provide the first isomorphism for possibly non-graphical building sets. Let $\calB$ be a building set such that every element of a building set $I\in \mc B$ is an interval, i.e., $I=\{a,a+1,\dots,b-1,b\}$ for some $a \le b$. We refer to such building sets as \textbf{interval building sets}.
\begin{theorem}\label{thm:intervals}
Let $\mc B$ be an interval building set on $[n]$. Define 
\[\mc B' := \mc B \cup \{\{n+1\},\{n,n+1\},\{n-1,n,n+1\},\ldots,[n+1]\}.\]
Then $\calB'$ is a building set on $[n+1]$ and $\mc N^{\sq}(\mc B) \simeq \mc N(\mc B')$.
\end{theorem}

\begin{proof}
First we show $\mc B'$ is a building set. Notice that $\mc B'$ contains all singletons and $[n+1]$. In addition, $\mc B'$ contains all intervals which contain $n+1$. Now suppose that $I,J\in\mc B'$ with $n+1$ in at least one of $I$ or $J$, and $I\cap J\neq\varnothing$. The union of two intersecting intervals is an interval, and $n+1\in I\cup J$, so we have that $I\cup J\in \mc B'$. If $I,J\in \mc B'$ are such that $n+1\notin I,J$ and $I\cap J\neq\varnothing$, then this implies that $I,J\in \mc B$, which is a building set. Thus, $I\cup J\in\mc B\subseteq\mc B'$. 

To show $\mc N^{\sq}(\mc B) \simeq \mc N(\mc B')$, we map elements of $\mc B$ to $\mc B \subseteq \mc B'$ by their natural inclusion, and we map $x_i$ to $[i+1,n+1]=\{i+1,\dots,n,n+1\}$ for each $i=1,\dots,n$. Call this map $\phi$.

Now we show that $\sigma$ is a minimal non-nested collection if and only if $\phi(\sigma)$ is.  Since $\mc B = \mc B'|_{[n]}$, the image of a minimal non-nested collection which does not contain a design vertex is a minimal non-nested collection. Similarly the preimage of a minimal non-nested collection which does not contain an element which contains the singleton $n+1$ is a minimal non-nested collection. The minimal non-nested collections which contain design vertices are exactly the sets of the form $\{x_i,e\}$ for a building set element $e$ which contains $i$. Since $[i,n+1] \in \mc B'$ for all $i$ and $\mc B'$ is an interval building set, the minimal non-nested collections which contain $[i+1,n+1]$ are exactly the sets of the form $\{[i+1,n+1],e\}$ where $e$ contains $i$ and does not contain $n+1$. Notice that $[i+1,n+1] = \phi(x_{i})$ and $e = \phi(e)$. It follows by Proposition \ref{prop:min_bij_equiv_iso} that $\phi$ defines a bijection.
\end{proof}

To prove Theorems~\ref{thm:mp_nonextended} and \ref{thm:mp_extended}, Manneville and Pilaud define the \textbf{rotation} as a simplicial automorphism of the dual of the associahedron, $\mc N(\mc B_{P_n})$. Here, we generalize this isomorphism to some other nested set complexes on interval building sets. In the following map, we consider the empty set $\varnothing$ and the entire set $[1,n]$ as elements of every nested collection. This becomes useful in later isomorphisms when we consider reverse inclusion on some nested collections.

\begin{definition}\label{defn:int_rot}
    Let $\mc B$ be an interval building set on $[n]$ such that $[1,k]\in\mc B$ for all $k\in[n]$. The \textbf{interval rotation} of $\mc B$ is the map $\Phi_{\text{rot}}:\mc B\to\mc B'$ defined by 
    \[\Phi_{\text{rot}}(I)=\begin{cases}
    [a-1,b-1],&\quad\text{if $I=[a,b]$ with $1<a\leq b\leq n$,}\\
    [b+1,n],&\quad\text{if $I=[1,b]$ with $b<n$,}\\
    \varnothing,&\quad\text{if $I=[1,n]$,}\\
    [1,n],&\quad\text{if $I=\varnothing$,}
    \end{cases}\]
    with $\mc B'$ the building set on $[n]$ defined by this map.
\end{definition}

\begin{proposition}\label{prop:interval_rotation}
    For a building set $\calB$ satisfying the conditions of Definition \ref{defn:int_rot}, interval rotation defines an isomorphism $\mc N(\mc B) \simeq \mc N(\mc B')$.
\end{proposition}
\begin{proof}
First, we show $\mc B'$ is a building set. Notice that for any singleton $\{a\}=[a,a]\in\mc B$ for $1<a\leq n$, the singleton $\{a-1\}\in \mc B'$, and $[1,n-1]\in\mc B$ implies that $\{n\}\in\mc B'$; thus, all singletons are in $\mc B'$. Now suppose that $[a,c],[b,d]\in\mc B'$ with $[a,c]\cap[b,d]\neq\varnothing$. Then either $a\leq b\leq c\leq d<n$ or $a\leq b\leq c\leq d=n$. If $a\leq b\leq c\leq d<n$, then $[a+1,c+1],[b+1,d+1]\in\mc B$ with $[a+1,c+1]\cap[b+1,d+1]\neq\varnothing$. Since $\mc B$ is a building set, this implies that $[a+1,d+1]\in\mc B$, so $\Phi([a+1,d+1])=[a,d]\in\mc B'$. Notice that by the hypothesis for $\mc B$ and the definition of $\Phi_{\text{rot}}$, we have that $[k,n]\in\mc B'$ for all $k\in[n]$. Thus, if $a\leq b\leq c\leq d=n$, then $[a,d]=[a,n]\in\mc B'$.

Now we show $\sigma$ is a minimal non-nested collection if and only if $\Phi_{\text{rot}}(\sigma)$ is a minimal non-nested collection. Since $\mc B_{[2,n]} \simeq \mc B'|_{[1,n-1]}$, the image of a minimal non-nested collection consisting of elements which do not contain the singleton $1$ is a minimal non-nested collection. Similarly, the pre-image of a minimal non-nested collection consisting of elements which do not contain the singleton $n$ is a minimal non-nested collection. Since $[1,i] \in \mc B$ for all $i$, the minimal non-nested collections which contain $[1,i]$ are exactly the sets of the form $\{[1,i],e \}$ where $1 \not \in e$ and $i+1 \in e$, similarly the minimal non-nested collections of $\mc B'$ which contain $[i+1,n]$ are exactly the sets of the form $\{e',[i+1,n]\}$ where $n \not \in e'$ and $i \in e'$. These sets are in bijection by our mapping $\Phi_{\text{rot}}$.
\end{proof}

We also introduce an analogue of interval rotation for extended nestohedra.
\begin{definition}\label{defn:ext_int_rot}
    Let $\mc B$ be an interval building set on $[n]$ such that $[1,k],[k,n]\in\mc B$ for all $k\in[n]$. The \textbf {extended interval rotation} of $\mc B$ is a map $\Phi^{\sq}_{\text{rot}}:\mc B\cup\{\text{design vertices of }\mc B\}\to\mc B^*\cup\{\text{design vertices of }\mc B^*\}$ defined by
    \[\Phi^{\sq}_{\text{rot}}(I)=\begin{cases}
    [a-1,b-1],&\quad\text{if $I=[a,b]$ and $a>1$,}\\
    x_b,&\quad\text{if $I=[1,b]$,}\\
    [b,n],&\quad\text{if $I=x_b$,}
    \end{cases}\]
    with $\mc B^*$ the building set $[n]$ defined by this map.
\end{definition}
\begin{proposition}\label{prop:ext_interval_rotation}
    With $\calB$ as in Definition \ref{defn:ext_int_rot}, extended interval rotation defines an isomorphism $\mc N^{\square}(\mc B) \simeq \mc N^{\square}(\mc B^*)$
\end{proposition}
\begin{proof}
Let $\tld{\Phi}$ be the isomorphism $\mc N^\sq(\mc B) \rightarrow \mc N(\mc B')$ from the proof of Theorem \ref{thm:intervals}, then 
\[
    \Phi^\sq_{\text{rot}} = \tld{\Phi}^{-1} \circ \Phi_{\text{rot}} \circ \tld{\Phi}
\]
which implies that $\Phi^\sq_{\text{rot}}$ is an isomorphism.
\end{proof}

The following trivial isomorphism, which we call the \textbf{flip} operation, will be involved in later isomorphism constructions. 
\begin{definition}
    Let $w\in\mf S_{n}$ be the permuation $w=(n,n-1,\ldots,2,1)$, i.e., sending $i$ to $n+1-i$. For a building set $\mc B$ on $[n]$, we define its \textbf{flipped building set} $\text{flip} (\mc B)$ to be 
    \[\text{flip}(\mc B)=\{\{w(s_1),\ldots,w(s_r)\}\mid \{s_1,\ldots,s_r\}\in\mc B\}.\]
    We then have the involution \[\flip:\mc B \rightarrow \flip(\mc B),\quad\text{ sending }\quad\{s_1,s_2,\ldots,s_r\} \mapsto \{w(s_1),\ldots,w(s_r) \},\] which extends to a bijection between nested collections. If we let $\flip(x_v)=x_{n+1-v}$ for all $v \in [n]$, then it also extends to a bijection between extended nested collections.
\end{definition}

We now proceed to generalize the notion of spider and octopus graph building sets to the non-graphical case.
\begin{definition}\label{defn:spider_bs}
A \textbf{spider building set} is a building set on $\{v_{i,j}\mid 0<i\leq m, 0<j\leq\ell_i\}$, where $m$ is the number of legs and $\ell_i$ is the length of leg $i$. Define the $i$-th leg $\calB_i$ of $\calB$ to be $\calB_i= \mc B|_{\{v_{i,1},v_{i,2} ,\ldots ,v_{i,\ell_i}\}}$. The building set $\mc B$ satisfies the following:
\begin{enumerate}
\item Each leg $\mc B_i$ of the spider is an interval building set consisting of intervals $[v_{i,a},v_{i,b}]$. Furthermore, $\mc B_i$ contains every interval of the form $[v_{i,1},v_{i,k}]$. Every building set element which intersects but is not contained in $\mc B_i$ contains $v_{i,1}$.
    \item Any union of building set elements containing $v_{i,1}$ for some $i$ is a building set element of $\mc B$.
    \item The restriction of any building set element to leg $i$ is an element of $\mc B_i$.
\end{enumerate}
We refer to any building set element containing a $v_{i,1}$ as a \textbf{body set} and any other element as a \textbf{leg set}.
\end{definition}

\begin{definition}\label{defn:octopus_bs}
An \textbf{octopus building set} is a building set on $\{*\} \cup \{v_{i,j}\mid 0<i\leq m, 0<j\leq\ell_i\}$  where $m$ is the number of legs and $\ell_i$ is the length of leg $i$. We say the leg $i$ of the building set is the set of elements of $\mc B$ consisting of $*,v_{i,1},v_{i,2}, \hdots ,v_{i,\ell_i}$. When referring to a leg, we also refer to $*$ as $v_{i,0}$. The building set $\mc B$ satisfies the following:
\begin{enumerate}
    \item Each leg $\mc B_i$ of the octopus is an interval building set consisting of intervals $[v_{i,a},v_{i,b}]$. $\mc B_i$ contains every interval of the form $[v_{i,k},v_{i,\ell_i}]$ and every interval of the form $[v_{i,0},v_{i,k}]$. Every building set element which intersects but is not contained in $\mc B_i$ contains $v_{i,0}$.
    \item Any union of building set elements containing $*$ is a building set element of $\mc B$.
    \item The restriction of any building set element to a leg $i$ is an element of $\mc B_i$.
\end{enumerate}
We refer to any building set element containing $*$ as a \textbf{body set}, any element containing $v_{i,1}$ and not $*$ as a \textbf{suction cup set}, and any other element as a \textbf{leg set}.
\end{definition}

Both spider and octopus building sets are ways of gluing together interval building sets. In the octopus case we glue together an end point of each of the legs and glue together the corresponding building set elements. In the spider building set we connect the ends of the legs into a larger body and connect the corresponding building set elements. One can check that spider and octopus building sets are indeed valid building sets.

\begin{proposition}\label{prop:spideroctopusvalid}
 All spider and octopus building sets are valid building sets.
\end{proposition}

\begin{remark}
\label{rem:min_non_nested_spi_oct}
The minimal non-nested collections of a spider building set are
\begin{enumerate}[(i)]
    \item Minimal non-nested collections of leg sets in a leg $\mc B_i$ which do not contain $v_{i,0}$,
    \item A body set $D$ and a leg set $E$ in leg $i$ such that $D$ intersected with the leg $i$ is incompatible with $E$, and
    \item Two body sets $D,E$ which intersect but neither contains the other.
\end{enumerate}
The minimal non-nested collections of an octopus building set are
\begin{enumerate}[(i)]
    \item Minimal non-nested collections of leg sets in a leg $\mc B_i$ which do not contain $v_{i,0}$,
    \item A body set $D$ and a leg set $E$ in leg $i$ such that $D$ intersected with the leg $i$ is incompatible with $E$, and
    \item Two body sets $D,E$ which intersect but neither contains the other.
\end{enumerate}
\end{remark}

Let $\mc B$ be a spider building set; we want to construct an octopus building set $\tld{\mc B}$ such that $\mc N^{\square}(\mc B) \simeq \mc N(\tld{\mc B})$ from the legs of the octopus and spider. Recall from Definition~\ref{defn:int_rot} that $\Phi_{\text{rot}}$ denotes the interval rotation.
For a leg $\mc B_i$, let $\overline{\mc B}_i$ be the building set from Theorem~\ref{thm:intervals} such that $\mc N^{\square}(\mc B_i) \simeq \mc N(\overline{\mc B}_i)$. Consider $\flip(\Phi_{\text{rot}}(\overline{\mc B}_i))$, which is the interval rotation of $\overline{\mc B}_i$, followed by a flip.
When we consider the elements of $\flip(\Phi_{\text{rot}}(\overline{\mc B}_i))$ and relabel vertices $v_{i,k} \rightarrow v_{i,k-1}$, denote the resulting building set $\tld{\mc B}_i$; these are the legs of the soon-to-be octopus building set. Gluing our legs at vertices of the form $v_{i,0}$ and relabeling the vertex $*$ yields the octopus building set $\tld{\mc B}$. 

\begin{example}
An example of a spider building set is 
\[\mc B=\mc B_1\cup\mc B_2\cup \mc B_3\cup\mc B_{body},\]
where $\mc B_1,\mc B_2,$ and $\mc B_3$ are given by the following posets ordered by containment:

\begin{figure}[H]
\centering
\subfigure[$\mc B_1$]{
\label{fig:B_1}
\includegraphics[scale=0.7]{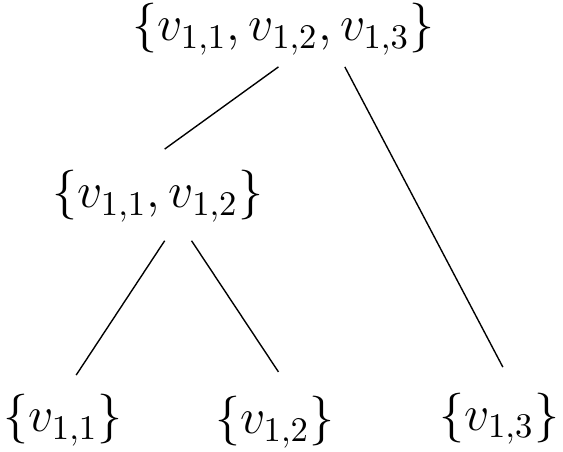}}\qquad\quad
\subfigure[$\mc B_2$]{
\label{fig:B_2}
\includegraphics[scale=0.7]{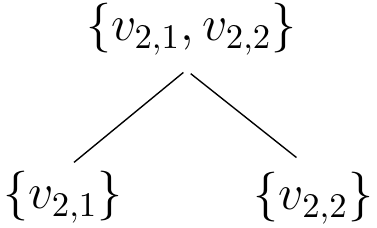}}\qquad\quad
\subfigure[$\mc B_3$]{%
\label{fig:B_3}%
\includegraphics[scale=0.7]{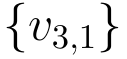}}
\label{fig:B_i_legs_octopus}
\end{figure}
Recall that we refer to any building set element containing a $v_{i,1}$ as a body set, and any other element as a leg set. The remaining body set elements of $\mc B$ are given by
\[\mc B_{body}=\{J_1\cup J_2\cup J_3\mid J_i\in \mc B_i\text{ a body set},\text{ or }J_i=\varnothing\}.\]
One can visualize the spider building set as in Figure~\ref{fig:eg_spider_Bs}; the leg sets are circled in dashes, the body sets that are restricted to a leg are encircled with a solid line, and the edges between $v_{1,1},v_{2,1}$, and $v_{3,1}$ represent the ways that the remaining body sets can be formed.

\begin{figure}[H]
    \centering
    \includegraphics{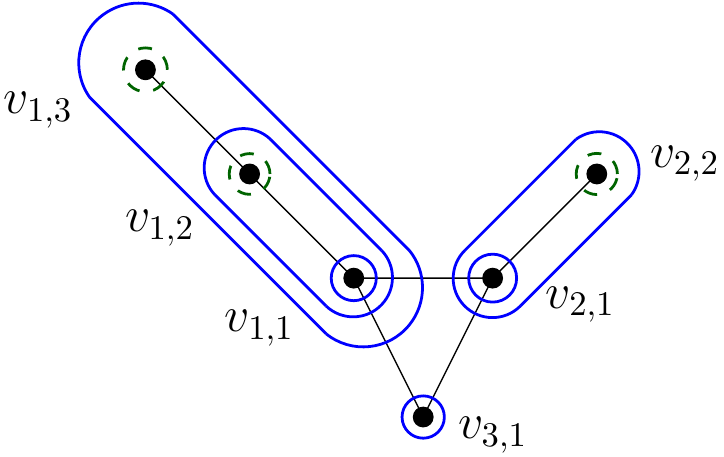}
    \caption{Spider building set $\mc B$.}
    \label{fig:eg_spider_Bs}
\end{figure}

We now construct $\tld{\mc B}_1$, the leg of the octopus corresponding to $\mc B_1$. First, we obtain $\overline{\mc B}_1$, shown in Figure~\ref{fig:overline_B_1}, such that $\mc N^{\sq}(\mc B_i)\simeq\mc N(\overline{\mc B}_i)$ from Theorem~\ref{thm:intervals}. Then, apply the interval rotation, $\Phi_{\text{rot}}(\overline{\mc B}_1)$, which is shown in Figure~\ref{fig:phi_rotation_B_1}. Finally, flipping and relabeling the vertices, we get $\tld{\mc B}_1$, shown in Figure~\ref{fig:tilde_B_1}, which is the leg of the octopus corresponding to $\mc B_1$ of the spider building set.

\begin{figure}[H]
\centering
\subfigure[$\overline{\mc B}_1$]{
\label{fig:overline_B_1}
\includegraphics[scale=0.7]{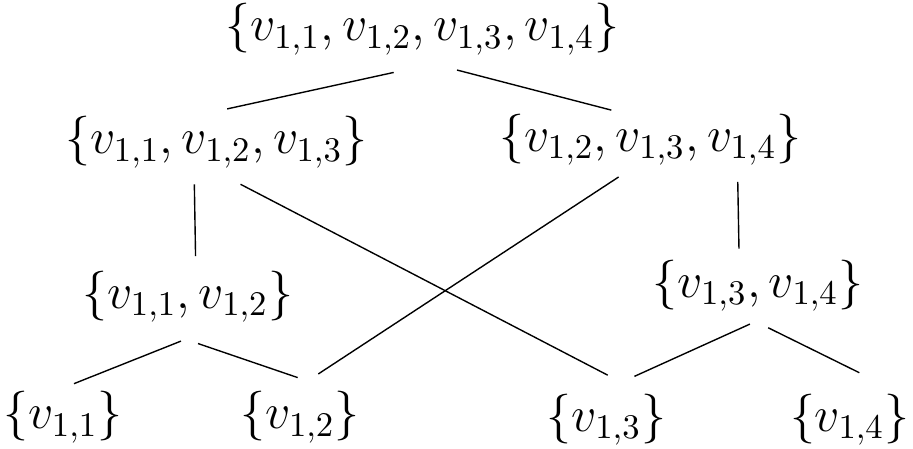}}\qquad\quad
\subfigure[$\Phi_{\text{rot}}(\overline{\mc B}_1)$]{
\label{fig:phi_rotation_B_1}
\includegraphics[scale=0.7]{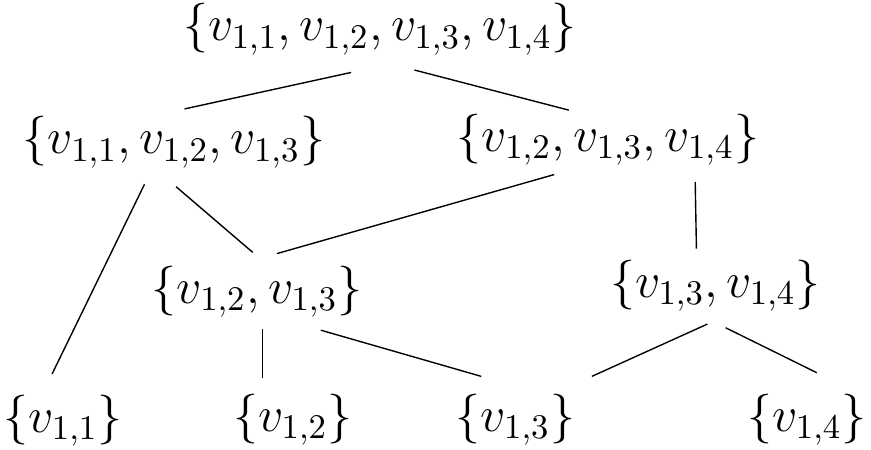}}\qquad\quad
\subfigure[$\tld{\mc B}_1$]{
\label{fig:tilde_B_1}
\includegraphics[scale=0.7]{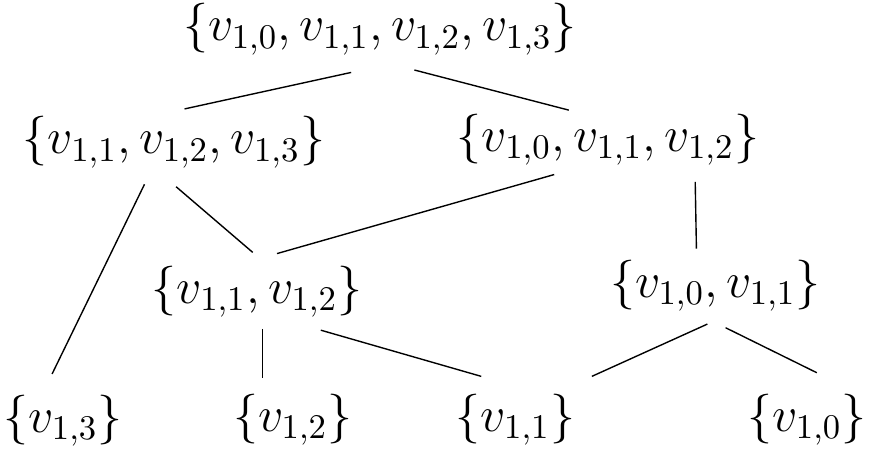}}
\caption{Building sets in the process of constructing $\tld{\mc B}_1$.}
\label{fig:B_1_leg_process}
\end{figure}

Applying this procedure to find the other legs of the octopus, we get
\begin{figure}[H]
\centering
\subfigure[$\tld{\mc B}_2$]{
\label{fig:tld_B_2}
\includegraphics[scale=0.7]{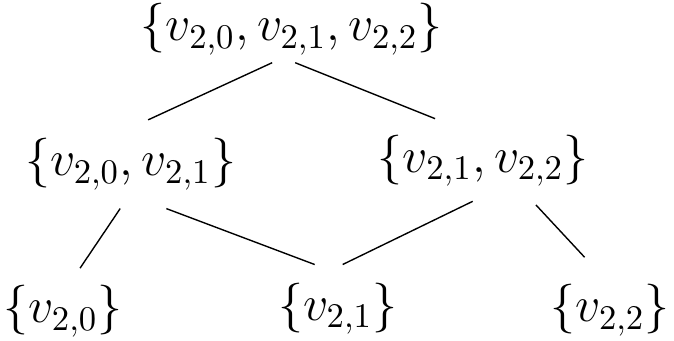}}\qquad\qquad
\subfigure[$\tld{\mc B}_3$]{
\label{fig:tld_B_3}
\includegraphics[scale=0.7]{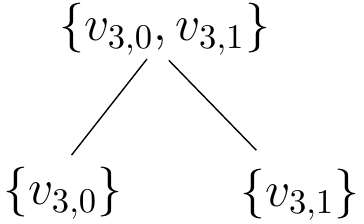}}
\label{fig:tilde_B_2_and_3}
\end{figure}

Gluing the vertices $v_{1,0},v_{2,0},v_{3,0}$ into a single vertex labeled $*$ yields the octopus building set with the following building set elements: 
\begin{align*}
    \textnormal{leg sets: } & \{\{v_{1,3}\},\{v_{1,2}\},\{v_{2,2}\}\}, \\
    \textnormal{suction cup sets: } & \{\{v_{1,1}\},\{v_{1,1},v_{1,2}\},\{v_{1,1},v_{1,2},v_{1,3}\},\{v_{2,1}\},\{v_{2,1},v_{2,2}\},\{v_{3,1}\}\},\\
    \textnormal{body sets: }&\{\{*\}\cup J_1\cup J_2\cup J_3\mid J_i\text{ is a suction cup set in leg $i$, or $J_i=\varnothing$}\}.
\end{align*}

We can visualize this octopus building set as in Figure~\ref{fig:eg_octopus_Bs}. The leg sets are denoted with dashed circles, the suction cup sets are encircled with solid lines, and the edges between $*$ and the vertices $v_{1,1},v_{2,1},$ and $v_{3,1}$ represent the ways that body sets can be formed.

\begin{figure}[H]
    \centering
    \includegraphics{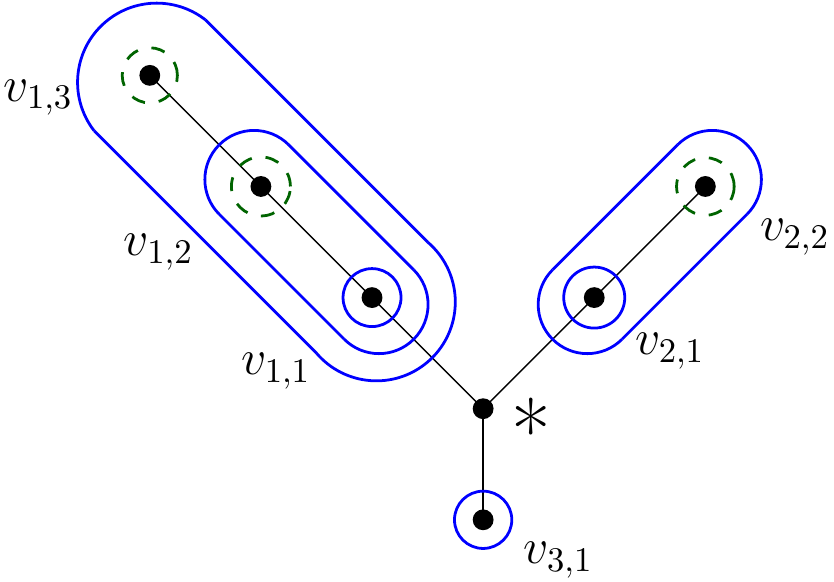}
    \caption{Octopus building set $\tld{\mc B}$.}
    \label{fig:eg_octopus_Bs}
\end{figure}

\end{example}

\begin{theorem}\label{thm:spider_octopus_bs_isomorphism}
    If $\mc B$ is a spider building set and $\tld{\mc B}$ is the corresponding octopus building set, then $\mc N^{\square}(\mc B) \simeq \mc N(\tld{\mc B})$.
\end{theorem}
\begin{proof}
By Theorem~\ref{thm:intervals} and Proposition~\ref{prop:interval_rotation}, the building set $\Phi(\overline{\mc B}_i)$ is a valid interval building set; since $\tld{\mc B}_i$ is obtained from $\Phi(\overline{\mc B}_i)$ by flipping and relabeling vertices and neither operation changes the structure of the building set, we have that $\tld{\mc B}_i$ is a valid interval building set.

By Proposition~\ref{prop:spideroctopusvalid}, the set $\tld{\mc B}$ is a valid building set since the leg $\tld{\mc B}_i$ is an interval building set. For each leg $\mc B_i$, composing the bijections between elements of the building set (and design vertices) and elements of the building set in our chain of isomorphisms gives a bijection $\varphi_i:\mc B_i \cup \{ \textnormal{design vertices of }\mc B_i\} \rightarrow \tld{\mc B}_i$ which extends to a bijection $\{ \textnormal{extended nested collections of }\mc B_i\} \rightarrow \{ \textnormal{nested collections of }\tld{\mc B}_i\}$.
We construct an isomorphism 
$\tld{\Omega}:\mc{B} \cup \{\textnormal{design vertices of }\mc B\} \rightarrow \tld{\mc B}$, where $C|_i$ is the restriction of $C$ to the leg $i$:

\[\tld{\Omega}(C)\coloneqq\begin{cases}
\bigcup_{i\in[m]}\varphi_i(C|_i),&\quad\text{if $C$ is a body building set element,}\\
\varphi_i(C|_i),&\quad\text{if $C$ is a leg building set element in leg $i$,}\\
\varphi_i(C|_i),&\quad\text{if $C$ is a design vertex $x_v$ with $v$ in leg $i$.}
\end{cases}\]

Now we will show $\tld{\Omega}$ extends to a bijection between minimal non-nested collections. For each leg $\mc B_i$, $\varphi_i$ is a composition of simplicial isomorphisms and thus $\tld{\Omega}$ extends to a bijection between minimal non-nested collections between the legs $\mc B_i$ and $\mc B'_i$. It follows that $\tld{\Omega}$ extends to a bijection between the minimal non-nested collection of forms (i) and (ii) for the spider and (i) and (ii) of the octopus from Remark~\ref{rem:min_non_nested_spi_oct}. Since for body sets $D,E$, in each leg $i$, $\varphi_i$ reverses the inclusion of $D$ and $E$, $\tld{\Omega}$ extends to a bijection between the minimal non-nested collection of form (iii) for the spider and (iii) of the octopus from Remark~\ref{rem:min_non_nested_spi_oct}.
\end{proof}

We now find an isomorphism between nested set complexes of spider building sets. Given a spider building set $\mc B$, let $\mc B'$ be the spider building set constructed in the following manner. For each leg $\mc B_i$, let $\mc B_i' = \flip(\Phi(\mc B_i))$, where $\Phi$ is still the interval rotation.

For each leg $\mc B_i$, composing the bijections between elements of the building set gives a bijection $\varphi_i:\mc B_i \rightarrow \mc B_i'$ which extends to a bijection $\{ \textnormal{nested collections of }\mc B_i\} \rightarrow \{ \textnormal{nested collections of }\mc B_i'\}$.
We construct an isomorphism 
$\Omega:\mc{B} \rightarrow \mc B'$ where $C|_i$ is the restriction of $C$ to the leg $i$, by
\[\Omega(C)\coloneqq\begin{cases}
\bigcup_{i\in[m]}\varphi_i(C|_i),&\quad\text{if $C$ is a body building set element,}\\
\varphi_i(C|_i),&\quad\text{if $C$ is a leg building set element in leg $i$.}
\end{cases}\]

\begin{theorem}\label{thm:nontriv_nestohedra_isom}
    The map $\Omega:\mc B\to\mc B'$ defines a non-trivial isomorphism $\mc N(\mc B)\simeq\mc N(\mc B')$.
\end{theorem}
\begin{proof}
As in our previous theorem, $\mc B'$ is a valid building set since each $\mc B'_i$ is an interval building set. \\
Now we will show $\Omega$ extends to a bijection between minimal non-nested collections. For each leg $\mc B_i$, $\varphi_i$ is a composition of simplicial isomorphisms and thus $\Omega$ extends to a bijection between minimal non-nested collections between the legs $\mc B_i$ and $\mc B'_i$. It follows that $\Omega$ extends to a bijection between the minimal non-nested collection of forms (i) and (ii) for the spider for $\mc B$ and $\mc B'$ from Remark~\ref{rem:min_non_nested_spi_oct}. Since for body sets $D,E$, in each leg $i$, $\varphi_i$ reverses the inclusion of $D$ and $E$, $\Omega$ extends to a bijection between the minimal non-nested collection of form (iii) for the spider and (iii) of the octopus from Remark~\ref{rem:min_non_nested_spi_oct}.
\end{proof}

We now find an isomorphism between extended nested set complexes of octopus building sets. For a leg $\mc B_i$, let $\overline{\mc B}_i$ be the building set from Theorem~\ref{thm:intervals} such that $\mc N^{\square}(\mc B_i) \simeq \mc N(\overline{\mc B}_i)$. Then let $\flip(\Phi(\overline{\mc B}_i))$ be the interval rotation and then flip of $\overline{\mc B}_i$. Finally, let $\mc B_i'$ be the building set from Theorem~\ref{thm:intervals} such that $\mc N^{\square}(\mc B_i') \simeq \mc N(\flip(\Phi(\overline{\mc B}_i))$.

For each leg $\mc B_i$, composing the bijections between elements of the building set gives a bijection 
\[\varphi_i:\{\text{vertices of }\mc N^{\sq}(\mc B_i)\}\rightarrow\{\text{vertices of }\mc N^{\sq}(\mc B_i')\},\]
which extends to a bijection $\{ \textnormal{extended nested collections of }\mc B_i\} \rightarrow \{ \textnormal{extended nested collections of }\mc B_i'\}$.
We construct an isomorphism 
\[\Omega^\sq :\{\text{vertices of }\mc N^{\sq}(\mc B)\}\rightarrow\{\text{vertices of }\mc N^{\sq}(\mc B)\},\]
where $C|_i$ is the restriction of $C$ to the leg $i$, by
\[\Omega^{\sq}(C)=\begin{cases}
\bigcup_{i\in m}\varphi_i(C|_i),\quad&\text{if $C$ is a body building set element,}\\
\varphi_i(C|_i),\quad&\text{if $C$ is a suction cup or leg building set element in leg $i$,}\\
\varphi_i(x_v),\quad&\text{if $C=x_v$ with $v\neq 1$ and in leg $i$,}\\
x_* = \varphi_i(x_*),&\text{if $C=x_*$ and for all $i$.}
\end{cases}.\]

\begin{theorem}
\label{thm:square_iso_octopus}
    The map $\Omega^{\sq}$ defines a non-trivial isomorphism $\mc N^{\sq}(\mc B)\simeq\mc N^{\sq}(\mc B')$.
\end{theorem}

The proof that this is a non-trivial isomorphism follows the proof of Theorem~\ref{thm:nontriv_nestohedra_isom} very closely.

\subsection{Independence complex and strong building sets}
We now work towards proving that all non-trivial isomorphisms of extended nested set complexes are, up to relabeling vertices, a rotation or $\Omega^\sq$ from Theorem~\ref{thm:square_iso_octopus}. In Manneville and Pilaud's characterization of isomorphisms between (extended) nested set complexes of graphical building sets, they determined that isomorphisms of products of (extended) nested set complexes of graphical building sets factor along connected components and use this as their main tool \cite[Lemma 78]{MP17}. In this subsection,
we will introduce the concept of \textbf{strongly connected components} within a building set and demonstrate that nested set complexes are determined by their strongly connected components, yielding a natural class of building sets which we call \textbf{strong building sets}. We then show that an isomorphism between products of extended nested set complexes of connected building sets and nested set complexes of strong and connected building sets factor into pairs of smaller isomorphisms. The main tool we will use is the independence complex, which we define soon. 

Throughout this subsection and the next, we assume that the underlying building set $\calB$ of a nested complex $\calN(\calB)$ does not contain a maximal element which is a singleton. This is because such singletons add nothing to the nested complex. In other words, if $\mc B$ contains a maximal element $e$ which is a singleton, then $\mc N(\mc B) \simeq \mc N(\mc B')$, where $\mc B'$ is $\mc B$ with the maximal singleton $e$ removed.

From our earlier discussion of minimal non-nested collections (see Lemma~\ref{lem:min_are_disjoint} and Proposition~\ref{prop:min_bij_equiv_iso}), we have the following corollary.

\begin{corollary}
\label{cor:min_same_component}
Let $\{I_1,I_2,\ldots,I_k\}$ be a minimal non-nested collection of $\calB$. Then $I_1,\ldots,I_k$ are in the same connected component of $\calB$. If we have an isomorphism $\Phi: \mc N(\mc B) \rightarrow \mc N(\mc B')$, then $\Phi(I_1),\Phi(I_2), \ldots ,\Phi(I_k)$ are in the same connected component of $\calB'$.
\end{corollary}

\begin{proof}
By Lemma~\ref{lem:min_are_disjoint}, $I=I_1 \cup \cdots\cup I_k \in \calB$, hence $I_1,\ldots,I_k$ are in the same connected component as $e$. If $\Phi: \calN(\calB) \to \calN(\calB')$ is an isomorphism, then $\{\Phi(I_1),\ldots,\Phi(I_k)\}$ is a minimal non-nested collection of $\calB'$, hence similar reasoning shows that they are also in the same connected component.
\end{proof}

We now introduce the independence complex, which captures the structure of the minimal non-nested collection of $\calB$.

\begin{definition}\label{def:independence_complex}
For a simplicial complex $\Delta$, we define its \textbf{independence complex} $\calI(\Delta)$ to be the simplicial complex with the minimal non-faces of $\Delta$ as its facets. We define the \textbf{independence graph} $\calG(\Delta)$ to be the $1$-skeleton of $\calI(\Delta)$. For a connected component $C$ of the independence graph $\calG(\Delta)$, we define the \textbf{M-size} of $C$ to be the dimension of the subcomplex $\Delta$ restricted to the vertices in $C$.

When $\Delta=\calN(\calB)$ is a nested complex, we define $\calG(\calB):=\calG(\calN(\calB))$ to be the \textbf{independence graph} of $\calB$. A \textbf{strongly connected component} of $\calB$ is a connected component of $\calG(\calB)$.
\end{definition}

We now show that there is a correspondence between isomorphisms of simplicial complexes and isomorphisms of their independence complexes. Note that $\calI(\Delta)$ has the same vertices as $\Delta$, hence a map $\Phi: \{\text{vertices of }\Delta\} \to \{\text{vertices of }\Delta'\}$ could be considered as both a map of vertices $\Phi: \Delta \to \Delta'$ and $\Phi: \calI(\Delta) \to \calI(\Delta')$.

\begin{proposition}\label{prop:isomorphism_preserves_M_size}
Let $\Phi$ be a map between the vertices of $\Delta$ and those of $\Delta'$. Then $\Phi:\Delta \to \Delta'$ is an isomorphism if and only if $\Phi: \calI(\Delta) \to \calI(\Delta')$ is an isomorphism. In this case, restricting to the $1$-skeleton gives an isomorphism $\Phi_1: \calG(\Delta) \to \calG(\Delta')$, which sends connected components of $\calG(\Delta)$ to connected components of $\calG(\Delta')$ and preserves $M$-size.

\end{proposition}
\begin{proof}
Since a simplicial map is an isomorphism if and only if it is a bijection of the facets, $\Phi: \calI(\Delta) \to \calI(\Delta')$ is an isomorphism if and only if $\Phi$ is a bijection from the facets of $\calI(\Delta)$ to the facets of $\calI(\Delta')$. But this is the same as saying $\Phi$ is a bijection between minimal non-faces of $\Delta$ and $\Delta'$. By an observation in Proposition~\ref{prop:min_bij_equiv_iso}, this is the case if and only if $\Phi: \Delta \to \Delta'$ is a simplicial complex isomorphism.

If $\Phi: \calI(\Delta) \to \calI(\Delta')$ is an isomorphism, then its restriction on the $1$-skeletons is also an isomorphism, and a graph isomorphism sends connected components to connected components. This mapping preserves M-size since for any connected component $C$ of $\calG(\Delta)$, we can restrict $\Phi$ to an isomorphism between $\Delta$ restricted to $C$ and $\Delta'$ restricted to $\Phi(C)$. In particular, the dimensions of these complexes are the same.
\end{proof}

Since the minimal non-faces of the join of simplicial complexes are the union of the sets of minimal non-faces for both complexes, we have the following Proposition.

\begin{proposition}\label{prop:independence_graph_disjoint}
For simplicial complexes $\Delta$ and $\Delta'$,
\[
    \calI(\Delta*\Delta') = \calI(\Delta) \sqcup \calI(\Delta'),\quad\text{and hence}\quad\calG(\Delta*\Delta')=\calG(\Delta) \sqcup \calG(\Delta').
\]
Furthermore, for a simplicial complex $\Delta$ with connected components $C_i$ of the $\calG(\Delta)$, $\Delta$ is equal to the join over $i$ of $\Delta$ restricted to the vertices of $C_i$.
\end{proposition}

\begin{remark}
By Corollary~\ref{cor:min_same_component}, any connected component of $\calG(\calB)$ must be a subset of a connected component of $\calB$. We call $\calB$ a \textbf{strong building set} if all connected components of $\calB$ are strongly connected components. We define the \textbf{size} of a strongly connected component to be the number of minimal elements in the strongly connected component.
\end{remark}

An important class of strong building sets are graphical building sets. A proof of the following statement for graphical building sets can be found in \cite[Lemma 78]{MP17}. More importantly, we show that any building set is equivalent to a strong building set, in the sense that their nested complexes are isomorphic. To do so, we need the following results.

\begin{lemma}\label{lem:minimal_non_nested_union_equals_e}
    For every non-singleton building set element $e$, there exists a minimal non-nested collection $e_1,e_2, \ldots, e_r$ such that $\bigcup_i e_i = e$
\end{lemma}
\begin{proof}
let $e_1,e_2, \ldots, e_s$ be the maximal element of $\mc B|_e \setminus \{e\}$. Since $\{i\} \in\mc B|_e \setminus \{e\}$ for all singletons $i \in e$, $i \in e_{j_i}$ for some $j_i$ and thus the union of $e_1,\dots,e_s$ equals $e$. Pick some minimal subset of them with union $e$; we can reindex so that they are $e_1,e_2, \ldots, e_r$. Since this subset is minimal and $e_1,e_2, \ldots, e_r$ are maximal in $\mc B|_e \setminus \{e\}$, no union of a subset of their elements can be an element of $\mc B|_e \setminus \{e\}$ and thus $e_1,e_2, \ldots, e_r$ is a minimal non-nested collection.
\end{proof}

\begin{lemma}
\label{lem:un_of_strong_components}
Let $\{C_i\mid i \in S\}$ be the set of strongly connected components of $\mc B$. Define $M_i = \bigcup_{e \in C_i} e$ and $M = \{M_i \mid i \in S\}$. Then we have the  following.
\begin{enumerate}[(i)]
    \item $M_i$ is an element of $\mc B$.
    \item Every building set element $e$ contained in $M_i$ is contained in a strongly connected component $C_j$ with $M_j \subseteq M_i$.
    \item Every building set element $e$ which intersects $M_i$ but is not contained in $M_i$ contains $M_i$.
\end{enumerate}
\end{lemma}

\begin{proof}
\begin{enumerate}
    \item[(i)] Choose any spanning tree in the connected component of $\calG(\calB)$, and for each edge $e_i=\{v_i,v_i'\}$ of the tree, choose a minimal non-nested collection $E_i$ that contains $v$ and $v'$. By Lemma~\ref{lem:min_are_disjoint}, the union of elements in $E_i$ is an element $u_i$ of the building set. Since we select a spanning tree, each $u_i$ has non-empty intersection with some other $u_j$. We then take the union of all $u_i$'s, which is in $\calB$ by property \ref{item:bs_B2}. However, this also equals the union of all elements in a strongly connected component, so we are done.

    \item[(ii)] We prove (ii) and (iii) simultaneously using an interdependent induction. We prove (ii) for an element $M_i \in M$ such that for every $M_j \in M$ satisfying $M_j \subsetneq M_i$, we have already proven (iii). We also need to prove a base case of (ii) where $M_i$ is a minimal element of $M$.
    
    We first prove the base case. Suppose for the sake of contradiction that there exists a building set element strictly contained in $M_i$ and not contained in the corresponding strongly connected component $C_i$. Then there exists some minimal building set element $J$ strictly contained in $M_i$ which is not contained in $C_i$. 
    
    First suppose $J$ is not a singleton. By Lemma~\ref{lem:minimal_non_nested_union_equals_e}, there exists a set $I_1,I_2,\ldots,I_r$ of building set elements whose union is $J$ and any strict subset of elements of $I_1,I_2,\ldots, I_r$ is a nested collection. By our assumption that $J$ is minimal, every element of $I_1,I_2,\ldots,I_r$ is contained in $C_i$. Thus, there exists some strict subset $R$ of $[r]$ and $K_i\in\mc B$ such that $\{I_{i}\mid i \in R\}\cup\{K_1,K_2,\ldots, K_r\}$ is a minimal non-nested collection whose union $U \subset M_i$ contains a singleton which is not in $J$.

    If $U$ contains $J$, then $U = J \cup K_1 \cup K_2\cup\cdots\cup K_r$ implies that (noting that any subset of $K_i$'s is a nested collection) there exists a subset of the $K_i$'s which form with $J$ a minimal non-nested collection. Thus $J$ is in the same strongly connected component as the $I_i$'s, which is a contradiction.

    If $U$ does not contain $J$, then $U,J$ intersect each other and are not nested, hence $U$ and $J$ are in the same strongly connected component. Then $U \cup J = U \cup \{I_i\mid i \not \in S\} \in \calB$ implies that a subset of the $I_i$'s form with $U$ a minimal non-nested collection. Thus $J$ is in the same strongly connected component as the $I_i$'s, which is again a contradiction.

    Now suppose $J$ is a singleton, then there exists a set $\{J,I_1,I_2,\ldots,I_n\}$ which is a minimal non-nested collection and $U=J\cup I_1\cup\cdots\cup I_n$ is not contained in $M_i$. Then for any building set element $J'$ with $J \subseteq J' \subsetneq M_i$, we have that $U \cup J' \in \mc B$, and $U \cup J' = J' \cup I_1 \cup I_2\cup\cdots\cup I_n.$ Since any subset of the $I_i$'s form a nested collection, there exists a subset of the $I_i$'s which, together with $J'$, forms a minimal non-nested collection. Thus, no element of $C_i$ contains $J$, which is a contradiction since the union of all elements in $C_i$ is $M_i$ which contains $J$.

    We conclude that (ii) is true for any minimal $M_i \in M$.
    
    Now we show that (ii) holds for $M_i$ if (iii) holds for every $M_j \in M$ satisfying $M_j \subsetneq M_i$. Consider the building set
    \[
    \mc B^* := \{I \in \mc B\mid  I \supseteq m'\textnormal{ for all } M_j \in M \textnormal{ satisfying }M_j \subsetneq M_i\}.
    \]
    By (iii), for all $M_j \subsetneq M_i$, the minimal non-nested collections of $\mc B^*$ are exactly the minimal non-nested collections of $\mc B$ whose edges in $\calG(\calB)$ are not contained in a connected component corresponding to an $M_j \in M$ satisfying $M_j \subsetneq M_i$. If we define $\{C_i^* \mid i \in S^*\}$ to be the set of strongly connected components of $\calB^*$, and $M^*=\{\bigcup_{e \in C_i^*}e\mid i \in S^*\}$, then the above sentence shows that $M_i$ is a minimal element of $M^*$. By the base case above applied to the building set $\calB^*$, we have that (ii) holds for $M_i$ in $\calB^*$. But since $\calB^* \subset \calB$, we conclude that (ii) holds for $M_i$ in $\calB$ as well.

\item[(iii)] Now we prove (iii) for any $M_j$ which we have already proven (ii) for.
     Suppose that an element $J \in \mc B$ does not contain nor is contained by $M_j$, so that $J \cup M_j \in \mc B$. By Lemma~\ref{lem:minimal_non_nested_union_equals_e}, there exists a minimal non-nested collection $I_1,I_2,\ldots,I_k$ of $C_j$ with $I_1 \cup I_2\cup\cdots\cup I_k = M_j$. We have $J \cup M_j = J \cup I_1 \cup I_2\cup\cdots\cup I_k \in \calB$. Since the $I_i$'s are contained in $M_j$ it is enough to show $J$ is in the same strongly connected component as the $I_i$'s.
    Since $J$ intersects $M_j$, $J$ must intersect some $I_i$. If $J$ does not contain this $I_i$, then $J$ is incompatible with $I_i$ (since $J$ contains an element outside of $M_j$, and thus outside of $I_i$) and thus $J$ is in the same strongly connected component as $I_i$. If $J$ contains $I_i$, then $J$ union some proper subset of the $I_i$'s is $J \cup M_j$. 
Since any union of a proper subset of the $I_i$'s form a nested collection, this implies there exists a subset of the $I_i$'s which, together with $J$, forms a minimal non-nested collection. Thus $J$ is in the same connected component as the $I_i$'s. 
\end{enumerate}
\end{proof}

\begin{corollary}
\label{cor:strong_contained_con}
Every strongly connected component of $\mc B$ is contained in a connected component of $\mc B$.
\end{corollary}
\begin{proof}
By part (i) of Lemma~\ref{lem:un_of_strong_components}, the union of all elements of a strongly connected component $C$ is in $\mc B$ and thus in some connected component of $\mc B$. Then all elements in $C$ are contained in that connected component of $\mc B$.
\end{proof}

\begin{proposition}
\label{prop:strong_natural}
 For every building set $\mc B$, there exists a strong building set $\mc B'$ such that $\mc N(\mc B) \simeq \mc N(\mc B')$ and $\mc B$ and $\mc B'$ have the same strongly connected size partition.
\end{proposition}
\begin{proof}
By Corollary \ref{cor:strong_contained_con}, it is enough to prove the statement for connected building sets $\mc B$.
Define $\mc B*J$ to be $\{I \in \mc B\mid I \supset J\}$. By Lemma~\ref{lem:un_of_strong_components} (iii), for any $m \in M$ (as defined in the lemma),
\[
    \mc N(\mc B) \simeq \mc N(\mc B|_m \sqcup (\mc B*m)),
\]
and the strongly connected component sizes of $\mc B|_M \sqcup (\mc B*M)$ are the same as the strongly connected component sizes of $\mc B$ under the natural bijection between strongly connected components.

Now let $\mc B_0 = \mc B$. We inductively pull strongly connected components from $\mc B_i$ of the form $\mc B_i|_{m_i}$, where $m_i$ is a minimal element of the set of unions of the elements of strongly connected components for $\mc B_i$ and form $\mc B_{i+1} = \mc B*{m_i}$. Repeating this process separates $\mc B$ into a strong building set
\[
    \mc B_0|_{m_0} \sqcup \mc B_1|_{m_1}\sqcup\cdots\sqcup \mc B_r|_{m_r},
\]
and 
\[
    \mc N(\mc B) \simeq \mc N(\mc B_0|_{m_0} \sqcup \mc B_1|_{m_1}\sqcup\cdots\sqcup \mc B_r|_{m_r}).
\]
\end{proof}

Notice that in Proposition~\ref{prop:strong_natural} and Lemma~\ref{lem:un_of_strong_components}, each $\mc N(\mc B_i|_{m_i})$ is isomorphic to the subcomplex of $\mc N(\mc B)$ restricted to a connected component $C$ of $\calG(\calB)$ of size $\abs{\{\textnormal{minimal elements of }C\}}$. Each $\mc B_i*m_i$ preserves the strongly connected component sizes and thus we have,

\begin{corollary}
The M-size of a connected component $C$ of $\calG(\calB)$ is $\abs{\{\textnormal{minimal elements of }C\}} - 1$.
\end{corollary}

We have a similar characterization for the M-size of an extended nested complex $\calN^\sq(\calB)$.

\begin{lemma}\label{lem:connected_components_of_indep_graph_extended}
    The connected components of $\calG(\mc N^\sq(\mc B))$ are the connected components of $\mc B$. The $M$-size of a connected component of $\calG(\mc N^\sq(\mc B))$ is the size of the corresponding connected component of $\mc B$ (i.e. the number of singletons). 
\end{lemma}
\begin{proof}
Let $C$ be a connected component of $\mc B$ with maximal element $M_C \in \mc B$. Any design vertex $x_v$ for $v \in C$ has an edge to $M_C$ in $\calG(\mc N^\sq(\mc B))$. Any building set element containing $v$ is contained in $C$ and has an edge to $x_v$ in $\calG(\mc N^\sq(\mc B))$. Thus, every two elements $e_1,e_2$ in $C$ are connected. Notice that for any minimal non-nested collections $\{e_1,e_2,\ldots,e_n\}$ of building set elements, the union $\bigcup_{i=1}^n e_i$ is an element of $\mc B$. Then, there is no edge from a building set element to a building set element in another connected component in $\calG(\mc B)$.

Since the minimal non-nested collections which contain a design vertex $x_v$ consist only of incompatible pairs $\{x_v,e\}$ with $v \in e$, it must be that $e$ is contained in the connected component which contains $v$. Thus, there are no edges between design vertices in a connected component and vertices in a different connected component. The connected components of $\calG(\mc N^\sq(\mc B))$ are the connected components of $\mc B$. For each of these connected components $C$, the subcomplex of $\mc N^\sq(\mc B))$ restricted to the vertices of $C$ is $\mc N^\sq(\mc B|_{M_C}))$, which has dimension equal to the size of $C$.
\end{proof}

We now prove that an isomorphism of (extended) nested complexes factors into isomorphisms between (extended) nested complexes of (strongly) connected components.

\begin{proposition}
\label{prop:product_iso_factors}
For indexing sets $A,B,C,D$ and for each $a \in A, b \in B, c \in C, d \in D$, let $\mc B_{a}, \mc B_{c}$ be connected building sets and $\mc B_{b},\mc B_{d}$ be strong and connected building sets. An isomorphism (where the product is the $*$-product)
\[
    \prod_{a \in A} \mc N^\sq(\mc B_{a})  * \prod_{b \in B}\mc N(\mc B_{b}) \simeq \prod_{c\in C} \mc N^\sq(\mc B_{c})  * \prod_{d \in D}\mc N(\mc B_d)
\]
factors into isomorphisms between single extended or non-extended set complexes on the LHS ($\mc N^\sq(\mc B_{a})$ or $\mc N(\mc B_{b})$) and RHS ($\mc N^\sq(\mc B_{c})$ or $\mc N(\mc B_d)$), which have equal M-sizes.
\end{proposition}

\begin{proof}
By Proposition~\ref{prop:independence_graph_disjoint}, the independence graph of LHS and RHS is the disjoint union of the graphs of each term in the join. By Lemma~\ref{lem:connected_components_of_indep_graph_extended} and the fact that $\calB_b,\calB_d$ are strong and connected, the independence graph of each of the term in the join above is connected. Thus, the connected components of the LHS and RHS are precisely the terms of the join. By Proposition~\ref{prop:isomorphism_preserves_M_size}, this isomorphism factors into isomorphisms between connected components of both sides, which are precisely isomorphisms of the form \[\calN^\epsilon(\calB_u) \simeq \calN^\delta(\calB_v),\] where $\epsilon,\delta \in \{\varnothing,\sq\}$, and $u,v$ are in the corresponding index sets according to $\epsilon,\delta$. Since these are isomorphisms, both sides have equal M-size.
\end{proof}

Applying the above Proposition to the case where both sides only has one (extended) nested complex yields the following.

\begin{corollary}\label{cor:same_con_components} Let $\calB$ be a connected building set on $[n]$.
\begin{enumerate}[(i)]
    \item If $\calN^\sq(\calB) \simeq \calN^\sq(\calB')$, then $\calB'$ is a connected building set on $[n]$.
    \item If $\calN^\sq(\calB) \simeq \calN(\calB')$, then $\calB'$ is a strong and connected building set on $[n+1]$.
    \item If $\calN(\calB) \simeq \calN(\calB')$ and $\calB$ is strong, then $\calB'$ is a strong and connected building set on $[n]$.
\end{enumerate}

\end{corollary}

When we consider links of vertices in our (extended) nested set complexes, by the link decomposition (proved in section 2), we have an isomorphism of the form
\[
    \prod_{a\in A} \mc N^\sq(\mc B_a)  * \prod_{b\in B}\mc N(\mc B_b) \simeq \prod_{c\in C} \mc N^\sq(\mc B_c)  * \prod_{d\in D}\mc N(\mc B_d),
\]
where $\mc B_b,\mc B_d$ might not be strong or connected. In this case, we can factor each of $\calN(\calB_b),\calN(\calB_d)$ into a join of its strongly connected components, and use Proposition~\ref{prop:product_iso_factors} to pair up these connected components on both sides into isomorphisms. We can thus bound the size of the image of an element in an isomorphism as follows.

\begin{lemma}
\label{lem:ext_size_bounds}
Let $\calB$ be a connected building set on $[n]$ with $e \in \mc B$.
\begin{enumerate}[(i)]
    \item If there is an isomorphism $\Phi^\sq: \mc N^\sq(\mc B) \rightarrow \mc N^\sq(\mc B')$ with $\Phi^\sq(e)$ not a design vertex of $\calB'$, then $\abs{\Phi^\sq(e)}\le \abs{e}$ or $\abs{\Phi^\sq(e)} \ge n-\abs{e}+1$.
    \item If there is an isomorphism $\tld{\Phi}: \mc N^\sq(\mc B) \rightarrow \mc N(\mc B')$, then $\abs{\tld{\Phi}(e)} \le \abs{e}$ or $\abs{\tld{\Phi}(e)} \ge n-\abs{e}+1$.
\end{enumerate}
\end{lemma}
\begin{proof}
\begin{enumerate}[(i)]
    \item Considering the links of $e$ and $\Phi^\sq(e)$ gives us an isomorphism \[\mc N^\sq(\mc B/e) * \mc N(\mc B|_e) \simeq \mc N^\sq(\mc B'/\Phi^\sq(e)) * \mc N(\mc B|_{\Phi^\sq(e)}).\] By the observation above, $\calN^\sq(\calB/e)$ is isomorphic to either $\calN^\sq(\calB'/\Phi^\sq(e))$ or a strongly connected component of $\calN(\calB'|_{\Phi^\sq(e)})$. Comparing their maximal M-size components gives us either $\abs{\Phi^\sq(e)} = \abs{e}$ or $\abs{\Phi^\sq(e)} \geq n-\abs{e}+1$.
    \item We have a similar isomorphism \[\calN^\sq(\calB/e)*\calN(\calB|_e) \simeq \calN(\calB'/\tld{\Phi}(e))*\calN(\calB'|_{\tld{\Phi}(e)}).\] Thus $\calN^\sq(\calB/e)$ is isomorphic to either a strongly connected component of $\calN(\calB'/\tld{\Phi}(e))$ or a strongly connected component of $\calN(\calB'|_{\tld{\Phi}(e)})$. This gives us $\abs{\tld{\Phi}(e)} \le \abs{e}$ or $\abs{\tld{\Phi}(e)} \geq n-\abs{e}+1$.
\end{enumerate}
\end{proof}

\begin{remark}\label{rem:singleton_to_singleton_or_complement}
This lemma implies that in such an isomorphism (either $\Phi^\sq$ or $\tld{\Phi}$), any singleton $s \in \mc B$ is either sent to a design vertex, a singleton or a complement of a singleton. We say a singleton $s \in \mc B$ is \textbf{maintained} by an isomorphism $\Phi$ if $|\Phi(s)| = 1$ and $s$ is \textbf{swapped} if $|\Phi(s)| = n$.
\end{remark}

\subsection{Isomorphisms between extended nested set complexes}
In this subsection, we will prove the following theorems.

\begin{theorem}\label{thm:ext_to_non_ext_isom_summary}
For a strong and connected building set $\mc B$, every isomorphism $\tld{\Phi}:\mc N^{\sq}(\mc B)\to\mc N(\mc B')$ is either the isomorphism from Theorem~\ref{thm:spider_octopus_bs_isomorphism} or the isomorphism from Theorem~\ref{thm:intervals}.
\end{theorem}

\begin{theorem}\label{thm:ext_to_ext_isom_summary}
For a connected building set $\mc B$, every non-trivial isomorphism $\Phi^{\sq}:\mc N^{\sq}(\mc B)\to\mc N^{\sq}(\mc B')$ is an extended rotation or the isomorphism from Theorem~\ref{thm:square_iso_octopus}.
\end{theorem}

We will use the results from the previous subsection on the factoring of isomorphisms and the possible images of the singletons. Our main focus will be to study the images of the design vertices, whose rigid structure allows us to determine to where most elements of the building set are mapped. We will also frequently use the results of the following Lemma.

\begin{lemma}\label{lem:nested_collection_contains_all_max_elements} Let $\calB$ be a building set.
\begin{enumerate}[(i)]
    \item Two elements $e,e' \in \calB$ are incompatible if and only if neither contain each other and their union is in $\calB$.
    \item If $\{I_1,\dots,I_k\}$ is a nested collection with $I=I_1 \cup\dots\cup I_k$, then $\{I_1,\dots,I_k\}$ contains all the maximal elements of $\calB|_I$.
\end{enumerate}
\end{lemma}

\begin{proof}
\begin{enumerate}[(i)]
    \item Two building set elements $e,e'$ are incompatible if and only if they intersect each other but neither is contained in the other, or they are disjoint but their union is in $\calB$. In both cases, we have that $e,e'$ don't contain each other, and in the first case $e \cup e' \in \calB$ as well since their intersection is non-empty.
    \item Note that the maximal elements of $\calB|_I$ are disjoint, and they form the connected components of $\calB|_I$. Thus, any $I_j$ is contained in one such element. Now, assume that $\{I_1,\dots,I_k\}$ doesn't contain some $M \in (\calB|_I)_{\max}$. Let $I_{j_1},\dots,I_{j_l}$ be the elements that are contained in $M$. Then their union must be $M$, or otherwise the union $I_1 \cup \dots \cup I_k$ is not all of $I$, contradiction. But then using the same argument as in Lemma~\ref{lem:minimal_non_nested_union_equals_e}, a subset of $\{I_{j_1},\dots,I_{j_l}\}$ will form a minimal non-nested collection. This contradicts the fact that $\{I_1,\dots,I_k\}$ is a nested collection, hence any subset is a nested collection.
\end{enumerate}
\end{proof}

\subsubsection{Proof of Theorem~\ref{thm:ext_to_non_ext_isom_summary}}

We first study the images of the design vertices and the singletons of $\calB$.

\begin{lemma}
\label{lem:image_of_design_and_singletons}
    For a connected building set $\mc B$ on $[n]$ and an isomorphism $\tld{\Phi}: \mc N^\sq(\mc B) \rightarrow \mc N(\mc B')$, let $\{L_i \mid i \in S\}$ be the maximal elements of the nested collection $\{\tld{\Phi}(x_s) \mid s \in [n]\}$, where $S$ is some indexing set. Then
    \begin{enumerate}[(i)]
        \item all images of design vertices are contained in exactly one $L_i$,
        \item for every singleton $s$ with $\tld{\Phi}(x_s) \subsetneq L_i$, $\tld{\Phi}(s)$ is a singleton in $L_i$
        \item all images of design vertices contained in $L_i$ are nested, and
        \item if $\tld{\Phi}(x_s) = L_i$, then $\tld{\Phi}(s)$ is either the unique singleton $\mc B' \setminus \bigcup_{i\in S} L_i$ or it is the complement of the unique singleton of the form $L_i - \tld{\Phi}(x_t)$.
    \end{enumerate}
\end{lemma}

\begin{proof}
\begin{enumerate}[label=(\roman*)]
    \item The image of any design vertex must be contained in some $L_i$ by definition. It is contained in exactly one $L_i$ since the images of design vertices form a nested collection, hence no two elements intersect without one containing the other.
    \item By Corollary \ref{cor:same_con_components}, $\mc B'$ is strong and connected on $[n+1]$, and $\left|\tld{\Phi}(s)\right| \in \{1,n\}$ by Remark~\ref{rem:singleton_to_singleton_or_complement}. If $\tld{\Phi}(s)$ is swapped, then it must not contain $L_i$ since its design vertex is contained in $L_i$. But then it is incompatible with $L_i$ which is the image of another design vertex since $\tld{\Phi}(x_s) \subsetneq L_i$, contradiction. Now let $\tld{\Phi}(s)$ be a singleton contained in $L_j$. Since the image of the design vertices which map to elements contained in $L_j$ form a maximal nested collection in $L_j$, $\tld{\Phi}(s)$ union the images of these design vertices is not a nested collection. This implies that $\{s\}$ union these design vertices is not a nested collection, and this only happens when $x_s$ is one of these design vertices. We then conclude that $j=i$ since $\tld{\Phi}(x_s) \subsetneq L_i$.
    \item We prove this by induction. Since the image of the design vertices is a maximal nested collection, there is some design vertex $x_s$ which maps to a singleton in $L_i$. Now suppose there is a design vertex $x_t$ which maps to an element of size $r < |L_i|$ inside $L_i$. The corresponding vertex $t$ is mapped to a singleton which implies that the union of these $\tld{\Phi}(x_t) \cup \tld{\Phi}(t)$ is in $\mc B'$. Let $\tld{\Phi}(e)$ be this union, where $e$ is either a building set element or a design vertex of $\mc B$. Since the images of $x_t$ and $[n]$ do not form a nested collection, $\tld{\Phi}(x_t) \cup \tld{\Phi}([n]) \in \mc B'$. Then  $\tld{\Phi}([n]) \cup \tld{\Phi}(e) = \left(\tld{\Phi}([n])\cup \tld{\Phi}(x_t)\right) \cup \tld{\Phi}(e) \in \mc B'$ as well, which implies that $[n]$ and $e$ are incompatible. Thus, $e$ is a design vertex whose image is of size $|\tld{\Phi}(x_t) \cup \tld{\Phi}(t)| = r+1$.
    \item Suppose $\left|\tld{\Phi}(s)\right| = 1$. By our argument in (ii), $\tld{\Phi}(s)$ cannot be contained in any $L_j$ for $j \neq i$. Then $\tld{\Phi}(s)$ cannot be contained in $L_i$ since $\tld{\Phi}(s)$ is incompatible with $\tld{\Phi}(x_s) = L_i$. Since the set of design vertices form a maximal nested collection, there is exactly one singleton not contained in $\cup_{i \in S} L_i$, which $s$ is mapped to.

Suppose $\left|\tld{\Phi}(s)\right| = n$. By our arguments in (iii), there exists a $t'$ such that $\tld{\Phi}(t') = L_i \setminus s'$ for some singleton $s'$. Thus $\tld{\Phi}(s)$ must contain $L_i \setminus s'$ and be incompatible with $L_i$. It follows that $\tld{\Phi}(s)$ is the complement of $s'$.
\end{enumerate}
\end{proof}

To prove Theorem~\ref{thm:ext_to_non_ext_isom_summary}, we first use the fact that $\mc B$ is strong and connected to conclude that $\abs{\tld{\Phi}([n])} \in \{1,n\}$. This is because considering the link of $[n]$ and $\tld{\Phi}([n])$ gives us the isomorphism \[\calN(\calB) \simeq \calN(\calB'|_{\tld{\Phi}([n])}) * \calN(\calB'/\tld{\Phi}([n]))\] Now Proposition~\ref{prop:product_iso_factors} implies that one of the terms in the RHS is trivial, hence $\tld{\Phi}([n])$ is either a singleton or the complement of a singleton. Therefore, to finish the proof it suffices to consider the two cases separately. They are dealt with in the following theorems.

\begin{theorem}
\label{thm:gen_of_85}
     For a connected building set $\mc B$ on $[n]$ and an isomorphism $\tld{\Phi}: \mc N^\sq(\mc B) \rightarrow \mc N(\mc B')$, assume $\abs{\tld{\Phi}([n])}=n$. Then $\mc B'$ is isomorphic to an interval building set which contains $[i,n+1]$ for all $i$, and $\mc B$ is an interval building set such that $\tld{\Phi}$ is isomorphic to the isomorphism in Theorem~\ref{thm:intervals}.
\end{theorem}
\begin{proof}
Since $\mc B$ is connected, Proposition~\ref{prop:product_iso_factors} implies that $\mc B'$ is strong and connected on $[n+1]$. We relabel so that $\tld{\Phi}([n])=[n]$. Since $\tld{\Phi}([n])$ is incompatible with the image of any design vertex, $\tld{\Phi}(x_s)$ must contain $n+1$ for all design vertex $x_s$. By Lemma~\ref{lem:image_of_design_and_singletons}, it follows that the images of all design vertices are nested.
We then re-label the singletons of $\mc B'$ and $\mc B$ such that $\tld{\Phi}(x_{i}) =[i+1,n+1]$ for all $i=1,\ldots,n$. Also by Lemma~\ref{lem:image_of_design_and_singletons}, for $i \geq 2$, $\tld{\Phi}(i) = \{i\}$ and $\tld{\Phi}(1)$ is either $\{1\}$ or $\{1\} \cup [3,n+1]$.

We now prove that $\mc B'$ is an interval building set. First note that the elements of $\calB'$ which contain $n+1$ are exactly the images of design vertices, since any such element is incompatible with $\tld{\Phi}([n])$, and thus its pre-image must be some design vertex. In particular, this rules out the case $\tld{\Phi}(1)=\{1\} \cup [3,n+1]$, hence $\tld{\Phi}(1)=\{1\}$. Also, if there was some element $J \in \mc B'$ which contains $i$ and $i-k>1$ but not $i-1$, then $[i,n+1] \cup J \in \mc B'$ contains $n+1$ but cannot be the image of a design vertex since all such images are intervals, a contradiction.
    
Finally, to show that $\tld{\Phi}$ is the isomorphism from Theorem~\ref{thm:intervals}, it suffices to show that $\tld{\Phi}^{-1}([a,b])=[a,b]$ for any interval $[a,b] \in \mc B'$ with $1\leq a \leq b \leq n$. This follows from two observations. Since $\{[a,b],[i,n+1]\}$ is a nested collection for $i<a+1$ or $ i>b+1$, $\tld{\Phi}^{-1}([a,b])$ is compatible with $x_i$ for $i<a$ or $i>b$. This implies $\tld{\Phi}^{-1}([a,b]) \subseteq [a,b]$. Similarly, since $\{[a,b],[i,n+1]\}$ is a not a nested collection for $a+1 \leq i \leq b+1$, $\tld{\Phi}^{-1}([a,b])$ is incompatible with $x_i$ for $a \le i \le b$, thus $\tld{\Phi}^{-1}[a,b] \supseteq [a,b]$. Together these observations imply $\tld{\Phi}^{-1}([a,b])=[a,b]$.
\end{proof}

\begin{theorem}
\label{thm:square_to_reg_phi(n)=1}
For a connected building set $\mc B$ on $[n]$ and an isomorphism $\tld{\Phi}: \mc N^\sq(\mc B) \rightarrow \mc N(\mc B')$, if $|\tld{\Phi}([n])| = 1$ then under relabeling vertices $\tld{\Phi}$ is the isomorphism from Theorem \ref{thm:spider_octopus_bs_isomorphism}.
\end{theorem}

\begin{proof}
Let $L_1',\dots,L_k'$ be the maximal elements among the images of the design vertices. Since the elements $L_1',\dots,L_k'$ are incompatible with $\{*\}=\tld{\Phi}([n])$, we have that $* \not\in L_i'$ for any $i$. By part (ii) of Lemma~\ref{lem:nested_collection_contains_all_max_elements}, $L_1',\dots,L_k'$ are the maximal elements of $\calB|_{L'}$, where $L'=L_1' \cup \dots \cup L_k'$. Let $\abs{L_i'}=n_i$ for all $i$. By Lemma 4.42, we can label the vertices of $\calB$ and $\calB'$ so that  $L_i'=\{v_{i,1},\dots,v_{i,n_i+1}\}$ for all $i$, $\calB$ is a building set on $L_1 \sqcup \dots \sqcup L_k$ with $L_i=\{v_{i,0},\dots,v_{i,n_i}\}$, and these vertices satisfy for each $i$,
\begin{enumerate}[(i)]
    \item $\tld{\Phi}(x_{i,n_i-j})=[v_{i,1},v_{i,j+1}]$ for $j=0,\dots,n_i$,
    \item $\tld{\Phi}(v_{i,n_i-j})=v_{i,j+2}$ for $j=0,\dots,n_i-1$, and
    \item $\tld{\Phi}(v_{i,0})=\{*\} \cup L' \setminus \{v_{i,n_i+1}\}$.
\end{enumerate}

We also identify $*$ with $v_{i,0}$ for any $i=1,\dots,k$. For the rest of the proof, we will proceed by steps as follows.
\begin{enumerate}[(\alph*)]
    \item Our first step is to show that $\calB'|_{L_i'}$ is an interval building set for each leg $L_i'$. It suffices to show that there does not exist any non-interval building set element that contains $v_{i,1}$. This is because if $e' \in \calB'|_{L_i}$ is not an interval, then there exists $1 \le j < l \le n_i+1$ such that $v_{i,j} \in e'$, $v_{i,j+1} \not\in e'$ and $v_{i,l} \in e'$. Then $[v_{i,1},v_{i,j}] \cup e'$ is a non-interval building set element that contains $v_{i,1}$, contradicting our assumption. Now consider $e' \in \calB'|_{L_i'}$ containing $v_{i,1}$ but is not an interval. Since all design vertices have been mapped to, $\tld{\Phi}^{-1}(e') \in \calB$ and thus is compatible with $[n]$. This implies $e'$ is compatible with $\{*\}$, hence their union $e' \cup \{*\} \not\in \calB'$. However, $e' \cup \{*\}=e' \cup (\{*\} \cup \{v_{i,1}\}) \in \calB'$, where $\{*\} \cup \{v_{i,1}\} \in \calB'$ because they are the images of $[n]$ and $x_{i,n_i}$ which are incompatible. This is a contradiction, thus $\calB'|_{L_i'}$ is an interval building set for each $i$.
    
    \item Next, we will show that for any interval $[v_{i,a},v_{i,b}]$ with $0<a<b$, \[\tld{\Phi}^{-1}([v_{i,a},v_{i,b}])=[v_{i,n_i-a+2},v_{i,n_i-b+2}].\] First, note that for any $i' \ne i$, $\tld{\Phi}(x_{i',j}) \subset L_{i'}'$ and $[v_{i,a},v_{i,b}] \subset L_i'$ are in two different connected components of $\calB'|_{L'}$, hence they are compatible and thus the preimage of $[v_{i,a},v_{i,b}]$ doesn't contain $v_{i',j}$ for any $i' \ne i$. In other words, the preimage is contained in $L_i'$ of $\calB$. Now, we have that $[v_{i,a},v_{i,b}]$ and $[v_{i,1},v_{i,t}]=\tld{\Phi}(x_{i,n_i-t+1})$ are incompatible for $a-1 \le t \le b-1$, since neither contain each other or that in the case of $t=a-1$, they are disjoint but their union is an element of $\calB'$. Similarly, they are compatible for $t \ge b$ since they are nested. Furthermore, they are compatible for $t \le a-2$ since their union is not an interval. Thus, $\tld{\Phi}^{-1}([v_{i,a},v_{i,b}])$ contains exactly $v_{i,t}$ for $n_i-a+2 \le t \le n_i-b+2$, which is what we want. 
    
    \item We now determine the preimage of elements $e' \in \calB'$ containing $\{*\}$ whose restriction to each leg $L_i'$ is an interval. Using the fact that $[n]$ is incompatible with any design vertex, we have that $\tld{\Phi}([n])=\{*\}$ is incompatible with $[v_{i,1},v_{i,j}]$ for any $i=1,\dots,k$, $j=1,\dots,n_i+1$. Thus, $\{*\} \cup [v_{i,1},v_{i,j}] \in \calB'$ for all such $i,j$. Now consider $S \subset [k]$ and $1 \le j_i \le n_i$ for all $i \in S$. Since $\{*\} \cup [v_{i,1},v_{i,j_i}] \in \calB'$ for all $i \in S$ and each of them contains $\{*\}$, their union is \[e_{\{j_i \mid i \in S\}}=\{*\} \cup \bigcup_{i \in S} [v_{i,1},v_{i,j_i}] \in \calB'.\] We will prove that \[\tld{\Phi}^{-1}(e_{\{j_i \mid i \in S\}})=\bigcup_{i \in S}[v_{i,1},v_{i,n_i-j_i}]\cup\bigcup_{i' \not\in S}L_{i'},\] where $[v_{i,1},v_{i,n_i-j}]=\varnothing$ if $j=n_i+1$. Indeed, note that $e_{\{j_i \mid i \in S\}}$ is incompatible with $\tld{\Phi}(x_{i',l})$ for any $i' \not\in S$, since their union is $e_{\{j_i \mid i \in S\} \cup \{n_i'-l\}} \in \calB'$. Thus, $\tld{\Phi}^{-1}(e_{\{j_i \mid i \in S\}})$ contains $v_{i',l}$ for any $i' \not\in S$. For each $i \in S$, note that $e_{\{j_i \mid i \in S\}}$ is compatible with $[v_{i,1},v_{i,j}]=\Phi(x_{i,n_i-j})$ for $j \le j_i$ and incompatible with $j \le j_i+1$, hence $\tld{\Phi}^{-1}(e_{\{j_i \mid i \in S\}})$ contains precisely $[v_{i,1},v_{n_i-j_i}]$ for each leg $L_i'$ with $i \in S$. This finishes our claim.
    
    \item Finally, we show that if $e' \in \calB'$ contains $\{*\}$, then $e' \cap L_i'$ is an interval of the form $[v_{i,0},\dots,v_{i,j_i}]$ for all $i$. For each leg $L_i'$, let $j_i=\max\{j \in [0,n_i+1] \mid [v_{i,1},v_{i,j}] \subset e'\}$. We then see that $e'$ and $\tld{\Phi}(x_{i,n_i-t+1})$ are compatible for $t \le j_i$ since they are nested. Furthermore, they are incompatible for $t \ge j_i+1$ since they are not nested and their union $e' \cup [v_{i,0},v_{i,t}]=e' \cup (\{*\} \cup [v_{i,0},v_{i,t}]) \in \calB'$. Thus, $\tld{\Phi}^{-1}(e')=\bigcup_{i=1}^k[v_{i,0},v_{i,n_i-j_i}]$, but this is the image of an element considered in (c), hence $e'=\tld{\Phi}(\tld{\Phi}^{-1}(e'))$ is of the form in (c), which finishes our claim.
\end{enumerate}

We now explain how the above parts complete the proof. Any $e' \in \calB'$ either contains $\{*\}$ or doesn't contain $\{*\}$. If $* \not\in e'$, then $e'$ is contained in one of the legs $L_i$, which by (a) is an interval building set. If $* \in e'$, then (d) shows $e'$ must be of the form considered in (c). These imply that $\calB'$ is an octopus building set. Parts (b) and (c) then shows that $\tld{\Phi}^{-1}$ acts the same way as the inverse of the isomorphism from Theorem~\ref{thm:spider_octopus_bs_isomorphism}, and thus $\calB$ is a spider building set. In other words, $\tld{\Phi}$ is the same as the isomorphism from Theorem~\ref{thm:spider_octopus_bs_isomorphism}.
\end{proof}

\begin{remark}\label{rem:counter_example_tld_Phi}
When $\mc B$ is not strong, it is possible to have an isomorphism $\tld{\Phi}: \mc N^\sq(\mc B) \rightarrow \mc N(\mc B')$ where $\abs{\tld{\Phi}([n])} \not\in \{1,n\}$. Take \[\mc B=\{\{1\},\{2\},\{3\},\{1,2\},\{1,2,3\}\}\] and \[\mc B'=\{\{1\},\{2\},\{3\},\{4\},\{1,3\},\{3,4\},\{1,3,4\},\{2,3,4\},\{1,2,3,4\}\}.\] Define $\tld{\Phi}$ to be the map sending \[x_1 \mapsto \{2,3,4\}, x_2 \mapsto \{3,4\}, x_3 \mapsto \{4\}\] and \[ \{1\} \mapsto \{1,3,4\}, \{2\} \mapsto \{2\}, \{3\} \mapsto \{3\}, \{1,2\} \mapsto \{1\}, \{1,2,3\} \mapsto \{1,3\}.\] We then check that $\tld{\Phi}: \calN^\sq(\calB) \to \calN(\calB')$ is an isomorphism with $\left|\tld{\Phi}([3])\right|=2$.
\end{remark}

\begin{remark}\label{rem:counter_example_Phi}
We might hope that there is a similar characterization of nested set complex isomorphism $\Phi: \calN(\calB) \to \calN(\calB')$ for a strong and connected building set $\calB$, namely that $\Phi$ is either the rotation isomorphism of Definition~\ref{defn:int_rot} or the isomorphism in Theorem~\ref{thm:nontriv_nestohedra_isom}. However, we can produce a counter-example from Remark~\ref{rem:counter_example_tld_Phi}. Namely, take $\calB$ and $\calB'$ as above, and let \[\calB''=\{\{1\},\{2\},\{3\},\{1,2\},\{1,2,3\},\{4\},\{3,4\},\{2,3,4\},\{1,2,3,4\}\}.\] Let $\tld{\Phi}':\calN^\sq(\calB) \to \calN(\calB'')$ be the isomorphism from Theorem~\ref{thm:intervals}. Then we can check that $\calB',\calB''$ are strong building sets but \[\tld{\Phi}' \circ \tld{\Phi}^{-1}: \calN(\calB') \to \calN(\calB'')\] is neither of the desired isomorphisms.
\end{remark}

\subsubsection{Proof of Theorem~\ref{thm:ext_to_ext_isom_summary}}

We now have the necessary results on isomorphisms between nested set complexes and between extended and non-extended nested set complexes to characterize isomorphisms between extended nested set complexes. We begin in the same way as we began studying $\tld{\Phi}$, namely by considering the images of design vertices and singletons as in Lemma~\ref{lem:image_of_design_and_singletons}. 

\begin{lemma}
\label{lem:square_image_of_design_and_singletons}
    For a connected building set $\mc B$ on $[n]$ and an isomorphism $\Phi^\sq: \mc N^\sq(\mc B) \rightarrow \mc N^\sq(\mc B')$, let $L_1,\dots,L_k$ be the maximal  (building set) elements of the extended nested collection $\{\Phi^\sq(x_s) \mid s \in [n]\}$. Then for each $L_i$, 
    \begin{enumerate}[(i)]
        \item all images of design vertices are either design vertices in $[n] \setminus \bigcup_{i=1}^k L_i$ or contained in exactly one $L_i$, 
        \item for every singleton $s$ with $\Phi^\sq(x_s) \subsetneq L_i$, $\Phi^\sq(s)$ is a singleton in $L_i$,
        \item all images of design vertices contained in $L_i$ are nested, and
        \item if $\Phi^\sq(x_s) = L_i \in \mc B'$, then $\Phi^\sq(s)$ is the design vertex of the unique singleton of the form $L_i - \Phi^\sq(x_t)$.
    \end{enumerate}
\end{lemma}
\begin{proof}
\begin{enumerate}[(i)]
    \item If $\Phi^\sq(x_s)$ is a design vertex $x_t$, then $t$ cannot be contained in any $L_i$ since otherwise $\Phi^\sq(x_s)$ is incompatible with $L_i$. Now consider the design vertices that map to a building set element. Their images form a nested collection, thus each of them is contained in some maximal element $L_i$. It is contained in exactly one $L_i$ since no two elements of a nested collection can intersect without one containing the other.
    \item By Corollary~\ref{cor:same_con_components}, $\mc B'$ is connected on $[n]$, and by Remark~\ref{rem:singleton_to_singleton_or_complement}, $\Phi^\sq(s)$ is either $[n]$, a design vertex, or a singleton. Note that $\Phi^\sq(s)$ cannot be $[n]$ since then it would be compatible with $\Phi^\sq(x_s)$. If $\Phi^\sq(s)$ is a design vertex, then it must be incompatible with $\Phi^\sq(x_s)$ and thus incompatible with $L_i \supset \Phi^\sq(x_s)$ which is the image of another design vertex, hence $\Phi^\sq(s)$ is not a design vertex. Now let $\Phi^\sq(s)$ be a singleton. It must be contained in $L_j$ for some $j$ since otherwise its design vertex is of the form $\Phi^\sq(x_t)$ for some $t \ne s$. This implies $s$ and $x_t$ are incompatible, a contradiction. If $j \ne i$, then $\Phi^\sq(s)$ union the images of all design vertices contained in $L_j$ is a nested collection, since none of these design vertices are $x_s$. This contradicts the fact that the images of these design vertices form a maximal nested collection in $L_j$. Thus, $\Phi^\sq(s)$ is a singleton in $L_i$.
    \item If $\Phi^\sq([n])$ is not a design vertex, the same argument as in Lemma~\ref{lem:image_of_design_and_singletons} implies (iii) holds. If $\Phi^\sq([n])$ is a design vertex, then the images of all design vertices in $L_i$ contain the corresponding singleton of $\Phi^\sq([n])$, and since the image of the design vertices form a nested collection all design vertices are nested.
    \item From part (ii), $\Phi^\sq(s)$ is either a singleton or a design vertex. Suppose $\Phi^\sq(s)$ was a singleton. By the same argument as in (ii), $\Phi^\sq(s)$ cannot be contained in any $L_j$ for $j \neq i$. Notice that $\Phi^\sq(s)$ cannot be contained in $L_i$ since $\Phi^\sq(s)$ is incompatible with $\Phi^\sq(x_s) = L_i$. Thus $\Phi^\sq(s)$ is a design vertex. $\Phi^\sq(s)$'s corresponding singleton must be contained in $L_i$ and not in the image of any other design vertex thus it's corresponding singleton is of the form $L_i - \Phi^\sq(x_t)$.
\end{enumerate}
\end{proof}

Next, we will prove a lemma about restricting the isomorphism $\Phi^\sq$ to each leg $L_i$.

\begin{lemma}
\label{lem:iso_restrict_square}
    For an isomorphism $\Phi^\sq: \mc N^\sq(\mc B) \rightarrow \mc N^\sq(\mc B')$ with $\Phi^\sq(x_v)$ as a maximal element among images of design vertices in $\mc B'$, let $S$ be set of singletons $s \in \mc B$ such that $\Phi^\sq(x_s) \subseteq \Phi^\sq(x_v)$. Then $\Phi^\sq$ restricts to an isomorphism
    \[
    \Phi^\sq|_S:\mc N^\sq(\mc B|_S) \rightarrow \mc N^\sq\left(\mc B'|_{\Phi^\sq(x_v)}\right),
    \]
    which, under relabeling singletons, is the extended rotation isomorphism of Proposition~\ref{prop:ext_interval_rotation}, and for any $e' \in \mc B$ which intersects but is not contained in $S$, $v \in e'$ and $e'$ intersected with $S$ is an interval building set. 
\end{lemma}
\begin{proof}
Let $\abs{S}=k$. Using Lemma~\ref{lem:square_image_of_design_and_singletons}, we may relabel the singletons of $\mc B$ and $\mc B'$ such that the singletons contained in $S$ and $\Phi^\sq(x_v)$ are labeled with $1,2,\ldots,k$ and
\[
\Phi^{\sq}(I)=\begin{cases}
[i,k],&\text{if $I=x_i$,}\\
x_1,&\text{if $I=1$,}\\
i-1,&\text{if $I=i$ for $i>1$.}
\end{cases}
\]

First we will show for any isomorphism $\Phi$ and any set $S$ as defined in the lemma, for all $e \in S$, $\Phi^\sq(e) \subseteq \mc B'|_{[k]}$. Considering the inverse will imply that $\Phi^{\sq-1}(e) \subseteq S$ since $\Phi^{\sq-1}(e)$ is also an isomorphism and by Lemma 4.47 and our earlier argument that $\Phi^{\sq-1}([1,i]) = x_i$, we know that $[k]$ is a valid building set element in $S$ (its design vertices are nested and the singleton $k$ maps to a design vertex).

Consider $\Phi^\sq$ acting on the link of $x_1$. The link of $x_1$ is of the form \[\calN^\sq(\calB|_{T_1})*\dots*\calN^\sq(\calB|_{T_l}),\] where $T_1,\dots,T_l$ are the maximal elements of $\calB|_{[n]\setminus 1}$, and the link of $\Phi^\sq(x_1)=[k]$ is the complex $\mc N^\square(\mc B'/[k]) * \mc N\left(\mc B'|_{[k]}\right)$. Proposition~\ref{prop:product_iso_factors} then implies that $\calN(\calB'|_{[k]})$ is isomorphic to the join of a subset of $\calN^\sq(\calB|_{T_j})$'s; upon reindexing, we may assume that \[\calN^\sq(\calB|_{T_1}) * \dots * \calN^\sq(\calB|_{T_p}) \simeq \calN\left(\calB'|_{[k]}\right).\]
The images of the design vertices in $S \setminus \{1\}$ are contained in $\mc N\left(\mc B'|_{[k]}\right)$, hence the connected components containing $S \setminus \{1\}$ must be contained in $T_1 \cup \dots \cup T_p$ hence $T_1 \cup \dots \cup T_p \supseteq S \setminus \{1\}$. The dimension of 
    $\mc N^\sq\left(\mc B|_{S \setminus \{1\}}\right)$
equals $k-1=\abs{S\setminus \{1\}}$ which is the dimension of $\mc N\left(\mc B'|_{[k]}\right)$, so equality must occur. This implies for all $e \subseteq S \setminus \{1\}$, $\Phi(e) \subseteq [k]$. It also implies the second to last statement in our Proposition, since if $e' \in \mc B$ does not contain $1$, then it is compatible with $x_1$ and thus lies in either $S \setminus \{1\}$ or the complement of $S$.

Now we will show that any building set element $e \in \mc B$ intersected with $S$ is an interval building set under our labeling. This will imply that the only elements $e \subseteq S$ which contain $1$ are $[1,i]$ which we have already shown map to design vertices.

A design vertex $x_i$, $i \in [k]$ is not compatible with $\Phi^\sq(x_j)$ if and only if $j \in \{1,2,\ldots,i\}$. It follows that the preimage of $x_i$ is $[1,i]$ under this labeling.
Let $e \in \mc B$ not contain a singleton $r$ of $S$ and contain a singleton $s$ with $s<r$ in our labeling of $S$.
If $\Phi^\sq(e)$ is design vertex $x_v$, since $x_v$ would need to be incompatible with $\Phi^\sq(x_1) = [k]$, thus $v \in [k]$, but all such design vertices are already the images of $[1,i]$. If $\Phi^\sq(e)$ is not a design vertex, $\Phi^\sq(e)$ is incompatible with $\Phi^\sq(x_s)$, thus $\Phi^\sq(e)$ cannot contain $\Phi^\sq(x_s) \subset \Phi^\sq(x_r)$ and
\[
\Phi^\sq(e) \cup \Phi^\sq(x_s) \in \mc B'.
\]
Since $\Phi^\sq(x_r) \subseteq \Phi^\sq(x_1)$,
\[
\Phi^\sq(e) \cup \Phi^\sq(x_r) \in \mc B'.
\]
Since $\Phi^\sq(x_r)$ and $\Phi^\sq(e)$ are compatible, $\Phi^\sq(e) \subseteq \Phi^\sq(x_r)$. But then for singletons $a<s$ in $S$, $\Phi^\sq(e) \subseteq \Phi^\sq(x_a)$, and so $a \not \in e$. 

Finally, we show that $\Phi^\sq$ acts like extended rotation on $S$. We have already described how $\Phi^\sq$ acts on any interval containing $1$ and in particular the image of any interval $[a,b]$ with $a>1$ is not a design vertex. Any interval $[a,b]$ with $a>1$ in $S$ is incompatible with $[1,i]$ if and only if $a-1 \leq i \le b-1$, it follows that a singleton $i \in [k]$ is in $a$ if and only if $a-1\leq i+1 \le b-1$. In other words, $\Phi^\sq([a,b])=[a-1,b-1]$.
\end{proof}

We now have the following corollaries.

\begin{corollary}
    For a connected building set $\mc B$ on $[n]$ and an isomorphism $\Phi^\sq: \mc N^\sq(\mc B) \rightarrow \mc N^\sq(\mc B')$, if $\Phi^\sq([n])$ is a design vertex, then $\Phi^\sq$ is an extended rotation.
\end{corollary}

\begin{proof}
From part (iii) of Lemma~\ref{lem:square_image_of_design_and_singletons}, all design vertices map to building set elements of $\mc B'$, and they only have $1$ maximal element. Thus, applying Lemma~\ref{lem:iso_restrict_square} with $S=[n]$ gives the desired result.
\end{proof}

\begin{corollary}
\label{cor:no_des_rotation}
    For a connected building set $\mc B$ on $[n]$ and an isomorphism $\Phi^\sq: \mc N^\sq(\mc B) \rightarrow \mc N^\sq(\mc B')$, if no design vertices maps to a design vertex, then $\Phi^\sq$ is an extended rotation.
\end{corollary}
\begin{proof}
    The image of $[n]$ must be compatible with all design vertices and thus must be a design vertex.
\end{proof}

This next lemma provides sufficient conditions on an isomorphism for it to be the isomorphism from Theorem~\ref{thm:square_iso_octopus}.

\begin{lemma}
\label{lem:one_des_ocotpus}
 For a connected building set $\mc B$ on $[n]$ with $n>1$ and an isomorphism $\Phi^\sq: \mc N^\sq(\mc B) \rightarrow \mc N^\sq(\mc B')$, if exactly one design vertex maps to a design vertex, then $\Phi^\sq$ is the isomorphism from Theorem~\ref{thm:square_iso_octopus}.
\end{lemma}

\begin{proof}
Suppose $x_i$ is the design vertex which maps to a design vertex, which we can relabel to be $x_i$ itself. Then the corresponding singleton $i$ cannot map to a design vertex since it is incompatible with $x_i$. It follows that $i$ maps to either $i$ itself or $[n]$.

First consider the case where $\Phi^\sq(i)$ is the singleton. Then $\Phi^\sq$ acting on the link of $i$ defines an isomorphism
\[
    \Phi^\sq/i: \mc N^\sq(\mc B/\{i\}) \rightarrow \mc N^\sq(\mc B'/\{i\}).
\]
Notice that $\Phi^\sq/i$ sends no design vertex to another design vertex, so by Corollary~\ref{cor:no_des_rotation}, it is an extended rotation. This implies that $\Phi^\sq([n])$ is a design vertex, but then $\Phi^\sq([n])$ is compatible with $\Phi^\sq(x_i)$, which implies no such isomorphism $\Phi^\sq$ exists.

Now consider the case where $\Phi^\sq(i) = [n]$.
Then $\Phi^\sq$ acting on the link of $x_i$ defines an isomorphism
\[
    \Phi^\sq|_{[n]\setminus \{i\}}: \mc N^\sq(\mc B|_{[n]\setminus \{i\}}) \rightarrow \mc N^\sq(\mc B'|_{[n]\setminus \{i\}}),
\]
where no design vertex is mapped to a design vertex. Using Proposition~\ref{prop:product_iso_factors}, we can factor this isomorphism into isomorphisms between each of the connected components of $\calB|_{[n]\setminus \{i\}}$ and the corresponding connected component of $\calB'|_{[n]\setminus i}$. By Corollary~\ref{cor:no_des_rotation}, each of these isomorphisms is an extended rotation. Thus we may relabel $\mc B$ and $\mc B'$ such that $i$ is now $*$, the connected components on $[n]\setminus \{i\}$ are the legs $L_1,\dots,L_k$ with size $n_1,\dots,n_k$ respectively. Furthermore, since the isomorphism between each leg is an extended rotation, we can relabel the vertices of each leg $L_i=\{v_{i,0},\dots,v_{i,n_i}\}$ such that every building set element contained in $L_i$ is an interval $[v_{i,a},v_{i,b}]$ and 
\[\Phi^{\sq}(I)=\begin{cases}
x_{v_{i,n_i-b}},&\text{if $I=[v_{i,1},v_{i,b}]$,}\\
[v_{i,1},v_{i,n_i+1-a}],&\text{if $I=x_{i,a}$,}\\
[v_{i,n_i+2-b},v_{i,n_i+2-a}],&\text{if $I=[v_{i,a},v_{i,b}]$ with $a,b>0$.}
\end{cases}\]

It remains to show that
\begin{enumerate}[(i)]
    \item the set $ \bigcup_{i=1}^k [v_{i,1},v_{i,j_i}]\cup\{*\} \in \mc B$ where $-1 \le j_i \le n_i$ for all $i=1,\dots,k$,
    \item the map $\Phi^\sq$ acts as the isomorphism from Theorem~\ref{thm:square_iso_octopus} on all building set elements which do contain $*$, and
    \item no elements not of the above form containing $*$ are in $\mc B$,
\end{enumerate}
We now tackle each of these claims.
\begin{enumerate}[(i)]
    \item The isomorphism $\Phi^{\sq-1}$ sends design vertices to $[v_{i,1},v_{i,j_i}]$ and $[n]$ to $*$, and since $*$ is incompatible with any design vertex, we have $\{*\} \cup [v_{i,1},v_{i,j_i}] \in \mc B$ for any $i, j_i$. Since for any two building elements which intersect, their union is in the building set, this implies $\bigcup_{i=1}^k [v_{i,0},v_{i,j_i}]\cup\{*\} \in \mc B$ for any set of $j_i$'s.
    \item Notice that the set $\bigcup_{i=1}^k [v_{i,0},v_{i,j_i}]\cup\{*\} \in \mc B$ is compatible with the interval $[v_{i,1},v_{i,b}]$ if and only if $b \leq j_i$. Since the intervals $[v_{i,1},v_{i,a}]$ map to design vertex $x_{v_{i,n_i-b}}$, this implies that $\Phi^\sq\left(\bigcup_{i=1}^k [v_{i,1},v_{i,j_i}]\cup\{*\}\right)$ contains $v_{i,n_i-b}$ if and only if $b > j_i$.
    \item This follows from the last statement in Lemma \ref{lem:iso_restrict_square}.
\end{enumerate}
\end{proof}

We are now able to prove our main result.

\begin{proof}[Proof of Theorem~\ref{thm:ext_to_ext_isom_summary}]
Let $S$ be the set of design vertices which map to design vertices.
Corollary~\ref{cor:no_des_rotation} and Lemma~\ref{lem:one_des_ocotpus} consider the cases when $\abs{S} \leq 1$. Now suppose $\abs{S} \geq 2$. As in Lemma~\ref{lem:one_des_ocotpus}, the image of a corresponding vertex $i$ to an $x_i \in S$ is either the corresponding singleton of the image of $x_i$ or $[n]$.
 
First suppose there exists a design vertex $x_i \in S$ where $\Phi^\sq(i) = [n]$. Then $\Phi^\sq(x_i)$ is a design vertex. The inverse image of the corresponding singleton to $\Phi^\sq(x_i)$ is $[n]$ since it must either be the singleton corresponding to $x_i$ or $[n]$, and it is not $i$ since $\Phi^\sq(i)=[n]$ already. But then for any other $x_j \in S$, the image of $x_j$, which is a design vertex different from $\Phi^\sq(x_i)$, is compatible with $\Phi^\sq([n])$, which is a singleton corresponding to $\Phi^\sq(x_i)$. This is a contradiction since $x_i$ is incompatible with $[n]$, hence there is no such isomorphism.
    
Now suppose for all $x_i \in S$, $i$ maps to the singleton corresponding to $\Phi^\sq(x_i)$. Consider $\Phi^\sq$ acting on the subcomplex formed by restriction $\mc N^\sq(\mc B)$ to the set of vertices that are compatible with each $i \in S$. This subcomplex is isomorphic to $\mc N^\sq(\mc B/{\{i \mid x_i \in S\}})$. Its image is the set of $\mc B'$ compatible with each $i \in S$, giving us an isomorphism 
\[
    \Phi^\sq/_{\{i \mid x_i \in S\}}: \mc N^\sq(\mc B/\{i \mid x_i \in S\}) \rightarrow \mc N^\sq(\mc B/\{\Phi(i) \mid x_i \in S\}).
\]
Notice that $\Phi^\sq/_{\{i \mid x_i \in S\}}$ sends no design vertex to another design vertex, so by Corollary~\ref{cor:no_des_rotation} it is an extended rotation. This implies that either $\Phi^\sq([n])$ is a design vertex or $\{i | x_i \in S\} = [n]$. If $\Phi^\sq([n])$ is a design vertex, then $\Phi^\sq([n])$ is compatible with $\Phi^\sq(x_i)$ for any $i \in S$, contradicting the fact that $[n]$ is incompatible with any design vertex, so we must have $\{i \mid x_i \in S\} = [n]$. But then we may relabel our elements of the building set $\mc B'$ so that $x_i \mapsto x_i$ and $i \mapsto i$. Then under this labeling for $e \in \mc B$, $\Phi^\sq(e) \subseteq e$ from considering design vertices such that $\{e,x_t\}$ is an extended nested collection. Similarly, $\Phi^\sq(e) \supseteq e$ since $\Phi^\sq(e)$ from considering design vertices such that $\{e,x_t\}$ is not a extended nested collection. Thus $\Phi^\sq$ is trivial.
\end{proof}

Since by Proposition~\ref{prop:product_iso_factors}, every isomorphism of $\Phi^\sq$ of non-connected building sets factors into isomorphisms between connected components of building sets, Theorem~\ref{thm:ext_to_ext_isom_summary} completely characterizes the isomorphisms $\Phi^\sq$ between extended nested set complexes.
\section{Face Counting}\label{sec:hgamma}

In this section, we study the face numbers of the extended nestohedron. This includes finding recursive formulas for their $f$- and $h$-polynomials and showing that the $\gamma$-vector is nonnegative when the extended nestohedron is flag.

\subsection{Formulas for the \texorpdfstring{$f$}{f}- and \texorpdfstring{$h$}{h}-vectors}

In \cite[Theorem 7.11]{postnikov2009permutohedra}, Postnikov provides, without proof, a recursive formula for the $f$-polynomial of the nestohedron $\calP(\mc B)$, which is the dual of the nested complex $\calN(\mc B)$. In this subsection, we first provide a short proof for the formula for the $f$-polynomial of $\calN(\mc B)$, and then show how to get the original result by Postnikov for the dual $\calP(\mc B)$.

\begin{theorem}\label{thm:f_dual_poly_recursion}

Let $\mc B$ be a building set on $[n]$. Recall that $\mc B_{\max}$ denotes the set of maximal elements, and let $\calS(\mc B)=\{S \subsetneq [n] \mid S \cap M \subsetneq M \;\forall M \in \mc B_{\max}\}$. The $f$-polynomials of the nested complex $\calN(\mc B)$ satisfy the following recursive formula:
\[f_{\calN(\varnothing)}(t)= 1\quad\text{and}\quad f_{\calN(\mc B)}(t) = \sum_{S \in \calS(\mc B)} t^{\abs{(\mc B|_S)_{\max}}}f_{\calN(\mc B|_S)}(t).\]
\end{theorem}

\begin{proof}
For any nested collection $N=\{I_1,\dots,I_k\}$ of $\calN(\mc B)$, let $S=\Supp N = I_1 \cup \dots \cup I_k$. Then $S \in \calS(\mc B)$, and $\{I_j\}$ contains all maximal elements of $B|_S$ by Lemma~\ref{lem:nested_collection_contains_all_max_elements}. We then obtain a $(k-\abs{(\mc B|_S)_{\max}})$-face of $\calN(\mc B|_S)$ by removing the maximal elements of $B|_S$ from $N$. Conversely, for any $S \in \calS(\mc B)$ and a $\ell$-face $N_S$ of $\calN(\mc B|_S)$, one obtain a $(\ell+\abs{(\mc B|_S)_{\max}})$-face of $\calN(\mc B)$ by adding the maximal elements of $B|_S$. One can verify that these two maps are inverse of to each other each other, so we obtain the recursive formula for the $f$-polynomial of $\calN(\mc B)$.
\end{proof}

\begin{corollary}\cite[Theorem 7.11]{postnikov2009permutohedra}\label{cor:f_poly_recursion}
The $f$-polynomial of $\calP(\mc B)$ satisfy the following recurrence relation:
\begin{enumerate}[(\alph*)]
    \item $f_{\calP(\varnothing)}(t)=1$.
    \item If $\mc B_1,\ldots, \mc B_k$ are the connected components of $\mc B$, then \[f_{\calP(\mc B)}(t)=f_{\calP(\mc B_1)}(t) \cdots f_{\calP(\mc B_k)}(t).\]
    \item If $\mc B$ is a connected building set, then \[f_{\calP(\mc B)}= \sum_{S \subsetneq [n]} t^{n-\abs{S}-1}f_{\calP(\mc B|_S)}(t).\]
\end{enumerate}
\end{corollary}

\begin{proof} The first condition is true, and the second condition follows from Lemma \ref{lem:IsJoinOfConnectedComponents}. For the last condition, recall that the $f$-polynomial of a polytope $\calP$ and its dual $\calP^*$ are related by the formula $f_{\calP}(t)=t^{\dim \calP}f_{\calP^*}(t^{-1})$. By \cite[Proposition 4.1]{zelevinsky2006nested}, we know that $\dim \calN(\calB)=\abs{S}-\abs{\calB_{\max}}$ if $\calB$ is a building set on $S$. Hence, we have \[f_{\calN(\mc B|_S)}(t)=t^{\abs{S}-\abs{(\mc B|_S)_{\max}}} f_{\calP(\mc B|_S)}(t^{-1}).\] Replacing the above in the $f$-polynomial equation of Theorem~\ref{thm:f_dual_poly_recursion} gives the desired conclusion.
\end{proof}

Using a similar argument we get 
\begin{proposition}\label{prop:f_vector_ext_wo_design_recursion}
    Let $\mc B$ be a building set on $[n]$. Recall that $\mc B_{\max}$ denotes the set of maximal elements. The $f$-polynomials of the nested complex $\calN(\mc B)$ satisfy the following recursive formula:
\[f_{\calN(\mc B)}(t)\cdot (t+1)^{|\mc B_{\max}|} = \sum_{S \subseteq B} t^{\abs{(\mc B|_S)_{\max}}}f_{\calN(\mc B|_S)}(t),\]
and its dual satisfies
\[
f_{\calP(\mc B)}(t)\cdot (t+1)^{|\mc B_{\max}|} = \sum_{S \subseteq B} f_{\calP(\mc B|_S)}(t)\cdot t^{n-|S|}.
\]
\end{proposition}
\begin{proof}
The second equation follows from the first so we will prove the first. Consider the set of faces in the extended nestohedron which do not contain a design vertex. Every such face is a nested collection of $\mc B$ combined with a subset of the maximal elements of $\mc B$; thus its $f$-vector equals the LHS of our equation. Each such face is also a set of maximal elements combined with a nested collection contained in those maximal elements which is the RHS.
\end{proof}

We now state several results about the $f$- and $h$-polynomials of an extended nestohedron $\calP^\sq(\calB)$. It turns out that one can relate the $f$-polynomial of $\calP^\sq(\calB)$ in terms of the $f$-polynomial of the nestohedron $\calP(\mc B|_S)$ for $S \subseteq [n]$.

\begin{theorem}\label{thm:f_poly_extended_to_original}
For a building set $\mc B$ on $[n]$, the $f$-polynomial of the extended nestohedron $\calP^\sq(\mc B)$ satisfies the following formulas:
\begin{align*}
    f_{\calP^\sq(\mc B)}(t) &= \sum_{S \subseteq [n]} (t+1)^{n-\abs{S}}f_{\calP(\mc B|_S)}(t),\\
    f_{\calP^\sq(\mc B)}(t) &= \sum_{S \subseteq [n]} (t+1)^{\abs{(\mc B|_S)_{\max}}}f_{\calP(\mc B|_S)}(t).
\end{align*}
\end{theorem}

\begin{proof}
For the first formula, note that an extended nested collection $N^\sq=\{I_1,\dots,I_k\} \cup \{x_{j_1},\dots,x_{j_\ell}\}$ corresponds to a nested collection $\{I_1,\dots,I_k\}$ on its support $S=I_1 \cup \dots \cup I_k$, and a subset $\{j_1,\dots,j_\ell\}$ of $[n] \setminus S$. Conversely, for any face $N$ in $\mc N({\mc B|_S})$ and any subset $T$ of $[n] \setminus S$, we obtain an extended nested collection $N \cup \{x_j \mid j \in T\}$ of $\calB$. Summing over all $S \subseteq [n]$ with this choice procedure, we have
\begin{align*}
    f_{\calP^\sq(\mc B)}(t) &= \sum_{S \subseteq [n]} \left(\prod_{i \not \in S} (t+1) \right) f_{\calP(\mc B|_S)}(t) \\
    &=\sum_{S \subseteq [n]} (t+1)^{n-\abs{S}}f_{\calP(B|_S)}(t).
\end{align*}
Notice that we are not over-counting since every face of $\mc N(\mc B|_S)$ must contain the maximal elements of $\mc B|_S$.

For the second formula, note that any extended nested collection $N^\sq=\{I_1,\ldots,I_k\} \cup \{x_{j_1},\ldots,x_{j_\ell}\}$ with maximal building set elements $\{I_{a_1},\ldots,I_{a_r}\}$ corresponds to a set of $S=\{j_1,\ldots,j_{\ell}\}$ combined with a nested collection $\{I_1,I_2,\ldots,I_k\} \setminus \{I_{a_1},\ldots,I_{a_r}\}$ in $\mc B|_{[n] \setminus S}$, as well as a subset of the maximal elements $\{I_{a_1},\ldots,I_{a_r}\}$ of $\mc B|_{[n] \setminus S}$. Conversely for any face subset $S \in [n]$, any nested collection $N$ in $\mc B|_{[n] \setminus S}$, and any subset of maximal elements $M$ of $\mc B|_{[n] \setminus S}$, we may obtain an extended nested collection $N \cup M \cup \{x_j \mid j \in S \}$. Summing over all $S \subseteq [n]$ with this choice procedure, we have that
\begin{align*}
    f_{\mc N^\sq(\mc B)}(t)=\sum_{S \subseteq [n]} t^{n-\abs{S}}(t+1)^{\abs{(\mc B|_S)_{\max}}}f_{\mc N(\mc B|_S)}(t).
\end{align*}
One can manipulate this expression to obtain the final result.
\end{proof}

Since the $h$-polynomial is given by $f_P(t)=h_P(t+1)$, we can also formulate the $h$-polynomial of the extended nestohedron in terms of the $h$-polynomial of nestohedra $\mc P(\mc B|_S)$ for $S\subseteq[n]$.

\begin{corollary}\label{cor:h_polynomial_extended_nest}
For a building set $\mc B$ on $[n]$, the $h$-polynomial of the extended nestohedron $\mc P^{\sq}(\mc B)$ is
\[h_{\mc P^{\sq}(\mc B)}(t)=\sum_{S\subseteq[n]}t^{n-|S|}h_{\mc P(\mc B|_S)}(t).\]
\end{corollary}

The following corollary of Theorem~\ref{thm:f_poly_extended_to_original} comes from the special case of M\"obius inversion, where the poset is the Boolean lattice, known as the ``inclusion-exclusion principle,'' i.e.
\[f(A) = \sum_{B \subseteq A} g(B) \qquad \Leftrightarrow \qquad 
g(A) = \sum_{B \subseteq A} (-1)^{|A|-|B|}  f(B).\]

\begin{corollary}\label{cor:f_poly_original_to_extended}
There is a reverse relation between $f$- and $h$-polynomials of $\calP(\calB)$ and those of $\calP^\sq(\calB|_S)$ for $S \subseteq [n]$ as follows:
\begin{align*}
    f_{\calP(\mc B)}(t) &= \sum_{S \subseteq [n]} (-t-1)^{n-\abs{S}}f_{\calP^\sq(\mc B|_S)}(t),\\
    (t+1)^{\abs{\mc B_{\max}}}&f_{\calP(\mc B)}(t) = \sum_{S \subseteq [n]}(-1)^{n-|S|} f_{\calP^\sq(\mc B|_S)}(t).
\end{align*}
\end{corollary}

Similar to Corollary~\ref{cor:f_poly_recursion}, we have a formula for the $f$-polynomial of $\calP^\sq(\mc B)$ based on the $f\text{-polynomials}$ of strictly smaller building sets $\{\mc B_S \mid S \subsetneq [n]\}$.

\begin{corollary}\label{cor:h_poly_ext_reg_recursions}
Let $\mc B$ be a building set on $[n]$. Then the $h$-polynomials of $\calP(\mc B)$ and $\calP^\sq(\mc B)$ satisfy the recursions:
\begin{align*}
    &\sum_{S \subseteq [n]}\left(t^{n-|S|} - t^{\abs{(\mc B|_S)_{\max}}}\right)h_{\calP(\mc B)}(t) = 0,\\
\sum_{S \subseteq [n]} (-&1)^{n-|S|}\left(t^{n-|S|+\abs{\mc (B|_S)_{\max}}} - 1\right)h_{\calP^\sq(\mc B)}(t) = 0.
\end{align*}
\end{corollary}

\subsection{\texorpdfstring{$a$}{a}- and \texorpdfstring{$b$}{b}-rational functions}
We now introduce a pair of interesting invariants for building sets, called the $a$- and $b$-numbers, that first appeared in \cite{choi2015new} and \cite{park2017graph} for graphical building sets. However, their definition have natural generalizations to arbitrary building sets, which we state here. A building set $\mc B$ is \textbf{even} if all connected components of $\mc B$ have even cardinality, and is \textbf{odd} if all connected components of $\mc B$ have odd cardinality.

\begin{definition}
For a building set $\mc B$ on $[n]$, define the \textbf{$a$-number} of $\mc B$ to be
\begin{equation}\label{eqn:a_number_def}
a(\mc B)\coloneqq \begin{cases} 1, &\text{ if }\mc B  = \varnothing, \\ 0, &\text{ if }\mc B\text{ is not even,} \\ -\sum_{S \subsetneq [n]} a(\mc B|_S),&\text{ otherwise.}\end{cases}
\end{equation}
Define the $b$\textbf{-number} of $\mc B$ to be
\begin{equation}\label{eqn:b_number_def}
b(\mc B)\coloneqq \begin{cases} 1, &\text{ if }\mc B  = \varnothing, \\ 0, &\text{ if }\mc B\text{ is not odd,} \\ -\sum_{S \subsetneq [n]} b(\mc B|_S),&\text{ otherwise.}\end{cases}
\end{equation}
\end{definition}

In \cite{park2017graph}, the $a$- and $b$-numbers of a graphical building set are shown to be related to the $h$-vector of the corresponding (extended) nestohedron. In particular, they proved the following result, which we will show in Proposition~\ref{conj:ab_numbers} to hold for a general building set.

\begin{proposition}\cite[Corollary 4.8]{park2017graph}
\label{prop:cor_4.8_park}
For any undirected graph $G$ with $n$ vertices, let $\mc B=\mc B_G$ be the corresponding building set. We then have
\begin{align*}
    a(\mc B)&=  h_{\calP^\sq(\mc B)}(-1)\quad\text{and}\quad b(\mc B)= (-1)^n h_{\calP(\mc B)}(-1).
\end{align*}
\end{proposition}

\begin{proposition}
\label{conj:ab_numbers}
For an arbitrary building set $\mc B$, 
\begin{align}
    a(\mc B)&= h_{\mc P^{\sq}(\mc B)}(-1),\label{eqn:a_number}\\
    b(\mc B)&=(-1)^n h_{\mc P(\mc B)}(-1).\label{eqn:b_number}
\end{align}
\end{proposition}

Before proving this, we will prove the following lemma.

\begin{lemma}
\label{lem:b_num_sumof_a_num}
For an odd connected building set $\mc B$,
    \[
    b(\mc B) = (-1)^n \sum_{S \subsetneq [n]} a(\mc B|_{S})
    \]
\end{lemma}
\begin{proof}
We proceed by strong induction on the number of elements in the building set. Our base case is when $\mc B$ is empty or consists of a exactly one singleton. For our inductive step, assume for all $m<n$, any odd connected building set $\mc B'$ on $m$ elements satisfies
\[
b(\mc B') = (-1)^n \sum_{S \subsetneq [m]} a(\mc B'|_{S}).
\]
Let $\mc B$ be an odd connected building set on $n$ elements; in particular, $n$ is odd. By Corollary~\ref{cor:f_poly_recursion},
\[
    b(\mc B) = -\sum_{S \subsetneq [n]} 2^{n-|S|-1}b(\mc B|_S).
\]
From our inductive hypothesis
\begin{align*}
-\sum_{S \subsetneq [n]} 2^{n-|S|-1}b(\mc B|_S) &= -\sum_{S \subsetneq [n]} 2^{n-|S|-1}(-1)^{|S|}\sum_{L \subseteq S}a(\mc B|_L)\\
&= -\sum_{L \subseteq B} a(B|_L) (-1)^{|L|}\cdot \frac{1}{2} \cdot \sum_{i=0}^{n-|L|-1}\binom{n-|L|}{i}(-1)^i2^{n-|L|-i}\\
&= -\sum_{L \subsetneq B} a(B|_L) \cdot \begin{cases}
    (-1)^{|L|}, &\quad\text{if } n-|L|\textnormal{ is odd,}\\
    0, &\quad\text{if } n-|L|\textnormal{ is even.}
\end{cases}
\end{align*}
Since $n$ is odd and the $a$-number of building sets with an odd number of singletons is $0$, we have
\[
-\sum_{L \subsetneq B} a(B|_L) = -\sum_{L \subsetneq B} a(B|_L) (-1)^{|L|}= -\sum_{L \subsetneq B} a(B|_L) \cdot \begin{cases}
    (-1)^{|L|}, &\quad\text{if } n-|L|\textnormal{ is odd,}\\
    0, &\quad\text{if } n-|L|\textnormal{ is even.}
\end{cases}
\]
\end{proof}

\begin{proof}[Proof of Proposition~\ref{conj:ab_numbers}]
We first prove Equation~\ref{eqn:a_number} via induction on the number of elements in the building set. For the base case, if $\mc B$ is empty, then $h_{\mc P^{\sq}(\mc B)}(-1) = 1$. If $\mc B$ is a singleton then $h_{\mc P^{\sq}(\mc B)}(-1) = 0$. For the inductive step, consider an arbitrary building set $\mc B$ on $n$ elements, and suppose that for any building set $\mc B'$ with $m<n$ elements,
\[a(\mc B')=h_{\mc P^{\sq}(\mc B')}(-1).\]
From the second recursion in Corollary~\ref{cor:h_poly_ext_reg_recursions},
\[
    h_{\calP^\sq(\mc B)}(-1) = \sum_{S \subsetneq [n]} h_{\calP^\sq(\mc B|_S)}(-1)\cdot \begin{cases}
    -1, &\quad\text{if } n-|S| \textnormal{ is even,}\\
    0, &\quad\text{if } n-|S| \textnormal{ is odd.}
    \end{cases}
\]
When $\mc B$ is not even, there exists a connected component $\mc B|_S$ of $\mc B$ with $|S|$ odd. Then,
\[
h_{\calP^\sq(\mc B)} = h_{\calP^\sq(\mc B|_S)}h_{\calP^\sq(\mc B|_{[n]\setminus S}),}
\]
and by induction,
\[
h_{\calP^\sq(\mc B)}(-1) = \sum_{\substack{R \subsetneq S,\\|R| \textnormal{ is odd}}} h_{\calP^\sq(\mc B|_S)}(-1) = 0.
\]
When $\mc B$ is even,
\begin{align*}
    h_{\calP^\sq(\mc B)}(-1) &= \sum_{\substack{S \subsetneq [n],\\|S| \textnormal{ is even}}} h_{\calP^\sq(\mc B|_S)}(-1)= \sum_{S \subsetneq [n]} h_{\calP^\sq(\mc B|_S)}(-1),
\end{align*}
with the second inequality coming from our inductive hypothesis.

We now prove Equation~\ref{eqn:b_number}. First we show if $\mc B$ is not odd then, $h_{\calP(\mc B)}(-1)=0$. If $\mc B$ is not odd, there exists a connected component $\mc B|_S$ of $B$ such that $|S|$ is even. 
By Corollary~\ref{cor:f_poly_original_to_extended} and Equation~\ref{eqn:a_number},
\[
h_{\calP(\mc B)}(-1) = \sum_{S \subseteq [n]} a(\mc B|_{S}) = 0, 
\]
with the last inequality coming from the recursive definition of $a$-numbers.

Now we show the case when $\mc B$ is odd by strong induction on $n$, the number of elements of a connected building set. Note that since the $h$-polynomials and $b$-numbers of building sets are multiplicative on connected components, our inductive step only needs to consider connected components. The base cases of when the building set is either empty or consists of exactly one singleton clearly holds. For the inductive step, assume that for all $m<n$, any building set $\mc B'$ on $m$ elements satisfies Equation~\ref{eqn:b_number}. Since $\mc B$ is connected, then by Lemma~\ref{lem:b_num_sumof_a_num} and Equation~\ref{eqn:a_number},
\[
    b(\mc B) = (-1)^n\sum_{S \subseteq [n]} h_{\calP^\sq(\mc B|_S)}(-1).
\]
Notice by Corollary~\ref{cor:f_poly_original_to_extended},
\[
(-1)^nh_{\calP(\mc B)}(-1) = (-1)^n\sum_{S \subseteq [n]} h_{\calP^\sq(\mc B|_S)}(-1).
\]
Thus, $b(\mc B) = (-1)^nh_{\calP(\mc B)}(-1)$.
\end{proof}

Using the definitions of $a$- and $b$-numbers, the authors of \cite{park2017graph} compute the Betti numbers of the real toric manifold corresponding to the polytopes $\calP(\calB)$ and $\calP^\sq(\calB)$ when $\calB=\calB_G$ is a graphical building set. Note that if we consider the complex toric manifold corresponding to a simple polytope $\calP$, then its Betti numbers are known to the coefficients of the $h$-polynomial of $\calP$.

\begin{theorem}[{\cite[Theorem 1.1 and Theorem 1.2]{park2017graph}}]
\label{thm:thm_1.1,1.2_park}
Let $G$ be an undirected graph with $V(G)=[n]$ and $\calB=\calB_G$. Then the $i$-th Betti number of the real toric manifold associated to $\calP(\calB)$ is given by \[\beta^i(X^\R(\calP(\calB)))= \sum_{\substack{S \subseteq [n],\\ \abs{S}=2i}} \abs{a\left(\calB|_S\right)}.\] Similarly, the $i$-th Betti number of the real toric manifold associated to $\calP^\sq(\calB)$ is given by \[\beta^i(X^\R(\calP(\calB)))= \sum_{\substack{S \subseteq [n],\\ \abs{S}+\kappa(\calB|_S)=2i}} \abs{b\left(\calB|_S\right)},\] where $\kappa(\calB|_S)$ is the number of connected components of $\calB|_S$.
\end{theorem}

They use this theorem to prove the following result.

\begin{theorem}\cite[Theorem 1.3]{park2017graph}\label{thm:same_betti_numbers}
Let $G$ be a forest, and let $L(G)$ be the line graph of $G$. Then we have: \[\beta^i(X^\R(\calP(\mc B_G)))=\beta^i(X^\R(\calP^\sq(\mc B_{L(G)}))).\]
Here, for a graph $G$, the line graph of $G$, denoted $L(G)$ is the graph constructed by associating a vertex with each edge of $G$ and connecting two vertices with an edge if the corresponding edges of $G$ have a vertex in common.
\end{theorem}

We observe a similar phenomenon for the $h$-polynomial of the corresponding simple polytopes, leading us to the following theorem.

\begin{theorem}\label{thm:same_h_vector}
Let $G$ be a forest and $L(G)$ be the line graph of $G$. Then \[h_{\calP(\mc B_G)}(t)=h_{\calP^\sq(\mc B_{L(G)})}(t).\] 
\end{theorem}

To prove Theorem~\ref{thm:same_h_vector} in the same flavor as the proof of Theorem~\ref{thm:same_betti_numbers}, we generalize the $a$ and $b$-number to their $t$-analogues of $a$ and $b$-rational functions, i.e., rational polynomials in $t$ that evaluate at $t=-1$ to the $a$- and $b$-numbers.

\begin{definition}
Let $\mc B$ be a building set on $[n]$. We define the \textbf{$a$- and $b$-rational functions} to be \[a(\mc B,t)=\frac{h_{\calP^\sq(\mc B)}(t)}{(-t)^n}\quad\text{and}\quad b(\mc B,t)=\frac{h_{\calP(\mc B)}(t)}{t^n}.\]
\end{definition}

These rational functions satisfies the following identities, which are $t$-analogues of the recurrence formulas of the $a$- and $b$-numbers and \cite[Theorem~4.4]{park2017graph}. These all follow from Theorem \ref{thm:f_poly_extended_to_original}, Corollary \ref{cor:f_poly_original_to_extended}, and Corollary \ref{cor:h_poly_ext_reg_recursions}.

\begin{proposition}\label{prop:t-analogues_a_b}
For any building set $\mc B$ on $[n]$, we have: 
\[\sum_{S \subseteq [n]} \left(t^{n+\abs{\mc (\mc B|_S)_{\max}}}-t^{|S|}\right)a(\mc B|_S,t)=0\quad\text{and}\quad \sum_{S \subseteq [n]} \left(t^{n}-t^{|S|+\abs{\mc (\mc B|_S)_{\max}}}\right)b(\mc B|_S,t)=0.\]
In addition, we have
\[a(\mc B,t)=(-1)^n\sum_{S \subseteq [n]} b(\mc B,t)\quad\text{and}\quad b(\mc B,t)=(-1)^n\sum_{S \subseteq [n]} a(\mc B,t).\]
\end{proposition}

Notice that when $\mc B = \mc B_G$ for a connected graph $G$, we recover the recurrence formulas for the $a$- and $b$-numbers from evaluating the above recurrences at $t=-1$. As an abuse of notation, from now on for a graph $G$ we write $b(G,t)$ and $a(G,t)$ to refer to $b(\mc B_G,t)$ and $a(\mc B_G,t)$ respectively. The $b-$rational function satisfies the following extension of Proposition~\ref{prop:f_vector_ext_wo_design_recursion}: 

\begin{lemma}\label{lem:tool_for_same_h}
For a forest graph $G$,
\[
    \frac{b(G,t)t^{|G|-|L(G)|}t^{|G|}}{(t-1)^{|G|}} = \sum_{S \subseteq G} \frac{b(S,t)t^{\abs{S}}}{(t-1)^{\abs{S}}}.
\]
\end{lemma}
\begin{proof}
Let $G'$ be $G \setminus \{\textnormal{isolated vertices of }G \}$. Then by Proposition~\ref{prop:f_vector_ext_wo_design_recursion} and noting $\abs{G'}-\abs{L(G')} = \kappa(G')$,
\[
    \frac{b(G',t)t^{\abs{G'}-\abs{L(G')}}t^{\abs{G'}}}{(t-1)^{\abs{G'}}} = \sum_{S \subseteq G'} \frac{b(S,t)t^{\abs{S}}}{(t-1)^{\abs{S}}}.
\]
Since adding isolated vertices to a graph $H$ multiplies $b(H,t)$ by $1/t$ and does not change its line graph, we have
\begin{align*}
    \frac{b(G,t)t^{\abs{G}-\abs{L(G)}}t^{\abs{G}}}{(t-1)^{\abs{G}}} &= \frac{b(G',t)t^{\abs{G'}-\abs{L(G')}}t^{\abs{G'}}}{(t-1)^{\abs{G'}}} \cdot \left(\frac{t}{t-1}\right)^{\abs{G}-\abs{G'}} \\
    &=\sum_{S \subseteq G'} \frac{b(S,t)t^{\abs{S}}}{(t-1)^{\abs{S}}}\left(1+\frac{1}{t-1}\right)^{\abs{G}-\abs{G'}}\\
    &=\sum_{S \subseteq G} \frac{b(S,t)t^{\abs{S}}}{(t-1)^{\abs{S}}}.
\end{align*}
The last equality comes from the bijection between subgraphs of $G$ and subgraphs of $G'$ union a subset of isolated vertices of $G$, and each isolated vertex included multiplies $\frac{b(S,t)t^{\abs{S}}}{(t-1)^{\abs{S}}}$ by $(t-1)^{-1}$.
\end{proof}

\begin{proof}[Proof of Theorem~\ref{thm:same_h_vector}]
For a graph $H$, let $S(H)$ denote the set of spanning subgraphs of $H$ without isolated vertices. Importantly, the set of induced subgraphs of $L(G)$ are in bijection with the line graphs of spanning subgraphs of induced subgraphs of $G$ without isolated vertices.
We want to show $b(G,t)t^{\abs{G}}$ is equal to
\[
   a(L(G),t)(-t)^{\abs{L(G)}} =t^{\abs{L(G)}}\sum_{L \subseteq L(G)} b(L,t)= t^{\abs{L(G)}}\sum_{H \subseteq G}\sum_{T \in S(H)} b(L(T),t).
\]
With the first equality coming from Proposition~\ref{prop:t-analogues_a_b}. By inclusion-exclusion, it is enough to show that
\begin{align}
        \sum_{T \in S(G)} b(L(T),t) = \sum_{H \subseteq G} (-1)^{\abs{G}-\abs{H}}b(H,t)t^{\abs{H}-\abs{L(H)}}.\label{eqn:mobius_inversion_equiv}
\end{align}
We will use induction to prove that Equation~\ref{eqn:mobius_inversion_equiv} holds. First, notice that it holds for the empty building set, since both sides are $0$. Now, assume that for all $S \subsetneq G$ the equation holds. By Proposition~\ref{prop:t-analogues_a_b} and our inductive assumption,
\begin{align*}
    \left(t^{\abs{L(G)}}-t^{\abs{G}}\right)\sum_{T \in S(G)} b(L(T),t) &= -\sum_{H \subsetneq G} \left(t^{\abs{L(G)}}-t^{\abs{H}}\right)\sum_{T \in S(H)} b(L(T),t)\\
    &=-\sum_{H \subsetneq G} \left(t^{\abs{L(G)}}-t^{\abs{H}}\right)\sum_{S \subseteq H} (-1)^{\abs{H}-\abs{S}}b(S,t)t^{\abs{S}-\abs{L(S)}} \\
    &= -\sum_{S \subseteq G} b(S,t)t^{\abs{S}-\abs{L(S)}}\sum_{i=0}^{\abs{G}-\abs{S}-1}\binom{\abs{G}-\abs{S}}{i}(-1)^i(t^{\abs{L(G)}}-t^{\abs{S}}t^i)\\
    &= \sum_{S \subsetneq G} b(S,t)t^{\abs{S}-\abs{L(S)}}(-1)^{\abs{G}-\abs{S}}\left(t^{\abs{L(G)}}+t^{\abs{S}}(t-1)^{\abs{G}-\abs{S}}-t^{\abs{G}} \right).
\end{align*}
Thus, it is enough to show
\[
    b(G,t)t^{\abs{G}-\abs{L(G)}}\left(t^{\abs{L(G)}}-t^{\abs{G}}\right) = \sum_{S \subsetneq G} b(S,t)t^{\abs{S}-\abs{L(S)}}(-1)^{\abs{G}-\abs{S}}t^{\abs{S}}(t-1)^{\abs{G}-\abs{S}},
\]
which is equivalent to 
\[
    b(G,t)t^{\abs{G}} = \sum_{S \subseteq G} b(S,t)t^{\abs{S}-\abs{L(S)}}(-1)^{\abs{G}-\abs{S}}t^{\abs{S}}(t-1)^{\abs{G}-\abs{S}}.
\]
This holds by applying M\"obius inversion to the equation of Lemma~\ref{lem:tool_for_same_h}.
\end{proof}

\subsection{Gal's conjecture for flag extended nestohedra}\label{subsec:gals_conjecture}

Recall that for a $d$-dimensional simple polytope $P$, the $\gamma$-vector $(\gamma_0,\gamma_1,\ldots,\gamma_{\lfloor{d/2\rfloor}})$ is defined by the $h$-polynomial:
\[h_P(t)=\sum_{i=0}^{\lfloor{d/2\rfloor}}\gamma_i t^i(1+t)^{d-2i}.\]
The $\gamma$-polynomial is then $\gamma(t)=\sum\gamma_i t^i$. Gal conjectured that the $\gamma$-vector is nonnegative for flag simple polytopes (Conjecture~\ref{conj:gal}), and Volodin showed that this conjectures holds for flag nestohedra \cite{volodin}. In this subsection, we show the analagous result for flag extended nestohedra.

\begin{theorem}[Gal's Conjecture for Flag Extended Nestohedra]\label{thm:gal_ext_flag}
Let $\mc P^{\sq}(\mc B)$ be a flag extended nestrodron. Then the $\gamma$-vector of $\mc P^{\sq}(\mc B)$ is nonnegative.
\end{theorem}

This implies that Gal's conjecture holds for this large class of flag simple polytopes. To prove this result, we use a similar technique as the one used by Volodin. The general outline of the technique is as follows. For a building set $\mc B$ with a flag extended nestohedron, there exists a minimal building set $\mc D\subseteq\mc B$ such that $\mc D$ also gives a flag extended nestohedron $\mc P^{\sq}(\mc D)$; this is our starting polytope, and one can show that the $\gamma$-vector of $\mc P^{\sq}(\mc D)$ is nonnegative. One can add back in elements of $\mc B\setminus\mc D$, with each added element corresponding to \textbf{shaving} a face of the starting polytope. We show that each time a face is shaved, the $\gamma$-vector remains nonnegative. 

First, we provide some preliminary definitions and lemmas. A building set $\mc B$ is \textbf{flag} if $\mc P(\mc B)$ is a flag simple polytope. By Lemma~\ref{lem:flag}, this means that $\mc P^{\sq}(\mc B)$ is a flag simple polytope as well. A connected building set $\mc D$ on $S$ is a \textbf{minimal flag building set} if it is flag and there does not exist a flag connected building set $\mc C\subsetneq \mc D$ that is also on $S$. Such building sets are characterized in \cite[Section 7.2]{PRW}. In particular, if $I$ is a non-singleton element of a minimal flag building set $\mc D$, then there exist two elements $I_1,I_2\in\mc D$ such that $I_1\sqcup I_2=I$.

\begin{definition}[\cite{aisbett14}]
Let $\mc B$ be a building set. A \textbf{binary decomposition} or \textbf{decomposition} of a non-singleton element $I\in \mc B$ is a set $\mc D\subseteq\mc B$ such that $\mc D$ forms a minimal flag building set on $I$.

If $I\in \mc B$ has binary decomposition $\mc D$, then the two inclusion-maximal elements $I_1,I_2\in \mc D\setminus\{I\}$ are the \textbf{maximal components} of $I$ in $\mc D$.
\end{definition}

Recall Lemma~\ref{lem:IsJoinOfConnectedComponents}, which says that if $\mc B$ is a building set with connected components $\mc B_1,\ldots,\mc B_k$, then 
\[\mc N^{\sq}(\mc B)\simeq\mc N^{\sq}(\mc B_1)*\cdots*\mc N^{\sq}(\mc B_k).\]
This implies that the $\gamma$-polynomial for $\mc P^{\sq}(\mc B)$ is given by
\[\gamma_{\mc P^{\sq}(\mc B)}(t)=\gamma_{\mc P^{\sq}(\mc B_1)}(t)\cdots\gamma_{\mc P^{\sq}(\mc B_k)}(t),\]
so it is enough to consider connected building sets for the remainder of this section in order to prove Theorem~\ref{thm:gal_ext_flag}. 

The following lemmas show that given a flag building set $\mc B$ and a minimal flag building set $\mc D\subseteq\mc B$, we can add elements of $\mc B\setminus\mc D$ to $\mc D$ and remain a flag building set with each addition. 

\begin{lemma}[\cite{aisbett12}, Lemma 7.2]\label{lem:flag_iff_decomp1}
A building set $\mc B$ is flag if and only if every non-singleton element $I\in\mc B$ has a binary decomposition.
\end{lemma}

\begin{lemma}[\cite{aisbett12}, Corollary 2.6]\label{lem:flag_iff_decomp2}
A building set $\mc B$ is flag if and only if for every non-singleton element $I\in \mc B$, there exist two elements $I_1,I_2\in\mc B$ such that $I_1\cap I_2=\varnothing$ and $I_1\cup I_2 = I$.
\end{lemma}

\begin{lemma}[\cite{aisbett12}, Theorem 3.1, \cite{volodin}, Lemma 6]\label{lem:flag_induct}
Let $\mc B$ and $\mc B'$ be connected flag building sets on $[n]$ such that $\mc B\subseteq\mc B'$. Then $\mc B'$ can be obtained from $\mc B$ by successively adding elements so that at each step, the set is a flag building set.
\end{lemma}

Given a connected flag building set $\mc B$ on $[n]$ with decomposition $\mc D$ of $[n]$, there exists an ordering $I_1,I_2,\ldots,I_k$ of the elements of $\mc B\setminus \mc D$, such that $\mc B_j = \mc D\cup \{I_1,\ldots, I_j\}$ is a flag building set for all $1\leq j\leq k$. Such an ordering exists by Lemma~\ref{lem:flag_induct}, and is called a \textbf{flag ordering} in \cite{aisbett14}.

Next, we show that each time we add an element to obtain $\mc B$ from $\mc D$, this corresponds to the geometric action of \textbf{shaving} (see Remark \ref{rmk:shave_faces}) a face of the extended nestohedron.

\begin{lemma}
\label{lem:codim2}
Suppose that a connected flag building set $\mc B'$ on $[n]$ is obtained from the flag building set $\mc B$ on $[n]$ by adding an element $I\subseteq [n]$. Then $\mc P^{\sq}(\mc B')$ can be obtained from $\mc P^{\sq}(\mc B)$ by shaving a codimension $2$ face.
\end{lemma}

\begin{proof}
We show the dual version of this statement, i.e., that $\mc N^{\sq}(\mc B')$ can be obtained from $\mc N^{\sq}(\mc B)$ by stellarly subdividing a face of dimension $2$. Since $\mc B'$ is flag, by Lemma~\ref{lem:flag_iff_decomp2} there exist two elements $I_1,I_2\in \mc B$ such that $I_1\cap I_2=\varnothing$ and $I_1\cup I_2 = I$. Notice that $\{I_1,I_2\}$ forms an extended nested collection of $\mc B$, since they are disjoint but $I_1\cup I_2 = I\notin \mc B$, so there exists a dimension $2$ face, or an edge, between the two vertices of $\mc N^{\sq}(\mc B)$ corresponding to $I_1$ and $I_2$. Stellarly subdivide this edge, adding in a vertex corresponding to the element $I$, and call this new simplicial complex $\mc M$. We'll show that $\mc M\simeq \mc N^{\sq}(\mc B')$.

First notice that any facet of $\mc N^{\sq}(\mc B)$ corresponding to a maximal extended nested collection $N$ that does not contain both $I_1$ and $I_2$ remains in $\mc M$. Such maximal collections are still maximal nested collections of $\mc B'$. However, for any facet of $\mc N^{\sq}(\mc B)$ corresponding to a maximal extended nested collection $N$ such that $I_1,I_2\in N$, stellar subdivision replaces $N$ with two new facets,
\[N_1= (N\setminus I_2)\cup I\quad\text{and}\quad N_2 = (N\setminus I_1)\cup I.\]
Since $I\in \mc B'$, an extended nested collection cannot have both $I_1$ and $I_2$. However, a collection can have either $I_1$ or $I_2$, as well as $I$, since $I_1,I_2\subseteq I$. Thus, $N_1$ and $N_2$ correspond exactly to the new facets of $\mc N^{\sq}(\mc B')$.
\end{proof}

Let $\Delta^d$ denote the $d$-dimensional simplex. The following lemma provides a recursive formula for the $\gamma$-polynomial of a polytope that was obtained by shaving the face of another polytope.

\begin{lemma}[\cite{volodin}, Corollary 1]
\label{lem:volocor}
Let $Q$ be the simple polytope obtained from the $n$-dimensional simple polytope $P$ by shaving the face $G$ of dimension $k$. Then
\[\gamma_Q(t)=\gamma_P(t)+\gamma_G(t)\cdot\gamma_{\Delta^{n-k-1}}(t).\]
\end{lemma}

By the above lemma, when we shave a codimension $2$ face $F_0$ from an extended nestohedra, then we add $\gamma_G(t)\cdot\gamma_{\Delta^{n-k-1}}(t) = t \cdot \gamma_{F_0}(t)$ to the $\gamma$-polynomial. By Lemma \ref{lem:codim2} and Proposition~\ref{prop:linkdecomp_extended_I}, we conclude the following.

\begin{proposition}\label{prop:extgammadecomp}
Let $\mc P^{\sq}(\mc B), \mc P^{\sq}(\mc B')$ be flag extended nestohedra such that $\mc B'$ is the result of adding $I$ to $\mc{B}$, and both $\mc B'$ and $\mc B$ are on $[n]$. Then,
\begin{align*}
\gamma_{\mc P^{\sq}(\mc B')}(t) &= \gamma_{\mc P^{\sq}(\mc B)}(t) + t \gamma_{\mc P(\mc B'|_{I})}(t)\gamma_{\mc P^{\sq}(\mc B'/{I})}(t) \\
&= \gamma_{\mc P^{\sq}(\mc B)}(t) + t \gamma_{\mc P(\mc B|_{I})}(t)\gamma_{\mc P^{\sq}(\mc B/{I})}(t).
\end{align*}
\end{proposition}

We are now able to prove our main result of this subsection.

\begin{proof}[Proof of Theorem~\ref{thm:gal_ext_flag}]
Let $\mc B$ be a connected building set on $[n]$. We strongly induct on $n$, the number of singletons in the building set, as well as induct on the number of elements $k$ of $\mc B\setminus\mc D$, where $\mc D$ is the decomposition of $[n]$ in $\mc B$.

Our base case of $n=1$ and $k=0$ is easily covered with $\mc B =\{\{1\}\}$. We now cover the other base cases of $n>1$ and $k=0$. By \cite{PRW} Proposition 7.3, the minimal flag building sets are binary trees. For any binary tree $\mc B$, there exists a plane binary tree $\mc B'$ constructed by relabeling the singletons in $\mc B$ such that $\mc N^{\sq}(\mc B) \simeq \mc N^{\sq}( \mc B')$ and each element of $\mc B'$ is an interval. By Theorem \ref{thm:intervals}, there exists a $\mc C$ such that $\mc N^{\sq}(\mc B') \simeq \mc N(\mc C)$. Since $\mc N^{\sq}(\mc B)$ is flag, we have that $\mc N(\mc C)$ is also flag. Thus by \cite{volodin}, we have that the polynomials $\gamma_{\mc N(\mc C)}(t) = \gamma_{\mc N^{\sq}(\mc B)}(t)$ have nonnegative coefficients.

Suppose $\mc B\subsetneq \mc B'$ are connected flag building sets on $[n]$, with $\mc B'=\mc B\cup\{I\}$. Assume for our inductive hypothesis that $\gamma_{\mc P^{\sq}(\mc B)}(t)$ has nonnegative coefficients, and for any flag building set $\mc C$ on $[m]$ with $m<n$, that $\gamma_{\mc P^{\sq}(\mc C)}(t)$ also has nonnegative coefficients. 

By Proposition~\ref{prop:extgammadecomp}, it is enough to show that the polynomials $\gamma_{\mc P(\mc B|_I)}(t)$ and $\gamma_{\mc P^{\sq}(\mc B/I)}(t)$ have nonnegative coefficients. Notice that since $\mc B'$ is a flag building set on $[n]$ and $I\in\mc B'$, then the building sets $\mc B'|_I=\mc B|_I$ and $\mc B'/I=\mc B/I$ are also flag. By \cite{volodin}, the polynomial $\gamma_{\mc P(\mc B|_I)}$ has nonnegative coefficients, as it is the $\gamma$-polynomial of a non-extended flag nestohedron. The building set $\mc B/I$ is isomorphic to a flag building set $\mc C$ on $[m]$, with $m<n$. In addition, the polynomial $\gamma_{\mc P^{\sq}(\mc B/I)}$ has nonnegative coefficients by our inductive hypothesis on $n$, the number of singletons. Thus, we have that the polynomial $\gamma_{\mc P^{\sq}(\mc B')}$ has nonnegative coefficients, as desired.
\end{proof}

\section{Chordal Building Sets and the \texorpdfstring{$\gamma$}{Gamma}-Vector}\label{sec:comb_gamma}

Although the previous section shows that the $\gamma$-vector is nonnegative for flag extended nestohedra, our proof for the result does not provide a combinatorial interpretation for the $\gamma$-vector. Before Volodin's result on the nonnegativity of the $\gamma$-vector for arbitrary flag nestohedra, Postnikov, Reiner, and Williams \cite{PRW} proved Gal's conjecture for flag nestohedra $\mc P(\mc B)$ when $\mc B$ is a chordal building set. They do this by providing a combinatorial interpretation for the $\gamma$-vector for these polytopes.

In this section, we give a combinatorial interpretation of the $h$- and $\gamma$-vectors of the extended nestohedra for chordal building sets by following the method used in \cite{PRW}. Along the way, we will define some combinatorial objects related to building sets that share some nice properties with extended nested complexes and extended nestohedra.

Most of our definitions, results, and proofs are analogous to ones given in \cite{PRW} for the non-extended case. We will indicate the correspondences accordingly.

\subsection{Extended \texorpdfstring{$\mc B$}{B}-Forests}\label{subsec:b_forests}

We first discuss extended $\mc B$-forests, which are combinatorial objects associated with extended nested set complexes. These are extensions of \textbf{$\mc B$-forests} and \textbf{$\mc B$-trees}, which were originally defined by Postnikov in \cite{postnikov2009permutohedra}. Our goal of this subsection is to define these forests and then show the following result.

\begin{proposition}\label{prop:h_trees}
For a connected building set $\mc B$ on $[n]$, the $h$-polynomial of the extended nestohedron $\mc P^{\sq}(\mc B)$ is given by
\[h_{\mc P^{\sq}(\mc B)}(t)=\sum_{S\subseteq [n]}t^{n-|S|}\sum_{F}t^{\des(F)},\]
where the second sum is over $\mc B|_S$-forests.
\end{proposition}

We now begin to define extended $\mc B$-forests. A \textbf{rooted tree} is a tree with a \textbf{root}, or a distinguished node. We can view a rooted tree $T$ as a partial order on the tree's nodes, with $i<_T j$ if node $j$ lies on the unique path from node $i$ to the root. We can also view a rooted tree as a directed graph, with all edges directed towards the root. 

Let $F$ be a forest of rooted trees $\{T^{(1)},\ldots,T^{(k)}\}$ on $S\subseteq [n]$, meaning that
\[S=\bigsqcup_{i=1}^k \{\text{nodes of $T^{(i)}$}\}.\]
Then $F$ is called a \textbf{rooted forest}, and its roots are the roots of each tree $T^{(i)}$.

If $i$ is a node of $F$ and $i\in T^{(j)}$, then let
\[F_{\leq i}\coloneqq \{\ell\mid \ell\text{ is a descendant of $i$ in $T^{(j)}$}\}.\]
Note that $i\in F_{\leq i}$. We can think of $F$ as the Hasse diagram for a poset, where nodes $i$ and $j$ are \textbf{incomparable} in the forest $F$ if either
\begin{enumerate}
    \item $i$ and $j$ are in two separate trees of $F$, or
    \item $i$ and $j$ are in the same tree of $F$, but neither is a descendant of the other.
\end{enumerate}
\begin{definition}[cf. \cite{postnikov2009permutohedra}, Definition 7.7]\label{defn:B_forest}
Let $\mc B$ be a connected building set on $[n]$ and $S\subseteq [n]$. Define a $\mc B|_S$\textbf{-forest} as a rooted forest $F$ on vertex set $S$ such that
\begin{enumerate}[(F1)]
    \item For any $i\in S$, the node set $F_{\leq i}\in\mc B|_S$. \label{item:bforest_F1}
    \item For $k\geq 2$ incomparable nodes $i_1,\ldots,i_k\in S$, we have that $\bigcup_{j=1}^k F_{\leq i_j}\notin \mc B$. \label{item:bforest_F2}
    \item The sets $F_{\leq i}$, for all roots $i$ of $F$, are exactly the maximal elements of the building set $\mc B$. \label{item:bforest_F3}
\end{enumerate}
The union of $\mc B|_S$-forests over all $S\subseteq[n]$ is the set of \textbf{extended $\mc B$-forests}.
\end{definition}

\begin{example}\label{eg:b_forest}
Consider $\mc B_\Gamma$ with $\Gamma = P_4$. Let $S = \{1,2,4\}$. Then an example of an $\mc B|_S$-forest is the forest with two trees, $T_1,T_2$, shown in Figure~\ref{fig:Bs_forest}, with the nodes $4$ and $2$ the roots of each tree respectively. 

\begin{figure}[H]
    \centering
    \includegraphics{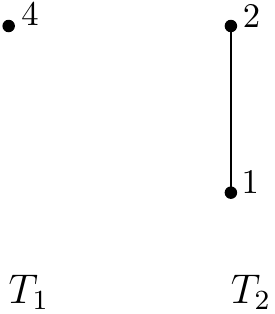}
    \caption{Extended $\mc B$-forest on $S$, consisting of two rooted trees $T_1$ and $T_2$.}
    \label{fig:Bs_forest}
\end{figure}

This is an extended $\mc B$-forest on $S$ because there are two trees, one for each connected component of $\mc B|_S$, and the node set $F_{\leq 2}=\{1,2\}$ is an element of $\mc B$. In addition, for any pair of incomparable nodes, such as $2$ and $4$, we have that
\[F_{\leq 2}\cup F_{\leq 4}=\{1,2,4\}\notin\mc B.\]
\end{example}

By condition \ref{item:bforest_F3}, the number of connected components (i.e., the number of trees) of a $\mc B|_S$-forest equals the number of connected components of the building set $\mc B|_S$. When $S=[n]$, a $\mc B|_S$ forest consists of a single rooted tree. In this case, the tree is actually a $\mc B$-tree. In \cite{PRW}, the authors provide the following result about the $h$-polynomial of non-extended nestohedra in terms of $\mc B$-trees.

\begin{proposition}\cite[Corollary 8.4]{PRW}\label{prop:h_poly_nonext}
For a connected building set $\mc B$ on $[n]$, the $h$-polynomial of the non-extended nestohedron $\mc P(\mc B)$ is given by
\[h_{\mc P(\mc B)}(t)=\sum_{T}t^{\des(T)},\]
where the sum is over $\mc B$-trees $T$.
\end{proposition}

Comparing the definitions of an extended maximal nested collection and an extended $\mc B$-forest, it is not hard to see that we have the following bijection.

\begin{proposition}
For a connected building set $\mc B$ on $[n]$, the map sending a forest of rooted trees $F$ on node set $S\subseteq [n]$ to the collection of elements
\[\{F_{\leq i}\mid i\in S\}\cup \{x_i\mid i\notin S\}\]
gives a bijection between extended $\mc B$-forests and maximal extended nested sets.
\end{proposition}

\begin{example}
The extended $\mc B$-forest of Example~\ref{eg:b_forest} corresponds to the maximal extended nested set
\[N = \{\{1\},\{4\},\{1,2\},x_4\}.\]
\end{example}

We now define a statistic on posets that is used in the formula for the $h$-polynomial in terms of extended $\mc B$-forests.

\begin{definition}\label{defn:tree_des}
Given a poset $F$ on $[n]$, define the \textbf{descent set} $\Des(F)$ to be the set of ordered pairs $(i,j)$ for which $i\lessdot_F j$ is a covering relation in $F$ but $i>_\Z j$, where $>_{\Z}$ is the standard partial order on the integers. Define the \textbf{descent number} to be $\des(F)\coloneqq \abs{\Des(F)}$.
\end{definition}

\begin{example}\label{eg:poset_des}
Consider the poset $F$ on $[8]$ shown in Figure~\ref{fig:poset_des}. The descent set of $F$ is
\[\Des(F)=\{(7,1),(6,1),(8,3)\},\]
and the descent number is $\des(F)=3$.

\begin{figure}[H]
    \centering
    \includegraphics{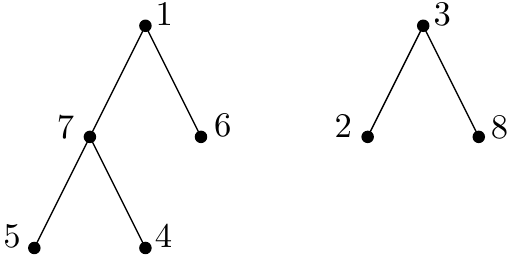}
    \caption{Poset $F$.}
    \label{fig:poset_des}
\end{figure}
\end{example}

Recall from Corollary~\ref{cor:h_polynomial_extended_nest} that the recursive formula for the $h$-polynomial of extended nestohedra is given by
\[h_{\mc P^{\sq}(\mc B)}(t)=\sum_{S\subseteq [n]} t^{n-|S|}h_{\mc P(\mc B|_S)}(t).\]
We are now ready to prove our result about $h$-polynomials of extended nestohedra in terms of extended $\mc B$-forests.

\begin{proof}[Proof of Proposition~\ref{prop:h_trees}]
Let $\{F_S\}$ denote the set of $\mc B|_S$-forests for $S\subseteq[n]$. We first show that
\[h_{\mc P(\mc B|_S)}(t)=\sum_{F\in\{F_S\}}t^{\des(F)}.\]
Suppose $\mc B|_S$ consists of connected components $\mc B_1,\ldots,\mc B_k$. Then, by Lemma~\ref{lem:IsJoinOfConnectedComponents}, we have that
\[\mc P(\mc B|_S)=\mc P(\mc B_1)*\cdots*\mc P(\mc B_k),\]
implying that the $h$-polynomial of $\mc P(\mc B|_S)$ can be given by
\[h_{\mc P(\mc B|_S)}(t)=h_{\mc P(\mc B_1)}(t)\cdots h_{\mc P(\mc B_k)}(t).\]
By Proposition~\ref{prop:h_poly_nonext}, we can rewrite the polynomial using $\mc B_i$-trees:
\[h_{\mc P(\mc B|_S)}(t)=\sum_{T\in\{\text{$\mc B_1$-trees}\}}t^{\des(T)}\cdots\sum_{T\in\{\text{$\mc B_k$-trees}\}}t^{\des(T)}.\]
Notice that each extended $\mc B$-forest on $S$ consists of exactly one $\mc B_i$-tree for every $i$. In addition, $F$ is an extended $\mc B$-forest on $S$ consisting of trees $T_1,\ldots,T_k$, with $T_i$ a $\mc B_i$-tree, then
\[\des(F)=\des(T_1)+\cdots+\des(T_k)\implies t^{\des(F)}=t^{\des(T_1)}\cdots t^{\des(T_k)}.\]
This shows that
\begin{align*}
    h_{\mc P(\mc B|_S)}(t)&=\sum_{T\in\{\text{$\mc B_1$-trees}\}}t^{\des(T)}\cdots\sum_{T\in\{\text{$\mc B_k$-trees}\}}t^{\des(T)}=\sum_{F\in\{F_S\}}t^{\des(F)}.
\end{align*}
Plugging into our recursive formula for the $h$-polynomial of the extended nestohedron gives us our result:
\[h_{\mc P^{\sq}(\mc B)}(t)=\sum_{S\subseteq[n]}t^{n-|S|}h_{\mc P(\mc B|_S)}(t)=\sum_{S\subseteq[n]}t^{n-|S|}\sum_{F\in\{F_S\}}t^{\des(F)}.\]
\end{proof}

Before continuing with our study of the $h$-polynomial, we use extended $\mc B$-forests to describe the vertices of the polytope defined in the construction of the extended nestohedron in Theorem~\ref{thm:minkowski_polytopality}. By following the maps of facets defined in Theorem~\ref{thm:minkowski_polytopality}, one can show through direct computation that the coordinates of the extended nested collection can be written as follows.

\begin{proposition}\label{prop:ext_nestohedra_vertices}
    The coordinates $v=(v_1,\dots,v_n)$ of a vertex corresponding to a maximal extended nested collection $N$ in the extended nestohedron $\mc P^{\sq}(\mc B)$ are given by the following:
    \[v_k = \begin{cases} 0, &\text{if }x_k \in N, \\ \abs{\{I \in \mc B \mid k \in I\}} - \abs{F_{\leq k}}+1, &\text{otherwise},\end{cases}\]
    where $F$ is the extended $\mc B$-forest corresponding to $N$.
\end{proposition}

\begin{example}\label{eg:vertices}
Consider the building set $\mc B=\mc B_{K_3}=\{\{1\},\{2\},\{3\},\{1,2\},\{1,3\},\{2,3\},\{1,2,3\}\}$, and the maximal extended nested collections
\[N_1=\{\{2\},\{2,3\},x_1\},\qquad N_2=\{\{3\},\{1,3\},\{1,2,3\}\}.\]
The extended $\mc B$-forests corresponding to $N_1$ and $N_2$ are $F_1$ and $F_2$ respectively, shown in Figure~\ref{fig:vertices}. We will show how to determine the coordinates of their corresponding vertices in the extended nestohedron $\mc P^{\sq}(\mc B)$.

First consider $N_1$ and corresponding vertex $v=(v_1,v_2,v_3)$. Since $x_1\in N_1$, we have that $v_1=0$. For the remaining coordinates, we have that
\begin{align*}
    v_2&=\abs{\{\{2\},\{1,2\},\{2,3\},\{1,2,3\}\}}-\abs{\{\{2\}\}}+1=4,\\
    v_3&=\abs{\{\{3\},\{1,3\},\{2,3\},\{1,2,3\}\}}-\abs{\{\{2\},\{3\}\}}+1=3,
\end{align*}
so $v=(0,4,3)$. Now consider $N_2$ and corresponding vertex $u=(u_1,u_2,u_3)$. Then,
\begin{align*}
    u_1&=\abs{\{\{1\},\{1,2\},\{1,3\},\{1,2,3\}\}}-\abs{\{\{1\},\{3\}\}}+1=3,\\
    u_2&=\abs{\{\{2\},\{1,2\},\{2,3\},\{1,2,3\}\}}-\abs{\{\{1\},\{2\},\{3\}\}}+1=2,\\
    u_3&=\abs{\{\{3\},\{1,3\},\{2,3\},\{1,2,3\}\}}-\abs{\{\{3\}\}}+1=4.
\end{align*}
Thus, $u=(3,2,4)$.
\begin{figure}[H]
    \centering
    \includegraphics[scale=0.7]{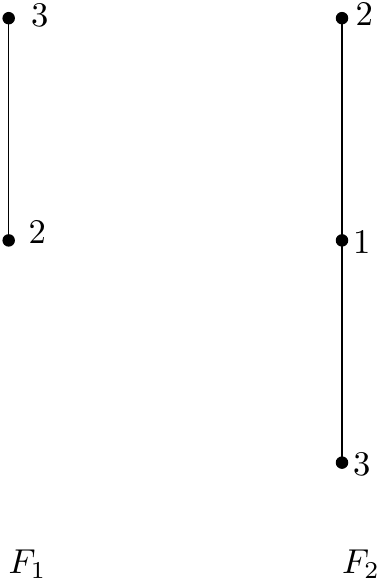}
    \caption{Extended $\mc B$-forests corresponding to $N_1$ and $N_2$.}
    \label{fig:vertices}
\end{figure}

\end{example}

\subsection{\texorpdfstring{$\mc B$}{B}-Partial Permutations and Extended \texorpdfstring{$\mc B$}{B}-Permutations}\label{subsec:parperm}

Next, we study permutations called $\mc B$-partial permutations and extended $\mc B$-permutations, which are analogous to $\mc B$-permutations defined in \cite{PRW}. We will show the following result, which is that when considering a special class of building sets, we can formulate the $h$-polynomial of the extended nestohedron in terms of extended $\mc B$-permutations.

\begin{theorem}\label{thm:ext_h_perm}
For a chordal building set $\mc B$, the $h$-polynomial of the extended nestohedron $\mc P^{\sq}(\mc B)$ is
\[h_{\mc P^{\sq}(\mc B)}(t)=\sum_{w\in \mf S^{\sq}_{n+1}(\mc B)}t^{\des(w)}.\]
\end{theorem}

To show this result, we will introduce $\mc B$-partial permutations and extended $\mc B$-permutations, and then show how these two sets relate to extended $\mc B$-forests. 

A \textbf{partial permutation} of $[n]$ is a permutation $w \in\mf S_S$, for some $S \subseteq[n]$, where $\mf S_S$ denotes the symmetric group acting on $S$. If $S=\varnothing$, then we denote the unique partial permutation with $S$ as its entry set by $()$. Let $\mf P_{n}$ denote the set of partial permutations on $[n]$. Notice that
\[\mf P_n=\bigcup_{S\subseteq[n]}\mf S_S.\]

\begin{example}\label{eg:partial_perm}
The set of partial permutations of $[2]$, $\mf P_2$, consists of the following permutations:
    \[(1,2),(2,1),(1),(2),().\]
\end{example}

We now begin to show how partial permutations relate to extended $\mc B$-forests. Like \cite{PRW}, who define a surjective map $\Psi_{\mc B}:\mf S_n\to\{\mc B\text{-trees}\}$, we recursively define a surjective map $\Psi^{\sq}_{\mc B}$ from all partial permutations to the extended $\mc B$-forests. 

\begin{definition}\label{defn:ext_surj}
Let $\mc B$ be a connected building set on $[n]$, and $S\subseteq [n]$ with $S=\{s_1<\cdots <s_k\}$. Given a permutation $w = (w(s_1), w(s_2),\ldots,w(s_k))\in \mf S_S\subseteq\mf P_n$, one recursively constructs a $\mc B|_S$-forest $F=F(w)$ as follows.

Let $\mc B^{(1)},\ldots, \mc B^{(r)}$ be the connected components of the building set $\mc B|_S$. Restricting $w$ to each of the sets $\mc B^{(i)}$ gives a subword of $w$, say $w_i=(w(s_{i_1}),\dots,w(s_{i_k}))$. For each $i=1,\ldots,r$, we construct a rooted tree $T^{(i)}$ using $w_i$. 

Let the root of $T^{(i)}$ be the node $w(s_{i_k})$. Let $\mc B^{(i)}_1,\ldots,\mc B^{(i)}_{i_r}$ be the connected components of the restriction $\mc B^{(i)}_{\{w(s_{i_1}),\ldots,w(s_{i_{k-1}}))\}}$. Restricting $w_i$ to each of the sets $\mc B^{(i)}_{i_j}$ again gives a subword of $w_i$, to which one recursively applies the construction, attaching each of these subsequent trees to the root node $w(s_{i_k})$. 
\end{definition}

\begin{example}
Consider the building set $\mc B=\mc B_\Gamma$, where $\Gamma = P_8$, and let $S = \{1,2,3,5,7\}$. We will show how to find the extended $\mc B$-forest $F=\Psi^{\sq}_{\mc B}(w)$, where $w=( 3, 7, 1, 5, 2)$.

The three restricted components of the building set $\mc B|_S$ are
\[\mc B^{(1)}=\mc B|_{\{1,2,3\}},\quad \mc B^{(2)}=\mc B|_{\{5\}},\quad\mc B^{(3)}=\mc B|_{\{7\}}.\]
Thus, $F$ will consist of $3$ rooted trees, $T^{(1)}, T^{(2)},$ and $T^{(3)}$. Since $\mc B^{(2)}$ and $\mc B^{(3)}$ each only consist of one singleton element, we know that each of these trees will just be a single node labelled $5$ and $7$ respectively. To construct $T^{(1)}$, consider the subword $w_1 = (3,1,2)$, which is obtained from $w$ by restricting to $\{1,2,3\}$. The last entry of $w_1$ is $2$, so the root of $T^{(1)}$ is the node labelled $2$. Restricting $B^{(1)}$ to $\{1,3\}$, we have two connected components, $\mc B|_{\{1\}}$ and $\mc B|_{\{3\}}$. The two subsequent trees obtained from these connected components are just single nodes, labelled $1$ and $3$. We connect these two trees to the root node, $2$, to obtain $T^{(1)}$. Thus, we obtain the forest $F$, shown in Figure~\ref{fig:surj}.

\begin{figure}[H]
    \centering
    \includegraphics{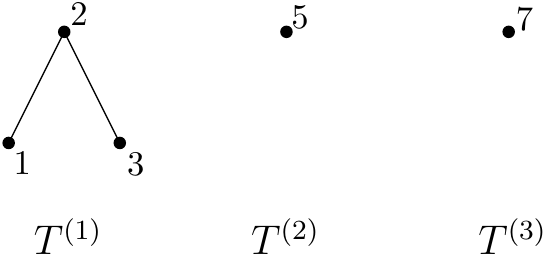}
    \caption{Extended $\mc B$-forest $F=\Psi^{\sq}_{\mc B}(w)$.}
    \label{fig:surj}
\end{figure}
\end{example}

This brings us to the definition of $\mc B$-partial permutations. Such permutations provide a nice section of the surjection $\Psi^{\sq}_{\mc B}$.

\begin{definition}[cf. \cite{PRW}, Definition 8.7]\label{defn:b_partial_perm}
Let $\mc B$ be a building set on $[n]$ and $S=\{s_1,\ldots,s_k\}\subseteq [n]$. Define the set $\mf P_n(\mc B|_S)$ of \textbf{partial $\mc B$-permutations on $S$} as the set of partial permutations $w\in \mf S_S$ such that for any $i\in [k]$, the elements $w(s_i)$ and $\max\{w(s_1),w(s_2),\ldots,w(s_i)\}$ lie in the same connected component of the restricted building set $\mc B|_{\{w(s_1),\ldots,w(s_i)\}}$.

Let $\mf P_n(\cal B)\coloneqq \bigcup_{S\subseteq [n]} \mf P_n(\mc B|_S)$ be the set of \textbf{$\mc B$-partial permutations}.
\end{definition}

Note that when $S=[n]$, the set of partial $\mc B$-permutations on $S$ is exactly the set of non-extended $\mc B$\textbf{-permutations} (see \cite[Definition 8.7]{PRW}. Next, we provide the following lemma which characterizes when a partial permutation is a $\mc B$-partial permutation.

\begin{lemma}[cf. \cite{PRW}, Lemma 8.8]\label{lem:ext_b_perm_iff}
Suppose $S=\{s_1<\cdots <s_k\}\subseteq [n]$ and $\mc B$ is a connected building set on $[n]$. A permutation $w\in \mf S_S$ is a $\mc B$-partial permutation on $S$ if and only if it can be constructed via Algorithm~\ref{algo:bpartialperm}.
\end{lemma}

\begin{algorithm}
	\KwData{Connected building set $\mc B$ on $[n]$ and $S=\{s_1<\cdots<s_k\}\subset[n]$}
	\KwResult{$\mc B$-partial permutation $w=(w(s_1),w(s_2),\ldots,w(s_k))\in\mf S_S$}
    $m\coloneqq s_k$\;
    $\mc C\coloneqq$ connected component of $\mc B|_S$ containing $m$\;
    Pick $w(s_k)$ from $\mc C$\;
    \For{i=k-1,k-2,\ldots,1}{
    $m=\max\{S\setminus\{w(s_k),w(s_{k-1}),\ldots,w(s_{i+1})\}\}$\;
    $\mc C=$ connected component of $\mc B|_{S\setminus\{w(s_k),w(s_{k-1}),\ldots,w(s_{i+1})\}}$ containing $m$\;
    Pick $w(s_i)$ from $\mc C$\;
    }
	return $w$\;
\caption{$\mc B$-partial permutation procedure.}
\label{algo:bpartialperm}
\end{algorithm}

\begin{example}\label{eg:ext_b_perm}
Again consider $\mc B_\Gamma$ with $\Gamma = P_4$ and $S = \{1,2,4\}$. We will construct $w$, a $\mc B$-partial permutation on $S$.

The restricted building set $\mc B|_S$ consists of two connected components, $\mc B|_{\{1,2\}}$ and $\mc B|_{\{4\}}$. Here, $s_k=4$, so we must have that $w(4)=4$. We then have two options for $w(s_{k-1})=w(2)$, since there is only one connected component of $\mc B|_{S\setminus \{w(4)\}} = \mc B|_{\{1,2\}}$, and it contains two elements (including the maximal element of $S\setminus \{w(4)\}$. Let $w(2)=2$. Then we must have that $w(1)=1$. Thus, the partial permutation $w=(1,2,4)$ is a $\mc B$-partial permutation on $S$.
\end{example}

Let $F$ be a collection of rooted trees on $S\subseteq [n]$, with $S=\{s_1<\cdots<s_k\}$. Recall that one can view $F$ itself as a poset. Then, the \textbf{lexicographically minimal linear extension} of $F$ is the permutation $w\in\mf S_S$ such that $w(s_1)$ is a leaf of a tree of $F$ and minimal in the usual order of $\Z$, $w(s_2)$ is the minimal leaf of $F\setminus\{w(s_1)\}$ (the forest $F$ with the vertex $w(s_1)$ removed), $w(s_3)$ is the minimal leaf of $F\setminus\{w(s_1),w(s_2)\}$, etc. 

One can also construct the lexicographically minimal linear extension of $F$ in the following way. 

\begin{lemma}[cf. \cite{PRW}, Lemma 8.9]\label{lem:lex_order_construct}
Let $w$ be the lexicographically minimal linear extension of a rooted forest $F$ on node set $S=\{s_1<\cdots<s_k\}\subseteq[n]$. Then the permutation $w$ can be constructed from $F$, as follows: $w(s_k)$ is the root of the connected component of $F$ that contains the maximal vertex of this forest in the usual order on $\Z$; $w(s_{k-1})$ is the root of the connected component of $F\setminus\{w(s_k)\}$ that contains the maximal vertex of this new forest, etc.

In general, $w(s_i)$ is the root of the connected component of the forest
\[F\setminus\{w(s_k),\ldots,w(s_{i+1})\}\]
that contains the vertex $\max(w(s_1),\ldots,w(s_i))$.
\end{lemma}

\begin{proof}
We induct on the number of vertices of the forest $F$. The base case, when $F$ consists of a single vertex, is clear. Assume for the induction hypothesis that for any forest $F'$ on $k-1$ nodes, its lexicographically minimal linear extension $w'$ can be constructed as by the procedure above.

Let $F'$ be the forest obtained from $F$ by removing the minimal (in the usual order of $\mb Z$) leaf $\ell$. If $w$ is the lexicographically minimal linear extension of $F$, then $w=(\ell,w')$, where $w'$ is the lexicographically minimal linear extension of $F'$ (where $w$ and $w'$ are written as lists). By the induction hypothesis, $w'$ was constructed by the procedure from $F'$. Notice that when constructing $F$, for all $i>1$, the vertex $\ell$ cannot be the root of the connected component of $F\setminus\{w(s_k),\ldots,w(s_{i+1})\}$ that contains the maximal vertex. Thus, the backward procedure described above to obtain $w$ from $F$ produces the same permutation as $w=(\ell,w')$.
\end{proof}

We can now now show a correspondence between extended $\mc B$-forests and $\mc B$-partial permutations. 

\begin{proposition}\label{prop:treeprmbij}[cf. \cite{PRW}, Proposition 8.10]
Let $\mc B$ be a connected building set on $[n]$. The set $\mf P_n(\mc B)$ of $\mc B$-partial permutations is exactly the set of lexicographically minimal linear extensions of the extended $\mc B$-forests. This implies that the set of $\mc B$-partial permutations and the set of extended $\mc B$-forests are in bijection.
\end{proposition}

\begin{proof}
Let $w\in \mf S_S\subseteq\mf P_n$ with $S=\{s_1<\cdots<s_k\}\subseteq [n]$, and let $F=F(w)$ be the corresponding extended $\mc B$-forest constructed, using Definition~\ref{defn:ext_surj}. Notice that for all $i=k,k-1,\ldots,1$, the connected components of the forest $F|_{\{w(s_1),\ldots,w(s_i)\}}$ correspond to the connected components of the building set $\mc B|_{w(s_1),\ldots,w(s_i)\}}$, where corresponding components between the forest and the building set have the same vertex sets. By Lemma~\ref{lem:lex_order_construct}, the partial permutation $w$ is the lexicographically minimal linear extension of $F$ if and only if $w$ is a $\mc B$-partial permutation constructred via Algorithm~\ref{algo:bpartialperm}.
\end{proof}

We now describe the special class of building sets $\mc B$ for which the extended $\mc B$-forests and corresponding $\mc B$-partial permutations agree on their descent numbers. This will allow us to write the $h$-polynomial of the extended nestohedron $\mc P^{\sq}(\mc B)$ as the descent generating function of the $\mc B$-partial permutations.

\begin{definition}\cite[Definition 9.2]{PRW}\label{defn:chordal}
A building set $\mc B$ on $[n]$ is \textbf{chordal} if it satisfies the following condition: for any $I = \{i_1 < \cdots < i_r\} \in \mc B$ and $s = 1,\dots,r,$ the subset $\{i_s, i_{s+1},\dots, i_r\}$ also belongs to $\mc B$.
\end{definition}

Note that if $\mc B$ is a chordal building set on $[n]$ and $S\subseteq [n]$, then $\mc B|_S$ is also chordal. The name for these building sets is justified by the fact that a graphical building set $\mc B_\Gamma$ is chordal if and only if $\Gamma$ is chordal and is labelled in a particular way (see \cite{PRW}, Proposition 9.4 for details). Thus, many nice families of building sets are chordal, such as the building sets $\mc B_\Gamma$ when $\Gamma$ is a path graph, complete graph, or star graph.

It turns out that the extended nestohedron $\mc P^{\sq}(\mc B)$ is a flag simple polytope for $\mc B$ chordal, by the following lemma.

\begin{lemma}\label{lem:chordal_flag}
If $\mc B$ is a chordal building set, then the extended nestohedron $\mc P^{\sq}(\mc B)$ is a flag simple polytope.
\end{lemma}

\begin{proof}
By \cite[Proposition 9.7]{PRW}, the non-extended nestohedron $\mc P(\mc B)$ is a flag simple polytope. By Lemma~\ref{lem:flag}, $\mc P^{\sq}(\mc B)$ is flag as well.
\end{proof}

By Theorem~\ref{thm:gal_ext_flag}, we know that the $\gamma$-vector for this polytope is nonnegative, so it is plausible to give its $\gamma$-vector a combinatorial interpretation. In order to do so, we first have to find a combinatorial interpretation of the $h$-vector. We now give some definitions and technical results to relate extended $\mc B$-forests and $\mc B$-partial permutations for chordal building sets. This will allow us to prove our result about the $h$-polynomial in terms of $\mc B$-partial permutations.

\begin{definition}
Let $S=\{s_1<\cdots<s_k\}$. A \textbf{descent} of a permutation $w\in\mf S_S$ is a pair $(w(s_i),w(s_{i+1}))$ such that $w(s_i)>w(s_{i+1})$. Let $\Des(w)$ be the set of all descents in $w$, and $\des(w)\coloneqq |\Des(w)|$.
\end{definition}

\begin{proposition}[cf. \cite{PRW}, Proposition 9.5]\label{prop:chordal_des_equal}
Let $\mc B$ be a connected chordal building set on $[n]$. Then, for any extended $\mc B$-forest $F$ and the corresponding $\mc B$-partial permutation $w$, one has $\Des(w)=\Des(T)$.
\end{proposition}

\begin{proof}
Let $F$ be an extended $\mc B$-forest with node set $S=\{s_1<\cdots<s_k\}$, and let $w$ the corresponding $\mc B$-partial permutation, which was constructed from $F$ using the procedure given in Lemma~\ref{lem:lex_order_construct}. Fix $i\in\{1,2,\ldots,k-1\}$, and consider the forest
\[F\setminus\{w(s_k),w(s_{k-1}),\ldots,w(s_{i+1})\}.\]
This forest consists of the subtrees $T_1,\ldots,T_r,T'_1,\ldots,T'_s$, where the trees $T_1,\ldots,T_r$ have roots that were children of the node $w(s_{i+1})$ in the original forest $F$, while the trees $T'_1,\ldots,T'_s$ are the remaining trees. We will show that there is exactly one descent edge between $w(s_{i+1})$ and one of the roots of $T_1,\ldots,T_r$ if and only if $w(s_i)>w(s_{i+1})$. 

Let $m=\max\{w(s_1),\ldots,w(s_i)\}$, and first suppose that $m\in T_j$ for some $1\leq j\leq r$; without loss of generality, suppose $m\in T_1$. By Lemma~\ref{lem:lex_order_construct}, it must be that $w(s_i)$ is the root of $T_1$. We show that all of the vertices of the trees $T_2,\ldots,T_r$ are less than $w(s_{i+1})$. If $I=F_{\leq w(s_{i+1})}\subseteq S$, then notice that $I\in\mc B$, by the definition of an extended $\mc B$-forest. Then, if $w(s_{i+1})$ is the maximum element of $I$, it is definitely true that the vertices of $T_2,\ldots,T_r$ are all less than $w(s_{i+1})$, since the node sets of $T_2,\ldots,T_r$ are contained within $I$. 

If $w(s_{i+1})$ is not the maximal element of $I$, then the set $I'=I\cap \{w(s_{i+1})+1, w(s_{i+1})+2,\ldots,n\}$ is an element of $\mc B$, since $\mc B$ is chordal, and $I'$ is nonempty since $m\in I'$. Notice that the vertex set of $T_1$, denote it by $J$, must be an element of the building set by the definition of the extended $\mc B$-forest, and since $m\in I'$ as well as $J$, it follows that $I'\subseteq J$. Thus, all of the nodes of the trees $T_2,\ldots,T_r$ are less than $w(s_{i+1})$. This implies that the only possible descent edge between $w(s_{i+1})$ and a child would have to be between $w(s_{i+1})$ and $w(s_i)$, the root of $T_1$. Notice that this edge is a descent edge if and only if $w(s_{i})>w(s_{i+1})$ which is exactly when a descent occurs in permutation $w$ at indices $i$ and $i+1$.

Now suppose that $m$ is in one of the subtrees that is not a descendant of $w(s_{i+1})$, say $T'_1$. This would imply that $w(s_{i+1})$ is greater than all $w(s_1),\ldots,w(s_i)$. If not, this would imply that $m>w(s_{i+1})$, and that $w(s_{i+1})$ would not have been chosen to be the $(i+1)$-st index of the permutation $w$, as per the algorithm of Lemma~\ref{lem:lex_order_construct}. Since $w(s_{i+1})>w(s_j)$ for all $j=1,\ldots,i$, none of the edges connecting $w(s_{i+1})$ with the roots of $T_1,\ldots,T_r$ could be descent edges and $w(s_i)<w(s_{i+1})$.
\end{proof}

Proposition~\ref{prop:chordal_des_equal} and Proposition~\ref{prop:h_trees} imply the following corollary.

\begin{corollary}\label{cor:h_ext_double_sum}
For a chordal building set $\mc B$ on $[n]$, the $h$-polynomial of $\mc P^{\sq}(\mc B)$ equals
\[h_{\mc P^{\sq}(\mc B)}(t) = \sum_{S \subseteq [n]} \sum_{w\in \mf P_n(\mc B|_S)}t^{\des(w)+n-|S|} ,\]
where $\des(w)$ is the number of descents of $w$.
\end{corollary}

In order to simplify the formula for the $h$-polynomial given in Corollary~\ref{cor:h_ext_double_sum} and not have a double sum, we define the following map from partial permutations to permutations.

Define the map $\varphi_n:\mf P_n\to \mf S_{n+1}$ as follows. For a permutation $w\in \mf S_S\subseteq \mf P_n$ with $S\subseteq [n]$, let $\varphi_n(w)$ be the permutation formed by appending $[n+1]\setminus S$ to the end of $w$ in descending order. The map $\varphi_n$ is an injection into $\mf S_{n+1}$. Let $\mf S^{\sq}_{n+1}(\mc B)\coloneqq \varphi_n(\mf P_n(\mc B))$ be the set of \textbf{extended $\mc B$-permutations}.

Notice that for $w\in \mf P_n(\mc B_S)$,
\[\des(w)+n-|S|=\des(\varphi_n(w)),\]
since the number of descents occurring in the subword of $\varphi_n(w)$ beginning with the entry $n+1$ is exactly equal to $n-|S|$.

This allows us to rewrite the $h$-polynomial of $\mc P^{\sq}(\mc B)$ as the descent-generating function over the set of extended $\mc B$-permutations when $\mc B$ is chordal:
\begin{align*}
    h_{\mc P^{\sq}(\mc B)}(t)=\sum_{w\in\mf{S}^{\sq}_{n+1}(\mc B)} t^{\des(w)},
\end{align*}
thus proving Theorem~\ref{thm:ext_h_perm}

\subsection{\texorpdfstring{$\gamma$}{Gamma}-Vector of Chordal Extended Nestohedra}

Recall that the $\gamma$-vector of a $d$-dimensional simple polytope is given by the $h$-polynomial: $h(t)=\sum h_it^i=\sum\gamma_i t^i(1+t)^{d-2i}$. If $\mc B$ is a chordal building set, then the extended nestohedron $\mc P^{\sq}(\mc B)$ is flag by Lemma~\ref{lem:chordal_flag}. We showed in the previous section that flag extended nestohedra have nonnegative $\gamma$-vectors. It is therefore possible to give $\gamma$-vectors combinatorial interpretations for such polytopes. In this section, we find such an interpretation for the $\gamma$-vector of chordal extended nestohedra. To do so, we use the technique used in \cite{PRW}, in which the analogous result for chordal nestohedra is shown.

The general outline of the technique is as follows. Suppose $P$ is a $d$-dimensional simple polytope that has a $h$-polynomial that has a combinatorial interpretation of the form $h_P(t)=\sum_{a\in A}t^{f(a)}$, where $f(a)$ is some statistic on the set $A$. Then, we show that $A$ can be partitioned into classes such that
\[\sum_{a\in C} t^{f(a)}=t^r(1+t)^{d-2r},\]
for each class $C\subseteq A$ and for some $r\in\mathbb{N}$. If $\widehat{A}\subseteq A$ denotes the set of representatives of each class where $f(a)$ takes its minimal value, then
\[\gamma_P(t)=\sum_{a\in\widehat{A}}t^{f(a)}.\]

For us, $A$ is the set of extended $\mc B$-permutations, and $f(a)$ is the descent number for permutations. Like \cite{PRW} who define operations on $\mc B$-permutations, we will define a series of operations on extended $\mc B$-permutations that will allow us to partition such permutations, and then we will able to give the $\gamma$-polynomial as the descent-generating function over a subset of extended $\mc B$-permutations.

First, we give several preliminary definitions related to the ``topography'' of permutations. For $w\in \mf S_{n+1}$, a \textbf{final descent} is when $w(n)>w(n+1)$; a \textbf{double descent} is pair of consecutive descents, i.e., a triple $w(i)>w(i+1)>w(i+2)$. A \textbf{peak} of $w$ is an entry $w(i)$ for $1\leq i\leq n+1$ such that $w(i-1)< w(i) >w(i+1)$, where here and for the rest of this section (unless otherwise specified), we set $w(0)=w(n+2)=0$, so peaks can occur at indices $1$ and $n+1$. A \textbf{valley} of $w$ is an entry $w(i)$ for $1<i<n$ such that $w(i-1)>w(i)<w(i+1)$. The \textbf{peak-valley sequence} of $w$ is the subsequence in $w$ formed by all peaks and valleys. An entry $w(i)$ is an \textbf{intermediary entry} if $w(i)$ is neither a peak nor a valley. We say that $w(i)$ is an \textbf{ascent-intermediary entry} if $w(i-1)<w(i)<w(i+1)$, and it is a \textbf{descent-intermediary entry} if $w(i-1)>w(i)>w(i+1)$.

One can graphically represent a permutation $w\in\mf S_{n+1}$ as a ``mountain range'' $M_w$ in the following way. Plot the points $(0,0), (1,w(1)),(2,w(2)),\ldots,(n+1,w(n+1)),(n+2,0)$ on $\R^2$, and then connect points by straight line intervals. Now, peaks in $w$ correspond to local maxima of $M_w$, valleys correspond to local minima, and ascent-intermediary (resp. descent-intermediary) entries correspond to points that are on increasing (resp. decreasing) slopes of $M_w$. 

\begin{example}\label{eg:perm_topography}
Consider the permutation $w = ( 2, 4, 1, 6, 5, 3)\in\mf S_{6}$. The mountain range $M_w$ is shown in Figure~\ref{fig:mt_range}. Notice that $w$ has a final descent $(3,0)$ and two double descents, $(6,5,3)$ and $(5,3,0)$. It has peaks at $4$ and $6$, and a valley at $1$. The descent-intermediary entries are $5$ and $3$, while the only ascent intermediary entry is $2$.

\begin{figure}[H]
    \centering
    \includegraphics{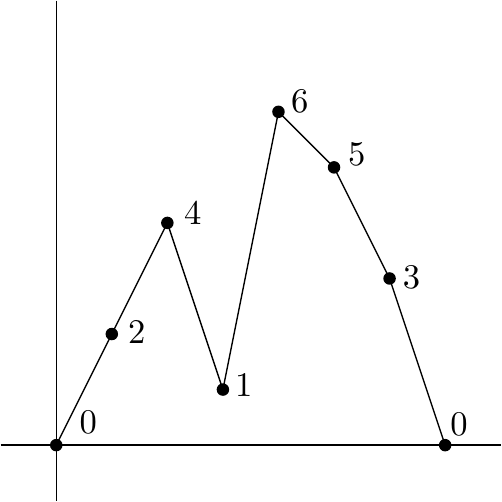}
    \caption{Mountain range $M_w$.}
    \label{fig:mt_range}
\end{figure}
\end{example}

We now begin to define some operations on permutations. The \textbf{leap operations} $L_a$ and $L_a^-$, as well as powers of leap operations, were first introduced \cite[Section 11.2]{PRW}, and they allow us to define operations on extended $\mc B$-permutations later on. The leap operations are defined as follows. The permutation $L_a(w)$ is obtained from permutation $w$ by removing an intermediary node $a$ from the $i$-th position in $w$ and inserting $a$ in the position between $w(j)$ and $w(j+1)$ where $j$ is the minimal index such that $j>i$ and $a$ is between $w(j)$ and $w(j+1)$. Similarly, the permutation $L_a^-(w)$ is obtained from removing $a$ from the $i$-th position and inserting $a$ between $w(k)$ and $w(k+1)$ where $k$ is the maximum index such that $k<i$ and $a$ be between $w(k)$ and $w(k+1)$.

Informally, if $w\in\mf S_{n+1}$ and $a$ is an intermediary entry, then the permutation $L_a(w)$ is obtained from $w$ by moving an intermediary point $a$ on the mountain range $M_w$ directly to the right until it hits the next slope of $M_w$. Likewise, the permutation $L_a^{-}(w)$ is obtained from $w$ by moving $a$ directly to the left until it hits the next slope of $M_w$.

\begin{example}\label{eg:leaps}
Again consider the permutation $w = ( 2, 4, 1, 6, 5 ,3)\in\mf S_6$ from Example~\ref{eg:perm_topography}. The permutation $L_2(w)$ is obtained by moving node $2$ to the right until hitting the next slope, so the resulting permutation is $L_2(w) = ( 4, 2, 1, 6, 5, 3)$. The mountain range for the this permutation is shown in Figure~\ref{fig:leap_pos}. The permutation $L^{-}_3(w)$ is obtained by moving node $3$ to the left until hitting the next slope, so the resulting permutation is $L^{-}_3(w) = ( 2, 4, 1, 3, 6, 5)$. The mountain range for this permutation is shown in Figure~\ref{fig:leap_neg}.

\begin{figure}[H]
\centering
\subfigure[$L_2(w)$]{%
\label{fig:leap_pos}%
\includegraphics[scale=0.8]{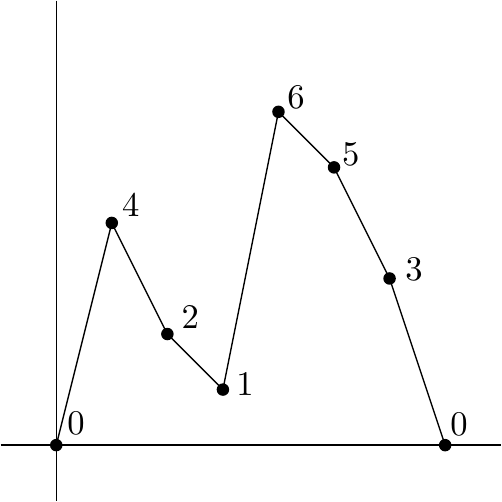}}\qquad \qquad
\subfigure[$L_3^{-}(w)$]{%
\label{fig:leap_neg}%
\includegraphics[scale=0.8]{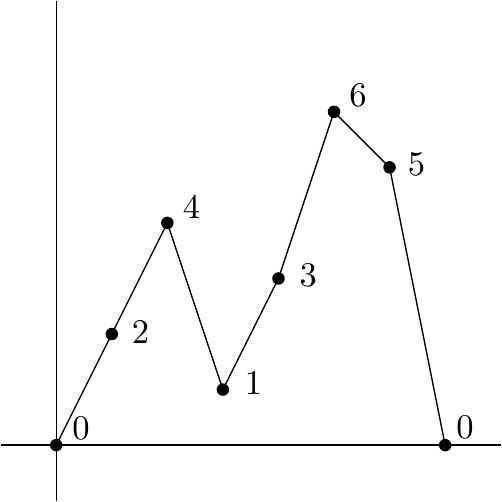}}%
\caption{Examples of leaps on $w$.}
\label{fig:leaps}
\end{figure}
\end{example}

Next, we define powers of the leap operations:
\[L_a^r\coloneqq\begin{cases}
(L_a)^r,\quad&\text{for $r\in\Z_{\geq 0}$,}\\
(L_a^-)^{-r},\quad&\text{for $r\in\Z_{\leq0}$.}
\end{cases}\]
In words, if $r$ is positive, then for a permutation $w$ and intermediary entry $a$, the permutation $L_a^r(w)$ is obtained from $w$ by moving $a$ to the right until it hits the $r^{\text{th}}$ slope from its original slope; if $r$ is negative, then the permutation $L_a^r(w)$ is obtained from $w$ by moving $a$ to the left until it hits the $-r^{\text{th}}$ slope from its original slope.

Notice that $L^r_a(w)$ is only defined whenever $r$ is in a certain integer interval, since there are finitely many slopes to the left and to the right of $a$'s current slope. Let $[r_{\min},r_{\max}]$ denote this interval. In addition, if $a$ is an ascent-intermediary entry in $w$, then $a$ is still ascent-intermediary in $L_a^r(w)$ for even $r$, but it is descent-intermediary for odd $r$. Similarly, if $a$ is a descent-intermediary entry in $w$, then $a$ is still descent intermediary in $L_a^r(w)$ for even $r$, but it is ascent-intermediary for odd $r$.

We now want to ensure that if $w\in\mf{S}^{\sq}_{n+1}(\mc B)$ is an extended $\mc B$-permutation, then there exists some $r$ such that $L_a^r(w)$ is an extended $\mc B$-permutation as well. To do so, we give the following reformulation of extended $\mc B$-permutations for chordal building sets.

For a permutation $w\in\mf S_{n+1}$ and $a\in[n+1]$ with $w(i)=a$, let
\[\{w\nwarrow a\}\coloneqq\{w(j)\mid j\leq i,w(j)\geq a\}\]
be the set of entries in $w$ that occur before $a$ but are greater than or equal to $a$; this set includes $a$ itself. In the graph $M_w$, the set $\{w\nwarrow a\}$ is the set of entries of $w$ that are located above and to the left of the point corresponding to $a$. In addition, if $n+1=w(k)$, then let $w^{\sq}$ denote the subword $(w(1),\ldots,w(k-1))$.

By Definition~\ref{defn:b_partial_perm} and our definition for $\mf{S}^{\sq}_{n+1}(\mc B)$, the set $\mf{S}^{\sq}_{n+1}(\mc B)$ is the set of permutations $w$ such that for all $i=1,\ldots,\abs{w^{\sq}}$, there exists $I\in \mc B$ such that both $w(i)$ and $\max\{w(1),\ldots,w(i)\}$ are elements of $I$, and $I\subseteq \{w(1),\ldots,w(i)\}$. If $\mc B$ is chordal, then $I'\coloneqq I\cap [w(i),\infty]$ is an element of $\mc B$. This implies that $w(i)$ and $\max\{w(1),\ldots,w(i)\}\in I'$ and $I'\subseteq\{w(1),\ldots,w(i)\}$. Also notice that $\max\{w(1),\ldots,w(i)\}=\max\{w\nwarrow w(i)\}$. Thus, we can reformulate our definition for $\mf {S}^{\sq}_{n+1}(\mc B)$ if $\mc B$ is a chordal building set. A similar reformulation for non-extended $\mc B$-permutations is given in \cite[Lemma 11.8]{PRW}.

\begin{definition}\label{defn:extbperm2}
Let $\mc B$ be a chordal building set on $[n]$. Then the set of extended $\mc B$-permutations $\mf{S}^{\sq}_{n+1}(\mc B)$ is the set of permutations $w\in\mf S_{n+1}$ such that for any $a\in w^{\sq}$, the elements $a$ and $\max\{w\nwarrow a\}$ are in the same connected component of $\mc B|_{\{w\nwarrow a\}}$. In other words, there exists $I\in\mc B$ such that for all $a\in I$, we have that $\max\{w\nwarrow a\}\in I$ and $I\subseteq \{w\nwarrow a\}$.
\end{definition}

If $w\in \mf{S}^{\sq}_{n+1}(\mc B)$ and $a$ is an intermediary entry of $w$, then there are $2$ possible reasons why the permutation $u=L^r_a(w)$ may no longer be an element of $\mf{S}^{\sq}_{n+1}(\mc B)$:
\begin{itemize}
    \item[A) ]if $a\in u^{\sq}$, and the entries $a$ and $\max\{u\nwarrow a\}$ are in different connected components of $\mc B|_{\{u\nwarrow a\}}$, i.e., there does not exist an element $I\in\mc B|_{\{u\nwarrow a\}}$ such that $a,\max\{u\nwarrow a\}\in I$, or
    \item[B) ]if another entry $b\neq a$ is in $u^{\sq}$ and $\max\{u\nwarrow b\}$ are in different connected components of $\mc B|_{\{u\nwarrow b\}}$. 
\end{itemize}
We call these two types of failures \textbf{A-failure} and \textbf{B-failure}. Note that these terms have slightly different definitions given in \cite{PRW}, as they are using them in the context of non-extended $\mc B$-permutations.

Next, we have the following technical lemma on when A- and B-failures can and cannot occur.

\begin{lemma}[cf. \cite{PRW}, Lemma 11.9]\label{lem:ABfailures} 
Suppose $w\in \mf{S}^{\sq}_{n+1}(\mc B)$ and $a$ is an intermediary entry of $w$.
\begin{enumerate}[(i)]
    \item For left leaps $u=L_a^r(w), r<0$, one can never have a B-failure.
    \item For the maximal left leap $u=L_a^{r_{\min}}(w)$, one cannot have an A-failure.
    \item For the maximal right leap $u=L_a^{r_{\max}}(w)$, one cannot have an A-failure.
    \item Let $u=L^r_a(w)$ and $u'=L_a^{r+1}(w)$ be two adjacent leaps such that $a$ is descent-intermediary in $u$ (implying that $a$ is ascent-intermediary in $u'$). Then there is an A-failure in $u$ if and only if there is an A-failure in $u'$.
\end{enumerate}
\end{lemma}

\begin{proof}
\begin{enumerate}[(i)]
    \item Let $b\neq a$ be an entry of $w^{\sq}$. Then, since $w\in \mf{S}^{\sq}_{n+1}(\mc B),$ there exists a subset $I\in\mc B$ such that $b,\max\{w\nwarrow b\}\in I$ and $I\subseteq \{w\nwarrow b\}$. The same subset $I$ works for $u$ since $\{u\nwarrow b\}=\{w\nwarrow b\}$ or $\{u\nwarrow b\}=\{w\nwarrow b\}\cup \{a\}$, and $a\neq\max\{u\nwarrow b\}$, since this would imply $a$ became a peak under the operation, which is impossible.
    
    \item In this case, $a$ is greater than all preceding entries in $u$, so $a=\max\{u\nwarrow a\}$. Thus, the singleton element $\{a\}$ satisfies the necessary conditions for an element of $\mc B$.
    
    \item Notice that $a$ becomes an element after $n+1$, i.e., $a\notin u^{\sq}$, so it is impossible to consider an A-failure for $a$.
    
    \item All of the entries between $a$ in $u$ and $a$ in $u'$ are less than $a$, so $\{u\nwarrow a\}=\{u'\nwarrow a\}$. Thus, $u$ has an A-failure if and only if $u'$ has an A-failure. 
\end{enumerate}
\end{proof}

This next lemma guarantees that for an extended $\mc B$-permutation $w$ and intermediary entry $a$, there exists at least some $r\in[r_{\min},r_{\max}]$ such that $L_a^r(w)$ is still an extended $\mc B$-permutation. 

\begin{lemma}[cf. \cite{PRW}, Lemma 11.7]\label{lem:leaps}
Let $\mc B$ be a chordal building set on $[n]$. Suppose that $w\in \mf{S}^{\sq}_{n+1}(\mc B)$.
\begin{enumerate}[(i)]
    \item If $a$ is an ascent-intermediary entry in $w$, then there exists an odd positive integer $r>0$ such that $L_a^r(w)\in\mf{S}^{\sq}_{n+1}(\mc B)$ and $L_a^s(w)\notin\mf{S}^{\sq}_{n+1}(\mc B)$ for all $0<s<r$.
    \item If $a$ is an descent-intermediary entry in $w$, then there exists an odd negative integer $r<0$ such that $L_a^r(w)\in\mf{S}^{\sq}_{n+1}(\mc B)$ and $L_a^s(w)\notin\mf{S}^{\sq}_{n+1}(\mc B)$ for all $r<s<0$.
\end{enumerate}
\end{lemma}

\begin{proof}
\begin{enumerate}[(i)]
    \item We will show that there exists a permutation $u$ without a B-failure, and then that among permutations without B-failures, there exists one without an A-failure. First suppose that there exists an entry $b\neq a$ in the permutation $w$ such that $b\in w^{\sq}$, and that $b$ and $m=\max\{w\nwarrow b\}$ are in different connected components of $\mc B|_{\{w\nwarrow b\}\setminus\{a\}}$. It cannot be that $a\notin\{w\nwarrow b\}$, since this would imply that $b$ and $m$ are in different connected components of $\mc B|_{\{w\nwarrow b\}}$, contradicting the fact that $w$ is an extended $\mc B$-permutation. Thus, we have that $a\in\{w\nwarrow b\}$, implying that $b<a$ and that $a$ is to the left of $b$ in $w$. Let $b$ be the leftmost entry of $w$ that satisfies the hypothesis for this case. Consider the permutation $u=L_a^t(w)$ for some $t>0$. If entry $a$ is to the right of $b$ in permutation $u$, then $u$ would have a B-failure; if entry $a$ stays to the left of $b$ in $u$, then there would not be a B-failure. Since $a$ is ascent-intermediary and $b<a$, we have that $a$ stays to the left of $b$ in $L_a^1(w)$, so a permutation $u$ without a B-failure exists in this case.
    
    Let $u=L_a^t(w)$ be the maximal right leap such that entry $a$ stays to the left of $b$, so $t>0$ is maximized; we have that $a,b\in u^{\sq}$. Notice that all entries in $u$ that are between the indices of $a$ and $b$ are less than $a$; otherwise, $t$ would not be maximal. Then, $m=\max\{u\nwarrow a\}=\max\{w\nwarrow b\}$. Since $w\in\mf{S}^{\sq}_{n+1}(\mc B)$, there exists an $I\in\mc B$ such that $b,m\in I$ and $I\subseteq\{w\nwarrow b\}$. Notice that we chose $b$ to not be in the same connected component as $m$ in $\mc B|_{\{w\nwarrow b\}\setminus\{a\}}$, so we necessarily have that $a\in I$ as well. Then, $I'\coloneqq I\cap [a,\infty]$ is an element of $\mc B$, with $a,m\in I'$ and $I'\subseteq\{u\nwarrow a\}$, implying that there is no A-failure in $u$ and so $u\in\mf{S}^{\sq}_{n+1}(\mc B)$.
    
    If no such entry $b$ in $w$ as above exists, then none of the permutations $L_a^r(w)$ have B-failures. By (iii) of Lemma~\ref{lem:ABfailures}, the permutation $L_a^{r_{\max}}(w)$ does not have an A-failure, so it is an extended $\mc B$-permutation.
    
    In both cases, there exists a positive integer $r$ such that $L_a^r(w)\in\mf{S}^{\sq}_{n+1}(\mc B)$ and for all $0<s<r$, only A-failures are possible for $L_a^s(w)$. Pick $r$ to be minimal. Then, $r$ must be such that $a$ is descent-intermediary in $L_a^r(w)$; otherwise, we could have chosen $r-1$ by 4) of Lemma~\ref{lem:ABfailures}. Thus, $r$ must be odd.
    
    \item By parts (i) and (ii) of Lemma~\ref{lem:ABfailures}, there necessarily exists a negative $r$ such that $L_a^r(w)\in\mf{S}^{\sq}_{n+1}(\mc B)$, namely $r=r_{\min}$. Choose $r$ with minimal possible absolute value such that $L_a^r(w)\in \mf{S}^{\sq}_{n+1}(\mc B)$. Notice that $r$ must necessarily be odd. If not, then $a$ would be descent intermediary in $L_a^r(w)$ and there would be no A-failure; by (iv) of Lemma~\ref{lem:ABfailures}, this would imply that $L_a^{r+1}(w)$ would also have no A-failure, so $L_a^{r+1}(w)\in\mf{S}^{\sq}_{n+1}(\mc B)$ with $|r+1|<|r|$, contradicting the minimality of $|r|$.
\end{enumerate}
\end{proof}
We are now able to define an operation that turns one extended $\mc B$-operation into another.

\begin{definition}[cf. \cite{PRW} Definition 11.10]\label{defn:ext_bhop}
An \textbf{extended $\mc B-$hop operations} $\mc B H^{\sq}_a$ is defined as follows. For $w\in\mf{S}^{\sq}_{n+1}(\mc B)$ with an ascent-intermediary (resp., descent-intermediary) entry $a$, the permutation $\mc B H^{\sq}_a(w)$ is the right leap $u=L_a^r(w)$, with $r>0$ (resp., the left leap $u=L_a^r(w)$ with $r<0$) with minimal possible $|r|$ such that $u\in\mf{S}^{\sq}_{n+1}(\mc B)$.
\end{definition}

Informally, the permutation $\mc B H^{\sq}_a(w)$ is obtained from $w$ by moving the point $a$ on the graph $M_w$ to the right if $a$ is ascent-intermediary in $w$, or to the left if $a$ is descent-intermediary, until $a$ hits a slope such that the new permutation is an extended $\mc B$-permutation. It is possible that the point $a$ passes through several slopes.

\begin{example}\label{eg:ext_b_hop}
Consider the building set
\[\mc B=\{\{1\},\{2\},\{3\},\{4\},\{1,4\},\{3,4\},\{1,3,4\},\{2,3,4\},\{1,2,3,4\}\}.\]
An example of an extended $\mc B$-permutation is $w=(4, 1, 3, 5, 2)$. The extended $\mc B$-hop on $w$ for the entry $2$ is $\mc B H^{\sq}_2(w)=L_2^{-3}(w)$; we require three left leaps for this extended $\mc B$-hop operation, since neither $L_2^{-1}(w)$ nor $L_2^{-1}(w)$ are extended $\mc B$-permutations. Thus, $\mc B H^{\sq}_2(w)$ corresponds to moving the node $2$ to the left, passing through two slopes, as shown in Figure~\ref{fig:extbhop}.

\begin{figure}
\centering
\subfigure[$w$]{%
\label{fig:extbhopw}%
\includegraphics{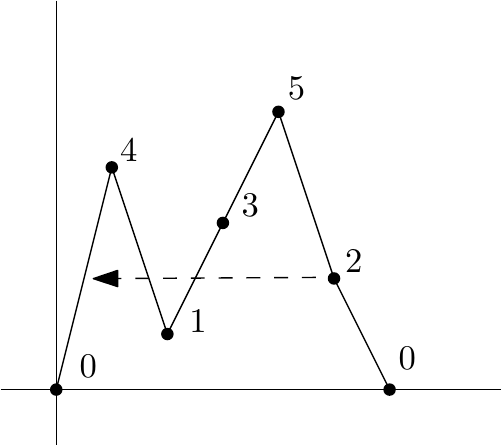}}\qquad \qquad
\subfigure[$\mc B H^{\sq}_2(w)$]{%
\label{fig:extbhopw1}%
\includegraphics{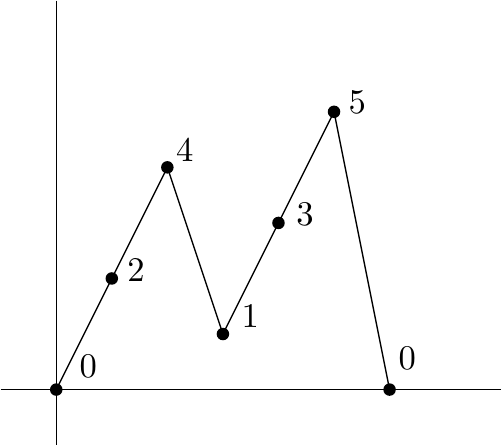}}%
\caption{Extended $\mc B$-hop on entry $2$ of $w$.}
\label{fig:extbhop}
\end{figure}

\end{example}

\begin{lemma}\label{lem:bhop}
For any extended $\mc B$-permutation $w$ and intermediary entry $a$ of $w$, the extended $\mc B$-hop operation $\mc B H^{\sq}_a(w)$ is well-defined. If $a$ is an ascent-intermediary entry in $w$, then $a$ is descent-intermediary in $\mc B H^{\sq}_a(w)$; if $a$ is descent-intermediary in $w$, then $a$ is ascent-intermediary in $\mc B H^{\sq}_a(w)$. In addition, $(\mc B H^{\sq}_a)^2(w)=w$.
\end{lemma}

\begin{proof}
All follows from Lemma~\ref{lem:leaps}.
\end{proof}

This next result shows that the extended $\mc B$-hop operations pairwise commute with each other. A similar statement for non-extended $\mc B$-hop operations is given in \cite{PRW}, and the proof for our extended case follows in the same way as the non-extended case.

\begin{lemma}[cf. \cite{PRW} Lemma 11.12]
Let $w$ be an extended $\mc B-$permutation with two intermediary entries $a$ and $b$. Then $\mc B H^{\sq}_a(\mc B H^{\sq}_b(w))=\mc B H^{\sq}_b(\mc B H^{\sq}_a(w))$.
\end{lemma}

For a set of extended $\mc B$-permutations with the same peak-valley sequence, the extended $\mc B$-hop operations $\mc B H^{\sq}_a$ generate the action of the group $(\Z/2\Z)^m$, where $m$ is the number of intermediary entries in any permutation of this set. 
We say that two extended $\mc B-$permutations are  \textbf{extended $\mc B$-hop equivalent} if they can be obtained from each other via a series of extended $\mc B$-hop operations $\mc B H^{\sq}_a$ for various intermediary entries $a$. This allows us to partition the set of extended $\mc B$-permutations into extended $\mc B$-hop equivalence classes. Notice that each class has exactly one permutation with no descent-intermediary entries since, if $w\in\mf{S}^{\sq}_{n+1}(\mc B)$ had a descent-intermediary entry $a$, then one can always apply the extended $\mc B$-hop operation to make it an ascent-intermediary entry.

Let $\widehat{\mf S}_{n+1}$ denote the set of permutations in $\mf S_{n+1}$ that do not have any final descents or double descents. Notice that this is equivalent to the set of permutations that have no descent-intermediary entries. For $\mc B$ a chordal building set on $[n]$, let $\widehat{\mf{S}}^{\sq}_{n+1}(\mc B)\coloneqq \widehat{\mf S}_{n+1}\cap \mf{S}^{\sq}_{n+1}(\mc B).$

We are now able to give the combinatorial interpretation for the $\gamma$-vector for chordal building sets.

\begin{theorem}[cf. \cite{PRW}, Theorem 11.6]\label{thm:gamma_chordal}
For a connected chordal building set $\mc B$ on $[n]$, the $\gamma$-polynomial of the extended nestohedron $\mc P^{\sq}_{\mc B}$ is the descent-generating function for the permutations in $\widehat{\mf S}^{\sq}_{n+1}(\mc B)$:
\[\gamma_{\mc P^{\sq}(\mc B)}(t) = \sum_{w\in \widehat{\mf S}^{\sq}_{n+1}(\mc B)}t^{\des(w)}.\]
\end{theorem}

\begin{proof}
For an extended $\mc B$-permutation $w\in\widehat{\mf S}^{\sq}_{n+1}(\mc B)$, the descent generating function of the extended $\mc B$-hop-equivalence class $C$ of $w$ is 
\[\sum_{u\in C}t^{\des(u)} = t^{\des(w)}(t-1)^{n- 2\des(w)}.\]
There is exactly one representative permutation from each extended $\mc B$-hop equivalence class without descent-intermediary entries in the set $\widehat{\mf S}^{\sq}_{n+1}$. Thus the $h$-polynomial of the extended nestohedron $\mc{P}^{\sq}_{\mc B}$ is

\[h_{\mc P^{\sq}(\mc B)}(t)=\sum_{w\in\mf S^{\sq}_{n+1}(\mc B)}t^{\des(w)}=\sum_{w\in\widehat{\mf S}^{\sq}_{n+1}(\mc B)}t^{\des(w)}(t-1)^{n-2\des(w)}.\]
By definition of the $\gamma$-polynomial, we are done.
\end{proof}

Thus, we have a combinatorial interpretation of the $\gamma$-polynomial of extended chordal nestohedra in terms of the descent number of our special class of extended $\mc B$-permutations.
\section{A Weak Order on Partial Permutations}\label{sec:parperm}

In this section, we define a partial order on partial permutations, $\mf P_n$, that is analogous to the weak Bruhat order on the symmetric group. We show that this partial order is a lattice, and that any linear extension of the partial order on $\mf P_n$ gives a shelling of the stellohedron. In addition, we define partial orders on maximal non-extended and extended nested collections whose Hasse diagrams can be realized by an acyclic directed graph on the $1$-skeleta of nestohedra and extended nestohedra. When this poset is a lattice, it has further nice properties that can applied to the partial order on partial permutations $\mf P_n$.

First, we give some preliminary definitions on partially ordered sets (posets). Let $P$ be a poset. If $u\leq v$ in $P$ and $u\leq z\leq v$ implies that $u=z$ or $z=v$, then $u\lessdot v$ is a \textbf{cover relation}. A poset $P$ is a \textbf{lattice} if for any pair of elements $x,y\in P$, there exist a unique least upper bound and a unique greatest lower bound for $x$ and $y$ which is also contained in $P$. The former is called the \textbf{join} of $x$ and $y$, denoted $x\vee y$, and the latter is called the \textbf{meet} of $x$ and $y$, denoted $x\wedge y$. An \textbf{interval} $[u,v]$ of the poset $P$ is a subposet of elements $z\in P$ such that $u\leq z\leq v$. For any poset $P$, the \textbf{dual poset}, denoted $P^*$, is the poset with $u\leq v$ in $P^*$ if and only if $v\leq u$ in $P$.


Next, we define the weak Bruhat order on the symmetric group $\mf S_n$. For any permutation $\pi\in\mf S_n$, let
\[\inv(\pi)\coloneqq\{(i,j)\mid 1\leq i< j\leq n\text{ and }\pi(i)>\pi(j)\},\]
denote the \textbf{inversion set} of $\pi$. The \textbf{weak Bruhat order} on $\mf S_n$ is the partial order given by containment of inversion sets: $\pi\leq \sigma$ if and only if $\inv(\pi)\subseteq\inv(\sigma)$. It is well-known that the weak Bruhat order on $\mf S_n$ is a lattice.

Recall that a \textbf{partial permutation} $\pi\in\mf P_n$ is an ordered sequence $\pi=(a_1,\ldots,a_r)$ with $a_i\in[n]$ for all $i$, and $r=0,\ldots,n$. If $r=0$, then $\pi$ is the empty permutation, which we denote by $()$. Recall the injective map
\[\varphi_n:\mf P_n\to\mf S_{n+1}.\]
Let $\tld \pi \coloneqq \varphi_n(\pi)$. We can use the weak Bruhat order on $\mf S_{n+1}$ to induce a partial order on $\mf P_n$ as follows.

\begin{definition}\label{def:partialweak}
Let $\pi,\sigma \in \mf P_n$ be two partial permutations. We say that $\pi \leq \sigma$ in the \textbf{partial weak Bruhat order} if and only if  $\tld \pi \leq \tld \sigma$ in the weak Bruhat order on $\mf S_{n+1}$.
\end{definition}

\begin{figure}[H]
\centering
\subfigure[$\mf P_2$]{%
\label{fig:P2}%
\includegraphics[scale=0.6]{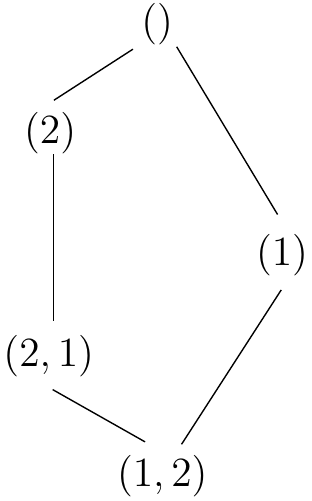}}\qquad \qquad
\subfigure[$\mf P_3$]{%
\label{fig:P3}%
\includegraphics[scale=0.6]{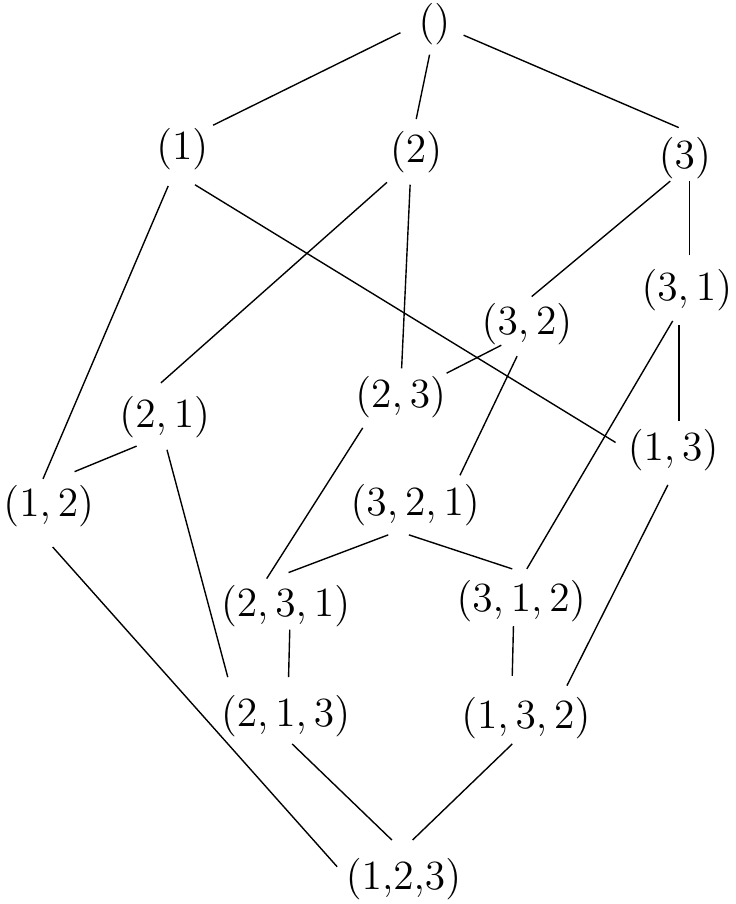}}%
\caption{Partial weak Bruhat order on $\mf P_2$ and $\mf P_3$.}
\label{fig:partialweak}
\end{figure}

\begin{remark}\label{rem:poset_isomorphism_BM}
In \cite{barnard2018lattices}, Barnard and McConville define a partial order $L_G$ on maximal nested collections for the graphical building set $\mc B_G$ that we extend later in this section. We note that the partial weak Bruhat order on $\mf P_n$ is the dual poset of the poset $L_G$, where $G$ is the star graph $K_{1,n}$. In particular, the partial permutation $\pi$ corresponds to the maximal nested collection of $\mc B_G$ with $\mc B$-permutation $(a_1,\ldots,a_{n+1})$, where $\tld{\pi}=(a_{n+1},\ldots,a_1)$.
\end{remark}

Since the Hasse diagram of the weak Bruhat order on the symmetric group can be realized using the $1$-skeleton of the permutohedron, $\mc P(\mc B_{K_n})$, one can think of the partial weak Bruhat order as the partial order that can be realized using the analogous extended nestohedron, namely the stellohedron $\mc P^{\sq}(\mc B_{K_n})=\mc P(\mc B_{K_{1,n}})$.

Barnard and McConville show that if a graph $G$ is \textbf{right-filled}, i.e., if the building set $\mc B_G$ is chordal, then $L_G$ is isomorphic to a subposet of the weak Bruhat order that is a lattice and a meet semilattice quotient of the weak Bruhat order (\cite[Theorem 4.10]{barnard2018lattices}). Thus, since the building set for the star graph is chordal and $\mf P_n$ is dual to $L_G$ by Remark~\ref{rem:poset_isomorphism_BM}, we have the following result.

\begin{theorem}\label{cor:lattice_partialperm}
The partial weak Bruhat order on $\mf P_n$ is a lattice and a join semilattice quotient of the weak Bruhat order on $\mf S_{n+1}$.
\end{theorem}

We also note that although the weak Bruhat order is semi-distributive, congruence normal, and congruence uniform, the partial weak Bruhat order on $\mf P_n$ does not have any of these properties, and is not a lattice quotient of the weak Bruhat order. 





\subsection{Shellings of the Dual of the Stellohedron}

The partial order defined on partial permutations provides us with shellings for the dual of the stellohedron. Throughout this subsection, let $\Delta$ be a pure $d$-dimensional finite simplicial complex. We will be considering linear orderings $F_1,F_2,\ldots$ of its facets. Given such an ordering, we define $\Delta_k \coloneqq \bigcup_{i = 1}^k C_i$ for $k\geq 1$ and let $\Delta_0 = \varnothing$.

\begin{definition}\label{def:shelling}
An $n$-dimensional simplicial complex $\Delta$ with $r$ facets is \textbf{shellable} if its facets can be arranged into a linear ordering $F_1,F_2,\ldots, F_r $ such that $\Delta_{k-1}\cap F_k$ is pure of dimension $n-1$ for all $2 \leq k \leq r$. Such an ordering is called a \textbf{shelling} of $\Delta$.
\end{definition}

For shellings of the dual of the permutohedron, $\mc N(\mc B_{K_n})$, Bj\"orner proved the following (in the more general setting of a the weak order in an arbitrary Weyl group).

\begin{theorem}[{\cite[Theorem 2.1]{bjorner1984some}}]
Let $\mc B = \mc B_{K_n}$, and $\mc N(\mc B)$ be nested complex for $\mc B$, whose facets $F_\pi$ are labeled by the permutations $\pi \in \mf S_n$.
If a total ordering $\pi_1 \leq \cdots \leq \pi_{n!}$ is a linear extension of the weak Bruhat order on $\mf S_n$, then $F_{\pi_1},\ldots,F_{\pi_{n!}}$ is a shelling for $\mc N(\mc B)$.
\end{theorem}

We will prove this theorem for the extended case, where the analogues of $\mc N(\mc B_{K_n})$, permutations $\mf S_n$, and weak Bruhat order are respectively $\mc N^{\sq}(\mc B_{K_n})$, partial permutations $\mf P_n$, and the partial weak order (see Definition~\ref{def:partialweak}). First, we describe how the faces of the dual of the stellohedron are indexed by partial permutations.

Recall the surjective map $\Psi_{\mc B}^{\sq}$ from all partial permutations to extended $\mc B$-forests (Definition~\ref{defn:ext_surj}). Using this map, we can construct a maximal extended nested set of $\mc B$ from a partial permutation $w=(w(s_1),w(s_2),\ldots,w(s_r))\in\mf P_n$, where $S=\{s_1<\cdots<s_r\}\subseteq[n]$. For $i=1,\ldots,r$, let
\[\mc B_i\coloneqq\mc B|_{\{w(s_1),\ldots,w(s_{r-i+1})\}}.\]
Then the elements of the building set that make up the corresponding maximal extended nested collection are
\[N_w\coloneqq\{C_{w,k}\in\mc B_k\mid w(s_{r-k+1})\in C_{w,k}\text{ and $C_{w,k}$ maximal with respect to inclusion, for }k=1,\ldots,r\},\]
and the maximal extended nested collection itself is $F_w\coloneqq N_w\cup\{x_j\mid j\in[n]\setminus S\}$.

\begin{theorem}\label{thm:stellohedronshelling}
Let $\mc N^{\sq}(\mc B_{K_n})$ be the extended nested set complex for $\mc B_{K_n}$, whose facets $F_\pi$ are labeled by the partial permutations $\pi \in \mf P_n$.
If a total ordering $\pi_1 \leq \cdots \leq \pi_{m}$ is a linear extension of the partial weak Bruhat order on $\mf P_n$, then $F_{\pi_1},\ldots,F_{\pi_{m}}$ is a shelling order for $\mc N^{\sq}(\mc B_{K_n})$.
\end{theorem}

In the proof, we will use the following equivalent condition for an ordering of facets to give a shelling.

\begin{lemma}[\cite{bjorner1984some}, Proposition 1.2]\label{equivshelling}
An ordering $F_1,F_2,\ldots,F_r$ of the facets of an $n$-dimensional simplicial complex $\Delta$ is a shelling if and only if for all $i \leq j$, there exists $\ell \leq j$ such that $F_i\cap F_j \subseteq F_\ell \cap F_j$ and $\abs{F_\ell \cap F_j} = n-1$. 
\end{lemma}
Our next three lemmas are used in the proof of Theorem~\ref{thm:stellohedronshelling}.
\begin{lemma}\label{lem:cover_relations_toggles_partial_order}
Let $\mc B=\mc B_{K_n}$. For two partial permutations $\pi,\sigma\in\mf P_n$, let $F_\pi$ and $F_\sigma$ be the corresponding maximal extended nested collections associated to $\pi$ and $\sigma$. If $\pi\lessdot\sigma$ in the partial weak Bruhat order, then $F_\pi$ and $F_\sigma$ differ by exactly one element.
\end{lemma}

\begin{proof}
Suppose $\pi,\sigma\in\mf P_n$ such that $\pi\lessdot\sigma$ in the partial weak Bruhat order. Then either
\begin{enumerate}
    \item $\pi=(a_1,a_2,\ldots,a_r)$ and $\sigma=(a_1,\ldots,a_{i-1},a_{i+1},a_i,\ldots,a_r)$ with $a_i<a_{i+1}$ for some $1\leq i\leq r-1$, or
    \item $\pi=(a_1,a_2,\ldots,a_{r-1},a_r)$ and $\sigma=(a_1,a_2,\ldots,a_{r-1})$.
\end{enumerate}
If we are in the first case, then there are only two pairs of elements $F_\pi$ and $F_\sigma$ that could possibly differ: the pair $C_{\pi,r-i}$ and $C_{\sigma,r-i}$, and the pair $C_{\pi,r-i+1}$ and $C_{\sigma,r-i+1}$. For all other corresponding pairs, the considered restricted building sets as well as the element needed to be included in the subset added the extended nested collection are the exact same. 

When $k=r-i$, the building set considered for finding $C_{\pi,k}$ and $C_{\sigma,k}$ is
\[\mc B|_{\{a_1,\ldots,a_{i-1},a_i,a_{i+1}\}}=\mc B|_{\{a_1,\ldots,a_{i-1},a_{i+1},a_i\}}.\]
Notice that $C_{\pi,k}$ must contain $a_{i+1}$ and $C_{\sigma,k}$ must contain $a_i$. Since the original building set is $\mc B_{K_n}$, every subset of $[n]$ is an element of the building set. Thus, the maximal element of the restricted building set containing $a_{i+1}$ or $a_i$ is in fact the same element, so
\[C_{\pi,k}=C_{\sigma,k}=\{a_1,\ldots,a_i,a_{i+1}\}.\]

For $k=r-i+1$, the building set considered for $F_\pi$ is $\mc B|_{\{a_1,\ldots,a_{i-1},a_i\}}$, and the element $C_{\pi,k}$ must contain $a_i$, so $C_{\pi,k}=\{a_1,\ldots,a_{i-1},a_i\}$. For $F_\sigma$, the considered building set is $\mc B|_{\{a_1,\ldots,a_{i-1},a_{i+1}\}}$, we have that $C_{\sigma,k}=\{a_1,\ldots,a_{i-1},a_{i+1}\}$. Notice that $C_{\pi,k}\neq C_{\sigma,k}$. Thus, the only difference between $F_\pi$ and $F_\sigma$ is that $C_{\pi,r-i+1}\in F_\pi$, whereas $C_{\sigma,r-i+1}\in F_\sigma$; all other elements of the maximal extended nested collections are the same.

If we are in the latter case, then $F_\pi=N_\pi\cup\{i\mid i\in[n]\setminus\{a_1,\ldots,a_r\}.$ Notice that
\begin{align*}
    F_\sigma&=N_\sigma\cup\{i\mid i\in[n]\setminus\{a_1,\ldots,a_{r-1}\}\}\\
    &=(N_\pi\setminus\{C_{\pi,r}\})\cup\{i\mid i\in[n]\setminus\{a_1,\ldots,a_{r-1}\}\},
\end{align*}
since for each $k=1,\ldots,r-1$, we have that $C_{\pi,k}=C_{\sigma,k}$. Thus, the only difference between $F_\pi$ and $F_\sigma$ is that $C_{\pi,r}\in F_\pi$, whereas $x_{a_r}\in F_\sigma$. 
\end{proof}

In addition, we make use of the following result, which was given in the more general context of any graph associahedron in \cite{barnard2018lattices}.

\begin{lemma}[\cite{barnard2018lattices}, Corollary 2.19]\label{lem:tubings_interval}
Let $G = K_{1,n}$. For any (non-extended) nested collection $N$ of $\mc B_G$, the set of maximal nested collections that contain $N$ is an interval in $L_G$.
\end{lemma}

Recall Corollary~\ref{cor:complete_star_isomorph}, which states that $\mc N^{\sq}(\mc B_{K_n})\simeq\mc N(\mc B_{K_{1,n}})$. Thus, any non-extended nested collection $N$ of the star graph building set corresponds exactly to an extended nested  collection $N^{\sq}$ of the complete graph building set, and the maximal nested collections and maximal extended nested collections that contain $N$ and $N^{\sq}$ respectively are in bijection. In addition, by Remark~\ref{rem:poset_isomorphism_BM}, the partial weak Bruhat order on $\mf P_n$ is dual to the poset $L_G$ when $G=K_{1,n}$, so an interval $[u,v]$ in $L_G$ is an interval $[v,u]$ in the partial weak Bruhat order on $\mf P_n$. Thus, we have the following result.

\begin{lemma}\label{lem:intersect_facets}
Consider $\mc B=\mc B_{K_n}$, and suppose $\pi_i,\pi_j\in\mf P_n$ are partial permutations. Then the subposet of elements $\sigma\in\mf P_n$ such that $F_{\pi_i}\cap F_{\pi_j}\subseteq F_\sigma$ forms an interval of the partial weak Bruhat order. In particular, there exists a unique minimal partial permutation $\rho$ such that $F_\pi\cap F_\sigma\subseteq F_\rho$.
\end{lemma}

We are now able to prove our main shelling result.

\begin{proof}[Proof of Theorem~\ref{thm:stellohedronshelling}]
By Lemma~\ref{equivshelling}, it suffices to show that for all partial permutations $\pi_i\nleq\pi_j$, there exists a partial permutation $\pi_\ell\leq\pi_i$ such that $F_{\pi_i}\cap F_{\pi_j}\subseteq F_{\pi_\ell}$ and $|F_{\pi_\ell}\cap F_{\pi_i}| = n-1$. Then fix any two partial permutations $\pi_i \nleq \pi_j$. By Lemma~\ref{lem:intersect_facets}, there exists a minimal partial permutation $\rho$ such that $F_{\pi_i}\cap F_{\pi_j}\subseteq F_\rho$. Since $\pi_i\nleq\pi_j$, we can conclude that $\pi_i\neq\rho$. By Lemma~\ref{lem:tubings_interval}, the set of partial permutations whose corresponding facets contain $F_{\pi_i}\cap F_{\pi_j}$ is an interval. In addition, there exists some partial permutation $\pi_\ell\lessdot\pi_i$ such that $F_{\pi_i}\cap F_{\pi_j}\subseteq F_{\pi_\ell}$. Since $\pi_{\ell}\lessdot\pi_i$ is a cover relation, by Lemma~\ref{lem:cover_relations_toggles_partial_order} we have $\abs{F_{\pi_\ell} \cap F_{\pi_i}} = n-1$, as desired.
\end{proof}

\subsection{Partial Orders on Maximal (Extended) Nested Collections}\label{subsec:po_nested_collections}

Motivated by poset-theoretic results by Hersh related to generic cost vectors, we show in this subsection that all nestohedra and extended nestohedra have the property that there exists a ``generic'' cost vector that gives an acyclic directed graph isomorphic to the Hasse diagram of a poset. This will allow us to use Hersh's results in the case that this poset is a lattice to give nice properties of this lattice. We will first provide some definitions related to such vectors. 

For a simple polytope $P\subseteq\mathbb{R}^d$, a \textbf{generic cost vector} $\textbf{c}\in\mathbb{R}^d$ is a vector such that $\textbf{c}\cdot u\neq\textbf{c}\cdot v$ for distinct vertices $u,v$ of $P$. Given such a vector \textbf{c}, we obtain the acyclic directed graph, denoted $G(P,\textbf{c})$, on the $1$-skeleton of $P$ by orienting each edge $e_{u,v}$ from $u$ to $v$ whenever $\textbf{c}\cdot u<\textbf{c}\cdot v$.

We now extend the poset $L_G$ defined by Barnard and McConville for graph associahedra to all nestohedra and extended nestohedra. In the process, we will show that there exists a generic cost vector \textbf{c} such that $G(P,\textbf{c})$ is the Hasse diagram of this poset. When extending the poset to all nestohedra, we omit many details; they follow very naturally from \cite{barnard2018lattices}.

First, we define the following function on maximal nested collections.

\begin{lemma}[cf. \cite{barnard2018lattices}, Lemma 2.4]\label{lem:top_fctn}
Let $\mc B$ be a building set on $[n]$, and let $N$ be a maximal nested collection. For each $I\in N\cup\{[n]\}$, there exists a unique element $\top_N(I)\in[n]$ not contained in any element of $\{J\in N\mid J\subsetneq I\}$. In particular, this function is a bijection between elements of $N\cup\{[n]\}$ and $[n]$.
\end{lemma}

Since maximal nested collections correspond to the vertices of the corresponding nestohedron, one can describe the coordinates of the vertices of the nestohedron in terms of the maximal nested collection; see \cite[Proposition 7.9]{postnikov2009permutohedra}. Recall that $\mc B$-trees are in bijection with maximal nested collections.

\begin{lemma}\label{lem:vertices_nestohedron}
Let $\mc B$ be a building set on $[n]$. If $N$ is a maximal nested collection and $T$ is the corresponding $\mc B$-tree, then the point $\textbf{v}_N=(v_1,\ldots,v_n)$ is a vertex of the nestohedron $\mc P(\mc B)$ where
\[v_i\coloneqq|\{I\in\mc B\mid i\in I\subseteq T_{\leq i}\}|.\]
\end{lemma}

We now define a partial order on maximal nested collections. Let $\mc B$ be a building set on $[n]$, with maximal nested collection $N$. Suppose that $I$ is a non-maximal element of $N$. There exists a unique building set element $J\neq I$ such that $M=N\setminus\{I\}\cup\{J\}$ is also a maximal nested collection of $\mc B$. Define a \textbf{flip} as the relation $N\to M$ if $\top_N(I)<\top_M(J)$. We say that $N\leq M$ if there exists a sequence of flips of maximal nested collections of the form $N\to\cdots\to M$.

\begin{theorem}[cf. \cite{barnard2018lattices}, Lemma 2.8]\label{thm:BMposet}
The set of maximal nested collections is partially ordered by the relation $\leq$ defined above.
\end{theorem}

\begin{proof}
Let $\textbf{c}=(n,n-1,n-2,\ldots,1)$. If $N$ and $M$ are maximal nested collections such that $M=N\setminus\{I\}\cup\{J\}$ for building set elements $I\neq J$, then $\textbf{c}\cdot(\textbf{v}_M-\textbf{v}_N)>0$, where $\textbf{v}_M,\textbf{v}_N$ are the vertices of $\mc P(\mc B)$ corresponding to $M$ and $N$ respectively. Thus, $N\to M$ on maximal nested collections is induced by the vector \textbf{c}, so the relation is acyclic. Thus, the transitive closure of such relations is a partial order.
\end{proof}

\begin{corollary}\label{cor:nestohedron_cost_vector}
For any nestohedron $\mc P(\mc B)$, there exists a generic cost vector \textbf{c} such that $G(\mc P(\mc B),\textbf{c})$ is the Hasse diagram of a poset.
\end{corollary}

We now prove an analogous statement for extended nestohedra. For a building set $\mc B$ on $[n]$ and maximal extended nested collection $N$, let
\[\Supp(N)\coloneqq\{i\in[n]\mid x_i\notin N\}.\]

\begin{lemma}
Let $\mc B$ be a building set on $[n]$, and let $N$ be a maximal extended nested collection. For each $I\in N$ such that $I\subseteq\Supp(N)$, there exists a unique element $\top_N(I)\in\Supp(N)$ not contained in any element of $\{J\in N\mid J\subsetneq I\}$. In particular, this function is a bijection between non-design vertex elements of $N$ and $\Supp(N)$.
\end{lemma}

Recall Proposition~\ref{prop:ext_nestohedra_vertices}, which gives the coordinates of the vertex of extended nestohedron $\mc P^{\sq}(\mc B)$ corresponding to a maximal extended nested collection.

We now define a partial order on maximal extended nested collections. Let $\mc B$ be a building set on $[n]$, with maximal nested collection $N$. Suppose that $I\in\mc B$ is a non-maximal element of $N$. There exists a unique building set element $J\neq I$ such that $M=N\setminus\{I\}\cup\{J\}$ is also a maximal extended nested collection of $\mc B$, with $\Supp(N)=\Supp(M)$. Like in the non-extended case, a flip is the relation $N\to M$ if $\top_N(I)<\top_M(J)$ and $\Supp(N)=\Supp(M)$ of if $\Supp(N)=\Supp(M)\cup\{i\}$ for some $i\notin\Supp(M)$. We say that $N\leq M$ if there exists a sequence of flips of maximal extended nested collections of the form $N\to\cdots\to M$.

\begin{theorem}\label{thm:BM_extended_poset}
The set of maximal extended nested collection is partially ordered by the relation $\leq$ defined above.
\end{theorem}

\begin{proof}
The edges of the extended nestohedron $\mc P^{\sq}(\mc B)$ are of one of two forms, depending on the maximal extended nested collections $N$ and $M$ corresponding to the vertices that the edge connects: either $\Supp(N)=\Supp(M)$, or $\Supp(N)=\Supp(M)\cup\{i\}$ for some $i\notin\Supp(M)$. If we are in the first case, then $M=N\setminus\{I\}\cup\{J\}$ for some $I,J\in\mc B$ with $I\neq J$. Let $i=\top_N(I)$ and $j=\top_M(J)$. Then using Proposition~\ref{prop:ext_nestohedra_vertices}, we see that $\textbf{v}_N$ and $\textbf{v}_M$, the vertices of $\mc P^{\sq}(\mc B)$ corresponding to $N$ and $M$ respectively, agree on every coordinate except the $i$-th and $j$-th coordinates. In fact, $\textbf{v}_M-\textbf{v}_N=k(-e_i+e_j)$, where $k$ is equal to the number of building set elements contained in $I\cup J$ and contain both $i$ and $j$.

If we are in the second case, then notice that $\textbf{v}_N$ and $\textbf{v}_M$ differ in only the $i$-th coordinate, with $\textbf{v}_M-\textbf{v}_N=ke_i$, where $k<0$. 

Let $\textbf{c}=(-n,-n+1,\cdots,-1)$. If $N$ and $M$ are maximal extended nested collections such that there is an edge connecting their corresponding vertices in $\mc P^{\sq}(\mc B)$, then $\textbf{c}\cdot (\textbf{v}_M-\textbf{v}_N)>0$. Thus, $N\to M$ on maximal extended nested collections is induced by the vector \textbf{c}, so the relation is acyclic. Thus, the transitive closure of such relations is a partial order.
\end{proof}

\begin{corollary}\label{cor:ext_nestohedron_general_cost}
For any extended nestohedron $\mc P^{\sq}(\mc B)$, there exists a generic cost vector $\textbf{c}$ such that $G(\mc P^{\sq}(\mc B),\textbf{c})$ is the Hasse diagram of a poset.
\end{corollary}

Let $L(\mc B)$ and $L^{\sq}(\mc B)$ denote the partial orders defined on maximal nested collections and maximal extended nested collections, respectively. This last result follows from \cite[Theorem 1.4]{hersh2018posets}.

\begin{proposition}\label{prop:lattice_intervals}
If $L(\mc B)$ or $L^{\sq}(\mc B)$ is a lattice, then each open interval $(u,v)$ in the lattice has order complex which is homotopy equivalent to a ball or a sphere of some dimension. Therefore, the M\"obius function $\mu(u,v)$ only takes values $0,1$, and $-1$.
\end{proposition}

Recall that if $\mc B$ is the building set for the star graph, then by \cite[Theorem 4.10]{barnard2018lattices}, the poset $L(\mc B)$ is a lattice. One can then apply this proposition to the partial weak Bruhat order that we defined, which is dual to $L(\mc B)$, to obtain nice properties on this partial order.

\begin{corollary}
Each open interval $(u,v)$ in the partial weak Bruhat order $\mf P_n$ has order complex which is homotopy equivalent to a ball or a sphere of some dimension. Therefore, the M\"obius function $\mu(u,v)$ only takes values $0,1,$ and $-1$.
\end{corollary}

\section*{Acknowledgements}
This research was carried out as part of the 2019 REU program at the School of Mathematics at University of Minnesota, Twin Cities. The authors are grateful for the support of NSF RTG grant DMS-1148634 for making the REU possible. The authors would like to thank Vic Reiner and Sarah Brauner for their mentorship and guidance. The authors also thank Yibo Gao for pointing out a significant error in a previous version of this paper.

\printbibliography
\end{document}